\documentclass[11pt,letterpaper,twoside,reqno,nosumlimits]{amsart}

\usepackage{mathabx} 
\usepackage{listings} 
\usepackage{multirow}
\usepackage{booktabs}
\usepackage{graphicx}

\AtBeginDocument{\RenewCommandCopy\qty\SI}

\usepackage{etoolbox}
\usepackage{comment}

\patchcmd{\section}{\scshape}{\bfseries}{}{}
\usepackage{amsfonts}
\usepackage{amsmath}
\usepackage{amsthm}
\usepackage{mathtools}
\usepackage{thm-restate}
\usepackage{verbatim}
\usepackage{enumitem}
\usepackage[margin=1in]{geometry}
\usepackage{bbm}
\usepackage{subcaption}
\usepackage{amssymb}
\usepackage{float}
\usepackage{scrextend}
\usepackage[dvipsnames]{xcolor}
\patchcmd{\section}{\normalfont}{\normalfont\color{MidnightBlue}}{}{}
\patchcmd{\subsection}{\normalfont}{\normalfont\color{MidnightBlue}}{}{}
\patchcmd{\subsubsection}{\normalfont}{\normalfont\color{MidnightBlue}}{}{}



\usepackage{fancyhdr}
\usepackage{amsfonts,amsbsy,amsgen,amscd,mathrsfs,amssymb,amsthm,mathtools,tensor}

\usepackage{tikz}
\usetikzlibrary{arrows}
\usepackage{overpic}

\usepackage[utf8]{inputenc}
\usepackage{setspace}
\usepackage{thm-restate}

 \usepackage{soul}
\usepackage{subeqnarray}
\usepackage{amsfonts} 
\usepackage{algorithm,algorithmic}
\usepackage{graphicx}
\usepackage{tabu}       
\usepackage{siunitx}    
\usepackage{etoolbox}   
\usepackage{booktabs}   
\usepackage{xspace}     
\usepackage{caption}

\usepackage[colorlinks=true,bookmarks=false,citecolor=blue,urlcolor=blue]{hyperref} 
\usepackage{cleveref}

\newtheorem{theorem}{Theorem}[section] \crefname{theorem}{theorem}{theorems}
\newtheorem{remark}{Remark}[section] \crefname{remark}{remark}{remarks}
\newtheorem{lemma}{Lemma}[section] \crefname{lemma}{lemma}{lemmas}
\newtheorem{assumption}{Assumption}[section] \crefname{assumption}{assumption}{assumptions}
\newtheorem{proposition}{Proposition}[section] \crefname{proposition}{proposition}{propositions}
\newtheorem{definition}{Definition}[section] \crefname{definition}{definition}{definitions}
 \crefname{corollary}{corollary}{corollaries}
 \crefname{example}{example}{examples}
 \crefname{problem}{problem}{problems}
 \crefname{property}{property}{properties}

\usepackage[all]{xy}
\usepackage{graphicx}
\usepackage{amssymb}
\usepackage{environ}
\usepackage{float}

\usepackage[foot]{amsaddr}

\newcommand{\R}{\mathbb{R}}

\newcommand{\N}{\mathbb{N}}

\newcommand{\dif}{\mathrm{d}}

\newcommand{\mH}{\mathcal H}

\newcommand{\mX}{\mathcal X}

\newcommand{\wh}[1]{\widehat{#1}}

\DeclareMathOperator*{\argmin}{arg\,min}

\begin{document}

\definecolor{codegreen}{rgb}{0,0.6,0}
\definecolor{codegray}{rgb}{0.5,0.5,0.5}
\definecolor{codepurple}{rgb}{0.58,0,0.82}
\definecolor{backcolour}{rgb}{0.92,0.92,0.92}

\lstdefinestyle{mystyle}{
    backgroundcolor=\color{backcolour},   
    commentstyle=\color{codegreen},
    keywordstyle=\color{magenta},
    numberstyle=\tiny\color{codegray},
    stringstyle=\color{codepurple},
    basicstyle=\ttfamily\footnotesize,
    breakatwhitespace=false,         
    breaklines=true,                 
    captionpos=b,                    
    keepspaces=true,                 
    numbers=left,                    
    numbersep=5pt,                  
    showspaces=false,                
    showstringspaces=false,
    showtabs=false,                  
    tabsize=2
}

\lstset{style=mystyle}

\title[Data-efficient Kernel Methods for Learning Hamiltonian Systems]{Data-efficient Kernel Methods for Learning \\Hamiltonian Systems}

 \author[Y. Jalalian]{Yasamin Jalalian$^1$}
 \address{$^1$ Department of Computing and Mathematical Sciences, Caltech, CA, USA.}
 \email{ yjalalia@caltech.edu}

 \author[M. Samir]{Mostafa Samir$^2$}
 \address{$^2$ Beyond Limits, CA, USA.}
 \email{mibrahim@beyond.ai}

\author[B. Hamzi]{Boumediene Hamzi$^{1,}$$^{3}$}
 \address{$^3$ The Alan Turing Institute, London, UK. Parts of this work were carried out while the author was a resident scholar at the Isaac Newton Institute for Mathematical Sciences, Cambridge, UK.}

 \author[P. Tavallali]{Peyman Tavallali$^4$}
 \address{$^4$ Jet Propulsion Laboratory, Caltech, CA, USA. }

 \author[H. Owhadi]{Houman Owhadi$^1$}


\begin{abstract}
Hamiltonian dynamics describe a wide range of physical systems. As such, data-driven simulations of Hamiltonian systems are important for many scientific and engineering problems. In this work, we propose kernel-based methods for identifying and forecasting Hamiltonian systems directly from trajectory data. We present two approaches: a 2-step method that reconstructs trajectories before learning the Hamiltonian, and a 1-step method that jointly infers both. Across several benchmark systems, including mass-spring dynamics, a nonlinear pendulum, and the Hénon-Heiles system, we demonstrate that our framework achieves accurate, data-efficient predictions and outperforms 2-step kernel-based baselines, particularly in scarce-data regimes, while preserving the Hamiltonian structure. Moreover, we prove \textit{a priori} error estimates, ensuring reliability of the learned models. We also provide a more general, problem-agnostic numerical framework that goes beyond Hamiltonian systems and can be used for data-driven learning of arbitrary dynamical systems.
\end{abstract}

\maketitle

\section{Introduction}\label{problem_statement}

Dynamical systems play a fundamental role in describing natural phenomena. Two essential aspects of their study are (i) \textit{system identification}, which involves determining parametric or nonparametric models of a system's evolution, and (ii) \textit{prediction or forecasting}, which focuses on projecting the system's state forward in time. With the advent of machine learning and its growing role in inferring physical models from data, often referred to as physics-informed machine learning \cite{karniadakis2021physics}, data-driven methods have become increasingly important for addressing both tasks.

Data-driven learning of dynamical systems has developed along several complementary directions. Classical time-series approaches use delay-coordinate embeddings to reconstruct state-space dynamics and enable nonlinear forecasting from scalar or partial observations \cite{takens1981detecting, kantz97, CASDAGLI1989, abarbanel2012analysis}. A second line of work seeks to identify explicit governing equations from data, including sparse-regression and library-based methods such as SINDy and weak SINDy \cite{brunton2016sindy, Messenger_2021}. Another direction involves operator-theoretic methods based on dynamic mode decomposition and Koopman spectral analysis \cite{schmid2010dynamic, tu2014dynamic, williams2015data}. Finally, continuous-time machine-learning models such as Neural ODEs \cite{chen2018neuralode}, universal differential equations \cite{rackauckas2020universal}, and physics-informed neural networks (PINNs) \cite{raissi2019physics} provide flexible frameworks for learning dynamics, incorporating prior information, and solving forward or inverse problems.

When the full state is not observed, the learning problem becomes more delicate because both the state representation and the governing dynamics must be inferred. Delay embeddings, autoencoders, and recognition models have therefore been combined with equation-learning or Neural-ODE formulations to learn dynamics from partial observations \cite{bakarji, mona}. Related work on discovering latent coordinates together with governing equations shows that identifying an appropriate state representation can be as important as learning the vector field itself \cite{champion2019data}. These developments motivate methods that jointly reconstruct missing dynamical variables and learn governing equations, while also incorporating structural information when it is available.

\subsection{Kernel and Gaussian process methods for dynamical systems.}
Among the various data-driven learning methods for dynamical systems, Gaussian process or kernel methods \cite{CuckerandSmale, kanagawa2018gaussian} have been used in dynamical systems through kernel analog forecasting \cite{ALEXANDER2020132520}, approximation of nonlinear dynamical systems \cite{bouvrie2017kernel, bouvrie2017keyquantities},
and kernel or Gaussian-process methods for solving and learning differential equations \cite{chen2021solving, long2023kernelapproachpdediscovery}. Compared to other learning approaches, kernel-based come with uncertainty quantification (UQ) estimates \cite{hou2024propagating} and provide guaranteed convergence along with \textit{a priori} error estimates \cite{batlle2023erroranalysiskernelgpmethods,
jalalian2025keql}.

A closely related line of work uses Gaussian processes for learning ordinary differential equations from sparse and noisy observations. In particular, Gaussian-process gradient matching methods infer unknown ODE parameters by combining smooth trajectory reconstructions with derivative information \cite{calderhead2008accelerating, dondelinger2013adaptive, barber2014gpode, macdonald2015gradient}. Other Gaussian-process approaches learn unknown nonlinear vector fields directly from state observations \cite{Heinonen2018learning}, while latent-force models combine mechanistic dynamics with nonparametric Gaussian-process components \cite{alvarez2009latent, sarkka2018latent}. Bayesian formulations combine Gaussian-process priors with Neural ODE models to learn partially observed dynamics from scarce and noisy data while quantifying uncertainty \cite{mohamed}.

\subsection{Learning Hamiltonian and structure-preserving dynamics}

Hamiltonian dynamical systems are a particularly important class of dynamical systems, as they describe the evolution of a wide range of physical systems while their geometric structure preserves fundamental quantities such as energy and follows specific symplectic principles. Consequently, the problem of learning Hamiltonian dynamical systems from data has recently been studied using a variety of structure-preserving methods.

\subsubsection{Hamiltonian structure and energy conservation}

A central approach to structure-preserving learning of conservative dynamics is to parameterize or estimate a scalar Hamiltonian and derive the vector field from Hamilton's equations. In the Gaussian-process literature, Hamiltonian GP models have been developed for learning energy-conserving dynamics with uncertainty quantification from noisy trajectories \cite{ross2023learning}. On the kernel side, structure-preserving kernel ridge regression methods recover Hamiltonian functions from noisy observations of Hamiltonian vector fields, with representer theorems and error estimates for derivative-based losses \cite{hu2024structure}. In both cases, structure preservation comes from generating the learned dynamics through an estimated Hamiltonian rather than an unconstrained vector field.

Neural-network approaches follow the same principle. Hamiltonian neural networks learn a Hamiltonian function and use Hamilton's equations to define the dynamics \cite{greydanus2019hamiltonian}, with later extensions incorporating additional structure such as automatic symmetry detection \cite{Dierkes_2023}. Related invariant-preserving methods instead learn conserved quantities directly from trajectory data and have been demonstrated on Hamiltonian examples including the gravitational three-body problem \cite{Liu2022machine}.

Sparse equation-discovery methods provide another route to Hamiltonian-structure preservation. In particular, weak-form SINDy has been extended to coarse-grain Hamiltonian systems by identifying reduced Hamiltonian dynamics from trajectory data while avoiding direct numerical differentiation \cite{messenger2023coarsegraininghamiltoniansystemsusing}.

\subsubsection{Symplectic-structure-preserving learning architectures}

A second line of structure-preserving work enforces symplectic geometry at the level of the learned map, architecture, or hypothesis space. 
\cite{valperga2022reversible} introduces neural architectures for learning reversible dynamics and adapt them to symplectic systems, thereby incorporating time-reversal symmetry together with symplectic structure. 
For high-dimensional Hamiltonian systems, graph-based symplectic architectures combine symplectic maps with permutation equivariance in order to learn many-body Hamiltonian dynamics \cite{varghese2024sympgnnssymplecticgraphneural}. 
Other neural architectures preserve Hamiltonian or symplectic structure through their time-stepping or network design \cite{zhong2020symoden, chen2020symplectic, chen2021neuralsymplectic, jin2020sympnetsintrinsicstructurepreservingsymplectic}. 
Recent variants further develop symplectic neural architectures, including symplectic neural Gaussian processes for meta-learning Hamiltonian dynamics from limited data \cite{iwata2024symplectic} and SympNet constructions based on geometric integrators for Hamiltonian differential equations \cite{tapley2024symplectic}. 

Symplectic structure has also been incorporated into Gaussian process and RKHS models. In particular, Symplectic-spectrum Gaussian processes encode this geometry through the GP spectral representation to learn Hamiltonians from sparse, noisy data \cite{tanaka2022symplectic}, while symplectic-kernel RKHS methods and random-feature approximations learn Hamiltonian vector fields with optional additional symmetries such as oddness \cite{smith2024learning}. Related approaches extend structure-preserving learning beyond canonical Hamiltonian systems by considering generalized symplectic or Poisson formulations, in which the structure matrix may be state-dependent, degenerate, or itself learned from data \cite{chen2021neuralsymplecticform, sipka2023direct}.

\subsubsection{Symplectic integrators and variational learning approaches}

A complementary class of methods preserves structure through numerical time discretization or variational formulations. Classical symplectic and variational integrators achieve stable long-time simulation by preserving properties such as symplecticity and momentum maps \cite{hairer2006geometric, marsden2001discrete}. Recent learning-based extensions include Gaussian-process models for continuous and discrete Lagrangians with convergence guarantees and uncertainty quantification \cite{offen2025machinelearningcontinuousdiscrete}, as well as hybrid approaches that combine variational integrators for nominal mechanical dynamics with Gaussian processes for residual effects, including constrained systems \cite{brudigam2022structure, ensinger2022structure}. Lagrangian neural networks similarly learn a Lagrangian directly, avoiding the need for canonical momenta while retaining physically structured dynamics \cite{cranmer2020lagrangian}.

\subsubsection{Port-Hamiltonian and dissipative formulations}

Structure-preserving learning has also been extended beyond closed conservative systems to controlled, forced, and dissipative dynamics. Dissipative SymODEN uses a port-Hamiltonian neural ODE to learn energy, dissipation, and control-input components of mechanical systems \cite{zhong2020dissipative}, while port-Hamiltonian neural networks model non-autonomous systems by jointly learning a Hamiltonian, damping, and time-dependent forcing \cite{desai2021port}. More generally, bracket-based graph neural dynamics encode reversible and irreversible evolution through Hamiltonian, gradient, double-bracket, or metriplectic structures, thereby enforcing properties such as energy conservation, energy dissipation, or simultaneous energy conservation and entropy production \cite{gruber2023reversible}.

\subsection{Summary of contributions and paper outline}
Many existing Hamiltonian learning methods assume access to the full phase-space state, its time derivatives, or samples of the Hamiltonian vector field. The setting considered here is different: starting from sparse observations of a trajectory, we seek to reconstruct the missing trajectory values while simultaneously identifying the underlying Hamiltonian. Thus, both the state evolution and the governing Hamiltonian must be inferred from the same limited time-series data.
In particular, consider the Hamiltonian system 
\begin{align*}
\begin{cases}
    \dot q= \partial_p H,\\
    \dot p=-\partial_q H,
\end{cases}
\end{align*}
where $q(t) \in \R^{m}$ and $p(t) \in \R^{m}$ represent the generalized position and the generalized momentum at time $t$, respectively, and $H(q,p)$ is the Hamiltonian.
Given discretized time series data on the interval $[0, t_f]$ for trajectories $q$ and $p$ and knowing that they come from an unknown Hamiltonian system, the problem of system identification corresponds to identifying $q,p$ and $H$ over the entire domain $[0, t_f]$, whereas the problem of forecasting corresponds to predicting values of $q$ and $p$ after time $t_f$ using the identified Hamiltonian. Our assumption of trajectory-based data more closely reflects most experimental scenarios in applications than observing the full governing vector-field.

In this work, we address both of these problems by using a similar methodology to the one of Kernel Equation Learning (KEqL) \cite{jalalian2025keql} which addresses problems in PDE learning. 
In particular, we propose two kernel-based methods, one that first aims at learning trajectories $q,p$ and then learns Hamiltonian $H$ and another method that simultaneously learns $q,p$ and $H$.
We demonstrate the efficiency of our simultaneous learning method compared to the 2-step kernel-based approach, particularly in a scarce-data regime, across various physical problems. Our proposed framework preserves the Hamiltonian structure by constructing the learned vector field from a learned Hamiltonian; consequently, the exact continuous-time flow of the learned system conserves the learned Hamiltonian. Additionally, we provide theoretical error bounds which quantify the convergence rate of our approximation to the truth. Moreover, the numerical framework\footnote{The framework can be found at \href{https://github.com/Mostafa-Samir/CGC}{https://github.com/Mostafa-Samir/CGC}.}\label{fn:code}  
that we provide with this work is a problem-agnostic general implementation that goes beyond solving the problem at hand and can be used for inferring unknown variables in any arbitrary physical system from data.

The paper is outlined as follows. In section 2, we begin by introducing the problem formally and proceed to describing the two suggested methods for its solution. In section 3, we prove \textit{a priori} error rates for the 1-step method. Section 4 is dedicated to discussion on the details of the implementation for both methods. In section 5, we demonstrate numerical results on several benchmark systems and discuss how they compare with each other in different data regimes. Finally, section 6 is dedicated to concluding remarks and future directions. 

\section{Problem Statement \& Methodology}

\subsection{Problem Statement}

We start by formally introducing the problem and setting up notation. Let $H: \mathbb{R}^{2m} \to \mathbb{R}$ be an unknown Hamiltonian  defining a dynamical system via the  equations
\begin{align*}
\begin{cases}
    \dot q = \partial_p H(q,p), \\
    \dot p = -\partial_q H(q,p),
\end{cases}
\end{align*}
where $q: \mathbb{R} \to \mathbb{R}^m, \, p: \mathbb{R} \to \mathbb{R}^m,$ denote the generalized position and momentum respectively. At time $t \in \mathbb{R}_+$, we define the state of the full system as
\begin{align*}
    y(t) := \big(q(t), p(t)\big)^\top \in \mathbb{R}^{2m},
\end{align*}
so that its evolution can be compactly written as
\begin{align*}
    \frac{\dif y}{\dif t} = 
    \begin{pmatrix}
        \partial_p H(q(t),p(t)) \\
        -\partial_q H(q(t),p(t))
    \end{pmatrix}.
\end{align*}
Let $N \in \N, t_f \in \R_+$. For a single trajectory \(y:[0,t_f]\to\mathbb{R}^d\), let
\(
T_{\mathrm{col}} \coloneqq \{t_i\}_{i=1}^N
\)
denote a discretization of the time interval \([0,t_f]\). The corresponding
trajectory values
\[
y(T_{\mathrm{col}} ) = \{y(t_i)\}_{i=1}^N
\]
form the full set of \textit{collocation points} associated with this trajectory.

To model partial observability, we assume that only a subset of these
trajectory values is available as data. We introduce a \emph{sparsity factor}
\(\alpha \in (0,1]\), representing the fraction of collocation points that are
unobserved, and define
\[
N_{\mathrm{obs}} \coloneqq \lfloor (1-\alpha)N \rfloor.
\]
We then select \(N_{\mathrm{obs}}\) time points uniformly at random from
\(T_{\mathrm{col}}\), and denote the resulting set of observation times by
\(
S_{\mathrm{obs}}
\coloneqq
\{s_j\}_{j=1}^{N_{\mathrm{obs}}}
\subset T_{\mathrm{col}}.
\)
Thus, for this trajectory, the available data consist of the observed
time-state pairs
\[
y(S_{\mathrm{obs}}) =
\{y(s_j)\}_{j=1}^{N_{\mathrm{obs}}},
\]
which we refer to as the \emph{observation points}. The remaining collocation
points correspond to unobserved locations along the same trajectory.

For an example system, in Figure \ref{fig:phasespace:m2s3}, we plot the phase-space trajectories of the Two-Mass-Three-Spring system for each mass component, showing both the full trajectories on the collocation points (in black) and the observed points (in red). 
\begin{figure}[htbp]
    \centering
    \begin{subfigure}{0.47\textwidth}
    \includegraphics[width=\linewidth]{ 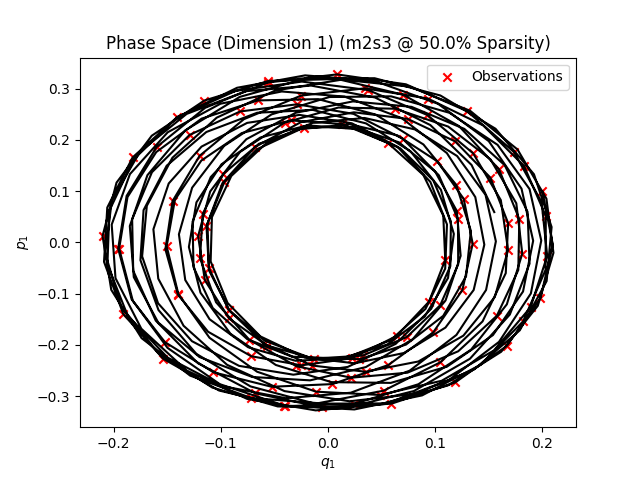}
    \end{subfigure}
    \begin{subfigure}{0.47\textwidth}
    \includegraphics[width=\linewidth]{ 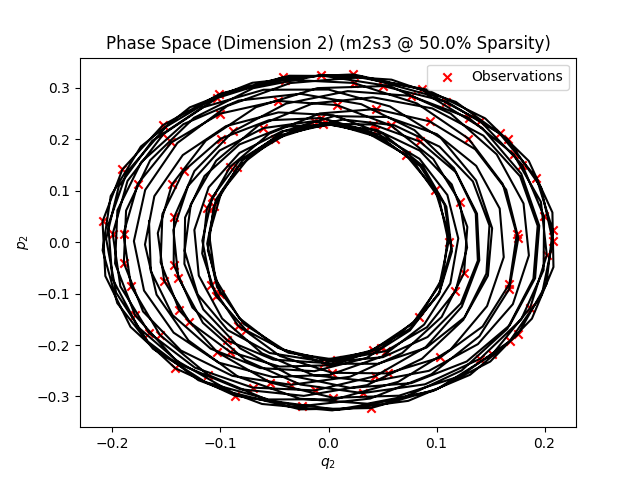}
    \end{subfigure}
    \caption{Phase space trajectories of the two-mass-three-spring system with sparsity factor $0.5$.}
    \label{fig:phasespace:m2s3}
\end{figure}

Our goal is to recover $q,p,H$ from these observation and collocation points and to forecast $q,p$ in time.

\begin{remark}
The extent to which a governing equation can be recovered from a single
trajectory is an identifiability question and depends on both the admissible
class of equations and the richness of the observed trajectory. Necessary and
sufficient identifiability conditions for several classes of ODEs and PDEs are
studied in~\cite{scholl2023uniqueness,Scholl2026}, while
\cite{he2024much} shows, in the PDE setting, how the information contained in a
single solution depends on the differential operator, the initial data, and the
resulting data space. Related results for linear and linear-in-parameters
dynamical systems show that a single trajectory can identify the model
parameters when an appropriate trajectory-richness condition is satisfied
\cite{StanhopeRubinSwigon2014}.

In the present Hamiltonian setting, partial knowledge of the trajectory \(y(t)\), together with the knowledge that the data are generated by a Hamiltonian system and the structural prior encoded through the choice of learning model, determines the Hamiltonian vector field, and hence
\[
\nabla H(y(t)),
\]
along the observed trajectory. Without additional structural assumptions,
however, these data do not generally determine \(\nabla H\) throughout an
open region of phase space: distinct Hamiltonians may have identical gradients
on the observed trajectory while differing away from it. Consequently, the
Hamiltonian obtained from a single partially observed trajectory should, in
general, be understood as an extension selected by the imposed hypothesis
class and the minimum-RKHS-norm criterion, rather than as an unconditionally
identifiable recovery of the true Hamiltonian over the entire phase-space
domain. Stronger assumptions can lead to global identifiability; for example,
\cite{rudnev2006inverse} proves such a result for real-analytic natural
Hamiltonian systems under specific geometric, topological, and energy
conditions.

Since our methodology extends straightforwardly from the single-trajectory to the multiple-trajectory setting, for ease of notation, we present the methodology, theoretical analysis, and implementation details using data consisting of partial observations of a single trajectory. In Section~\ref{sec:numerics}, we present a numerical example involving multiple trajectories, briefly describe the corresponding modifications to the implementation, and discuss how the theoretical results can be extended to this setting.
\end{remark}

\subsection{The proposed method}\label{sec:method}

In the following, we outline our proposed methodology for learning and forecasting the Hamiltonian system based on the theory of reproducing kernel Hilbert spaces, closely following the methodology introduced in \cite{jalalian2025keql} for equation learning. To begin, we represent the kernels we are going to use throughout the paper for approximating $q,p,H$.
Let $\Gamma: \mathbb{R} \times \mathbb{R} \to \mathbb{R}$ and $\Sigma: \mathbb{R} \times \mathbb{R} \to \mathbb{R}$ be positive-definite kernels with respective RKHS(s) $\mH_{\Gamma}$ and $\mH_{\Sigma},$ and let $\Psi: \mathbb{R}^{2m} \times \mathbb{R}^{2m} \to \mathbb{R}$ be a positive-definite kernel with RKHS $\mH_{\Psi}$.
For precise definitions of the kernels used in numerical experiments, see subsection \ref{sec:kernels}.

\subsubsection{2-step method: first learn $q,p$, then learn $H$}\label{twostep_method}

The first method we study is the classic kernel interpolation method. As a first step, we interpolate the functions $q$ and $p$. 
Formally, we denote the observation vector and the collocation vector as
\begin{alignat*}{2}
S &\coloneqq \big(s_1, \dots, s_{N_{\mathrm{obs}}}\big)^\top \in \R^{N_{\mathrm{obs}}},
  &\quad& \text{for } s_j \in S_{\mathrm{obs}}, \\
T &\coloneqq \big(t_1, \dots, t_N\big)^\top \in \R^{N},
  &\quad& \text{for } t_i \in T_{\mathrm{col}},
\end{alignat*}
and we approximate each function $q$ and $p$ by finding the interpolant with the minimum RKHS norm. That is, we solve
\begin{align}\label{q-kernel-interpolant}
    q^\star \coloneqq \argmin_{\wh q \, \in\, \mH_{\Gamma}} \quad \big\| \wh q \big\|_{\Gamma} \quad \text{subject to} \quad \wh q(S) = q(S),
\end{align}
and
\begin{align}\label{p-kernel-interpolant}
    p^\star \coloneqq \argmin_{\wh p \,\in\, \mH_{\Sigma}} \quad \big\| \wh p \big\|_{\Sigma} \quad \text{subject to} \quad \wh p(S) = p(S).
\end{align}
As a second step, using the interpolants of the first step, we find the best approximation for the function $H$. In particular, we explicitly compute the derivatives of the approximants $q^\star, p^\star$ from the first step, and we solve 
\begin{align}\label{H-kernel-interpolant}
    H^\star \coloneqq \argmin_{\wh H\,\in \,\mH_{\Psi}} \quad \|\wh H\|^2_{\Psi} + 
    \frac{1}{\lambda_1} \|\partial_p \wh H ((q,p) (T)) -\dot q^\star(T)\|_2^2 + 
    \frac{1}{\lambda_2} \|\partial_q \wh H((q,p)(T))+\dot p^\star(T)\|_2^2,
\end{align}
where $\lambda_1, \lambda_2 > 0$ are regularization parameters.
It is well-known that, assuming there exists true functions $p^\dagger, q^\dagger, H^\dagger$ satisfying the data, the optimization problems \eqref{q-kernel-interpolant}, \eqref{p-kernel-interpolant}, and \eqref{H-kernel-interpolant} are well-posed and admit unique solutions $q^\star,p^\star,H^\star$; more details are given in section \ref{sec:implementation}.

\subsubsection{1-step method: simultaneously learn $q,p$ and $H$}\label{onestep_method}
In the second method, we estimate $q,p,H$ by a single joint optimization problem, enforcing the constraint that $(q,p)$ solve the same system of ODEs governed by $H$. For ease of notation, we keep the same notation for the minimizers as in the 2-step method but make explicit which method is being used.
We solve
\begin{equation}\label{1-step-optimal-recovery}
\begin{alignedat}{3}
     \min_{\wh q \,\in\, \mH_\Gamma,\,\wh p \,\in\, \mH_\Sigma,\,\wh H \, \in \,\mH_\Psi}\, \,\, &\|\wh H\|_{\Psi}^2 + \frac{1}{\lambda_1} \|\wh q\|_{\Gamma}^2 + \frac{1}{\lambda_2} \|\wh p\|_{\Sigma}^2
    \\
    \quad \quad &\text{subject to} \quad
    \begin{cases}
        \wh q(S) = q(S),\\
        \wh p(S) = p(S),\\
        \dot {\wh q}(T) = \partial_p \wh H((\wh q,\wh p)(T)),\\
        \dot {\wh p}(T) = - \partial_q \wh H((\wh q,\wh p)(T)),
    \end{cases}
\end{alignedat}
\end{equation}
where the first two constraints are the observation constraints and the last two are the dynamical constraints that enforce the physics of the problem on collocation points. 
The observation constraints
are linear equality constraints, since point evaluation is a linear functional on the respective RKHSs. In contrast, for the dynamical constraints, although differentiation and evaluation of \(\widehat H\) at a fixed phase-space point are linear operations in \(\widehat H\), the evaluation locations \((\widehat q(t_i),\widehat p(t_i))\) are themselves unknown and depend on \(\widehat q\) and \(\widehat p\). This joint dependence makes the dynamical constraints nonlinear in the joint variables \((\widehat q,\widehat p,\widehat H)\) and, in general, causes the feasible set to be nonconvex. Hence, despite the objective of \eqref{1-step-optimal-recovery} being a strictly convex quadratic objective over the product space
\(\mathcal H_\Gamma \times \mathcal H_\Sigma \times \mathcal H_\Psi\), 
the problem is overall a nonlinear equality-constrained and generally nonconvex optimization problem. Section \ref{sec:implementation} shows how we compute its minimizers.

\begin{remark}
In both methods, we recover a Hamiltonian as an optimal-recovery solution in the prescribed RKHS. In the single-trajectory setting, however, this should not in general be interpreted as unconditional identification of the true Hamiltonian throughout an open phase-space region. Regional recovery of \(\nabla H\), and hence of \(H\) up to an additive constant, requires additional structural assumptions.
\end{remark}

\subsection{Comparison of the 2-step and 1-step methods}
In the 2-step method, we discard the physics constraints in the first step and use only the data. Then, in the second step, we compute derivatives of the approximants from the first step to create values for the optimization problem of the second step. However, in a scarce-data regime, the derivative of the approximant may not provide an accurate approximation for the derivative of the true unknown functions. 
The 1-step method surmounts these obstacles by taking into account all information from data and the physics of the problem simultaneously. 
We thus expect that the 1-step method provides more accurate approximations especially in a scarce-data regime. This argument is numerically validated on various examples in section \ref{sec:numerics}.

\subsection{Computational Graph Completion: the visual perspective of the methods}\label{sec:CGC}

A visual scheme of the methodology is given in figure \ref{fig:comp_graph}. Inspired by the Computational Graph Completion (CGC) framework \cite{houman_cgc}, we model our problem as a computational graph problem by representing state variables as vertices, and  functional dependencies between variables as  directed edges. We represent unknown functions as red edges, and  data as dotted edges. Then solving our problem is equivalent to completing the graph by approximating the unknown functions with the minimizers of optimal recovery problems over suitably-chosen reproducing kernel Hilbert spaces. This methodology is equivalent to replacing the unknown functions of the graph with Gaussian Processes and computing their MAP estimators given constraints imposed by the data and the structure of the graph. The RKHS perspective inherits the convergence and error estimate guarantees for kernel methods. On the other hand, the GP perspective allows us to access the whole posterior distribution of point-wise approximations and employ them for UQ estimates. 

\begin{figure}[h!]
    \centering
    \includegraphics[width=0.51\linewidth]{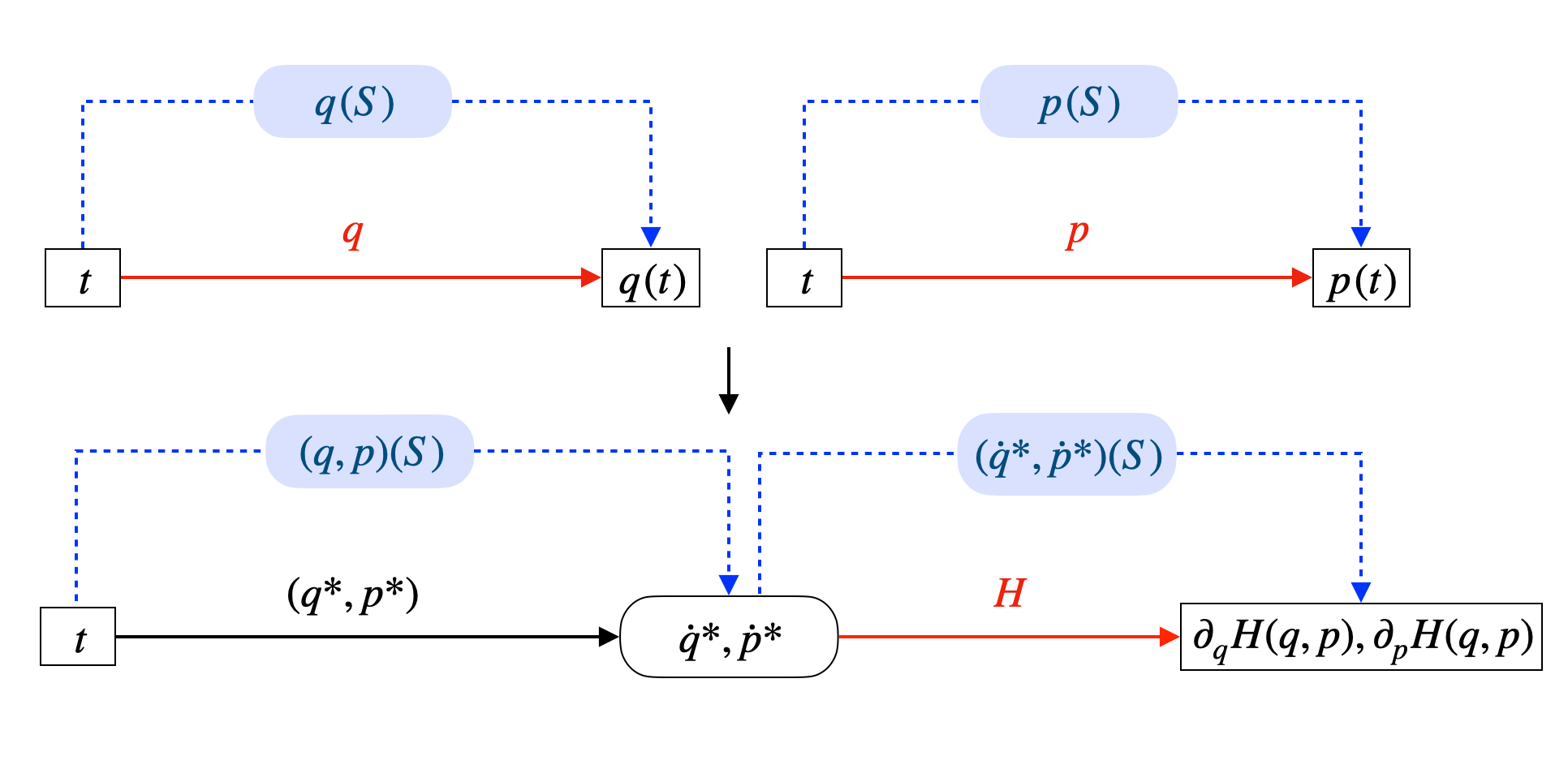}
    \includegraphics[width=0.45\linewidth]{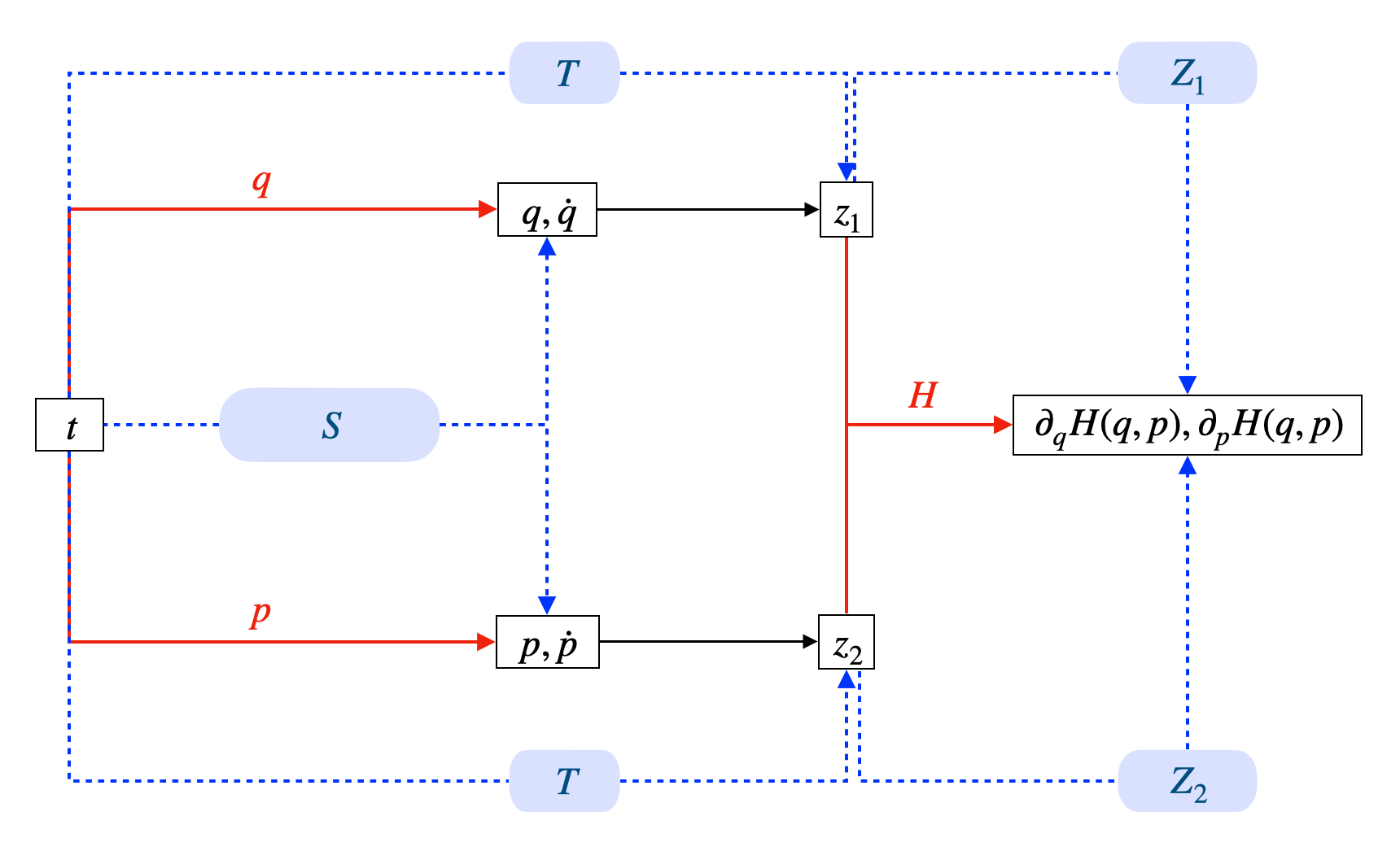}
    \caption{Computational graph \cite{houman_cgc} of the 2-step method (left) - Computational graph of the 1-step method (right), for one trajectory.}
    \label{fig:comp_graph}
\end{figure}

\section{Theoretical Analysis}\label{sec:theory}

In this section, we present \textit{a priori} error bounds for the 1-step method introduced in section~\ref{onestep_method}. These bounds quantify the approximation error of the learned trajectories and Hamiltonian in terms of the smoothness of the true functions and the fill distance of the sampled collocation points. The analysis builds on classical results in the approximation theory of RKHS(s) \cite{wendland2005approximate, wendland2004scattered, narcowich2005sobolev, fuselier2012scattered}, and adapts them to the context of Hamiltonian systems with coupled derivative constraints. For an expanded treatment of both 1-step and 2-step convergence theory for a PDE learning setup, see \cite[Supplementary Information B]{jalalian2025keql}. We refer to some lemmas and theorems proved in \cite{jalalian2025keql}; for convenience, we have restated them in appendix \ref{app:theory}. 

For this section, we rewrite optimization problem \eqref{1-step-optimal-recovery} in the following equivalent form
\begin{equation}\label{1-step-optimal-recovery-theory}
\begin{alignedat}{3}
    (y_N^\star, H_N^\star) \coloneqq &\argmin_{\wh y,\wh H} \|\wh H\|_{\Psi}^2 + \frac{1}{\lambda_1} \|\wh y_1\|_{\Gamma}^2 + \frac{1}{\lambda_2} \|\wh y_2\|_{\Sigma}^2
    \\
    &\quad \quad \text{subject to} \quad
    \begin{cases}
        \wh y(S) = y(S),\\        \dot{\widehat y}(T) = J \nabla \widehat H\big(\wh y(T)\big),
\quad
J = \begin{pmatrix} 0 & I \\ -I & 0 \end{pmatrix}.
    \end{cases}
\end{alignedat}
\end{equation}
Note that, in contrast to the previous notation, here we have made the dependence of the estimates on $N$ explicit by referring to them as $y^\star_N, H^\star_N$.
Our goal is to obtain \textit{a priori} error rates for the reconstructed functions $y_N^\star, H_N^\star$. 
We accomplish this for $y_N^\star$, however, 
notice that the equality constraints of optimization problem \eqref{1-step-optimal-recovery-theory} are only on $\nabla H$ and not on $H$. Thus, we can only obtain rates for the recovery of $\nabla H$ which is equivalent to the recovery of $H$ up to a constant.

To establish our error analysis, we require certain regularity conditions on $q,p$ and $H$ as well as fill-distance assumptions on data sampling. We state these assumptions below. Without loss of generality, we assume that $\lambda_1, \lambda_2$, the weights in the optimal recovery problem \eqref{1-step-optimal-recovery-theory}, are both equal to one as they only affect the constant and not the rates.

\begin{assumption}[\textbf{Assumptions on $q,p$}] \label{assumption_qp} Assume that the following holds:
\begin{enumerate}
\renewcommand{\labelenumi}{(\roman{enumi})}
    \item $\Gamma, \Sigma: \mathbb{R} \times \mathbb{R} \to \mathbb{R}$ are continuous, positive-definite kernels with associated RKHS(s) $\mathcal{H}_\Gamma, \mathcal{H}_\Sigma$.
    \item $\mH_\Gamma$ is continuously embedded in the Sobolev space $ H^\gamma([0, t_f]; \mathbb{R}^m)$ for some $\gamma > 0$ and $q \in H^\gamma ([0,t_f];\R^m)$;
    \item $\mH_\Sigma$ is continuously embedded in the Sobolev space $ H^\sigma([0, t_f]; \mathbb{R}^m)$ for some $\sigma > 0$ and $p \in H^\sigma ([0,t_f];\R^m)$;
    \item the observation times $S_{\mathrm{obs}} \subset [0, t_f]$ are sampled with fill distance $\rho_N$ defined as
    \begin{align*}
        \rho_{N} \coloneqq \sup_{t \in [0, t_f]} \inf_{t' \in S_{\mathrm{obs}}}\, | t - t' |.
    \end{align*}
\end{enumerate}
\end{assumption}

\begin{assumption}[\textbf{Assumptions on $H$}]  \label{assumption_H} Assume that the following holds:
\begin{enumerate}
\renewcommand{\labelenumi}{(\roman{enumi})}
    \item $\Psi: \mathbb{R}^{2m} \times \mathbb{R}^{2m} \to \mathbb{R}$ is a continuous, positive-definite kernel with associated RKHS $\mathcal{H}_\Psi$.
    \item $\mH_\Psi$ is continuously embedded in the Sobolev space $H^\eta(\Omega; \mathbb{R})$ for some $\eta > m+1$ and some $\Omega \subset \mathbb{R}^{2m}$ a compact domain containing the state-space point cloud generated by the observed trajectory at the collocation times
    \[ Z_{\mathrm{col}} \coloneqq\{ y(t) : t \in T_{\mathrm{col}} \} \subset \R^{2m},\] and $H \in H^\eta(\Omega; \mathbb{R})$; 
    
    \item For a compact domain $B\subset\Omega$ with Lipschitz boundary, define the trajectory fill distance in \(B\) by
    \[
    \varrho_N(B)
    := \sup_{x\in B}
    \inf_{x'\in Z_{\mathrm{obs}}\cap B}
    \|x-x'\|_2,
    \]
where $Z_{\mathrm{obs}}$ is the observed phase-space point cloud
\[
Z_{\mathrm{obs}}:=\{y(s):s\in S_{\mathrm{obs}}\}.
\]
\end{enumerate}
\end{assumption}

\begin{theorem}[\textbf{\textit{A priori} error bounds for the 1-step method}]\label{thm:rates}
Assume that assumptions \ref{assumption_qp} and \ref{assumption_H} hold. Let the functions $q^\star_N \in \mathcal{H}_\Gamma, \, p^\star_N \in \mathcal{H}_\Sigma$, and $H^\star_N \in \mathcal{H}_\Psi$ be the minimizers of the 1-step optimization problem \eqref{1-step-optimal-recovery-theory}.
Then, there exist constants $\rho_0, \varrho_0(B) \in (0, 1)$ so that whenever $\rho_N < \rho_0$ and $\varrho_N(B) < \varrho_0 (B)$ it holds that 
\begin{enumerate}
\renewcommand{\labelenumi}{(\alph{enumi})}
    \item if $q \in \mH_\Gamma$ and $p \in \mH_\Sigma$, then for any smoothness indices $0 \leq \gamma'<\gamma$, $0 \leq \sigma'<\sigma$, there exist constants $C_q, C_p > 0$ such that 
    \begin{align*}
        \|q^\star_N - q\|_{H^{\gamma'}}^2 &\leq C_q \, \rho_N^{2(\gamma - \gamma')} \left( \|H \|_{\Psi}^2 + \| q\|_{\Gamma}^2 +  \|p\|_{\Sigma}^2  \right), \\
        \|p^\star_N - p\|_{H^{\sigma'}}^2 &\leq C_p \, \rho_N^{2(\sigma - \sigma')} \left( \|H \|_{\Psi}^2 + \| q\|_{\Gamma}^2 +  \|p\|_{\Sigma}^2  \right);
    \end{align*}
    \item if $q \in \mH_\Gamma^2$, $p \in \mH_\Sigma^2$, and $\sigma, \gamma>3/2$, then for any smoothness indices $3/2< \gamma'<\gamma,\, 3/2<  \sigma'<\sigma$, there exists a constant $C_H(B) > 0$ such that 
\begin{align*}
\| \nabla H - \nabla H^\star_N \|_{L^\infty({B})}^2
\le C_H{(B)} \left(\varrho_N {(B)}^{2(\eta - 1- m)} + \rho_N^{2(\gamma -\gamma')} + \rho_N^{2(\sigma -\sigma')}\right) \left( \|H\|_\Psi^2 + \| q \|_{\Gamma^2}^2 + \| p \|_{\Sigma^2}^2 \right),
\end{align*}
where the nested RKHSs $\mH_\Gamma^2, \mH_\Sigma^2$ and their corresponding norms are defined in appendix \ref{app:theory}.
\end{enumerate}
\end{theorem}

\begin{remark}
 {We emphasize that the above error bound is conditional on the sampled
trajectory being sufficiently space-filling in the region where recovery is
claimed. In particular, the estimate for $\nabla H$ holds on a set
$B\subset \mathbb R^{2m}$ only when the observation point cloud
\(Z_{\mathrm{obs}}=\{y(s):s\in S_{\mathrm{obs}}\}\) has sufficiently small fill distance in
$B$. Therefore, in the one-trajectory setting, the theorem should be understood
as a local recovery result along the observed trajectory, or in a neighborhood
of it, rather than as global recovery of $\nabla H$ on an arbitrary phase-space
domain. Accuracy for new initial conditions is justified only when the resulting
trajectory remains in the data-covered region $B$.}
\end{remark}

\begin{proof}
\begin{enumerate}
\renewcommand{\labelenumi}{(\alph{enumi})}
    \item 
     Observe that $(q^\star_N-q)(S_{\mathrm{obs}}) = 0$ as $q^\star_N$ is interpolating $q$ passing through $S_{\mathrm{obs}}$ and so we can apply the noiseless Sobolev sampling inequality 
     \ref{prop:sobolev-sampling-inequality}(a). Indeed, directly applying that result followed by the continuous embedding assumption \ref{assumption_qp}(ii) we obtain that for $\rho_{N} < \rho_0$, 
\begin{align*}
    \| q^\star_N - q \|_{H^{\gamma'}} 
    \le C \, \rho_N^{\gamma - \gamma'}\, \| q^\star_N - q \|_{H^{\gamma}}
    \le \Tilde{C}\, \rho_N^{\gamma - \gamma'}\, \| q^\star_N - q \|_\Gamma.
\end{align*}
Using the triangle inequality, the identity $(a + b)^2 \le 2 (a^2 + b^2)$, and $\min a + \min b \leq \min (a+b)$ for $q^\star_N$, we can write 
\begin{equation}\label{eq:rate-q}
    \begin{aligned}
        \| q^\star_N - q \|^2_{H^{\gamma'}}
        \le 2\,\Tilde{C}^2 \, \rho_N^{2(\gamma -\gamma')}\, \left( \| q^\star_N\|_\Gamma^2 + \| q \|_\Gamma^2 \right) 
        \le C_q \, \rho_N^{2(\gamma -\gamma')}\, \left( \|H\|_\Psi^2 + \| q \|_\Gamma^2 + \| p \|_\Sigma^2 \right).
    \end{aligned}
\end{equation}
The proof for $p$ is done similarly.

\item Let us break down the proof into several parts.
\\
\textbf{Finding a uniform bound for $H_N^*$:}
For $q$, define the optimal interpolant 
\begin{align*}
    \overline{q}_N:= \argmin_{r \in \mH_\Gamma}  {\| r\|_\Gamma}
    \quad \text{s.t.} \quad
    r(S) = q(S).
\end{align*}
Representer theorem gives $\overline{q}_N = \Gamma(S, \cdot)^\top \Gamma(S, S)^{-1} q(S)$ and a direct calculation using 
the reproducing property implies that 
$\langle q - \overline{q}_N, \overline{q}_N \rangle_\Gamma =0$
which in turn gives the identity 
\begin{align*}
\| q - \overline{q}_N \|_\Gamma^2 = \| q\|_{\Gamma}^2 - \| \overline{q}_N\|_\Gamma^2.  
\end{align*}
Moreover, by $\min a + \min b \leq \min (a+b)$, we have the bound $\| \overline{q}_N \|_{\Gamma} \leq \| q^\star_N \|_\Gamma.$
Thus we obtain the lower bound
\begin{equation}\label{disp-000}
    \| q^\star_N \|_\Gamma^2 \ge \| q\|_\Gamma^2 - \| q - \overline{q}_N\|_\Gamma^2.
\end{equation}
On the other hand, thanks to the assumption that $q \in \mH_\Gamma^2$, applying lemma 
\ref{lem:H-gamma-bound} then gives the bound 
\begin{align*}
\| q - \overline{q}_N\|_{\Gamma}^2 
\le \| q - \overline{q}_N\|_{L^2([0,t_f])} \| q\|_{\Gamma^2},  
\end{align*}
and since $(\bar q_N-q)(S_{\mathrm{obs}}) = 0$, we apply the noiseless Sobolev sampling inequality 
\ref{prop:sobolev-sampling-inequality}(a) to the bound above to further get the bound
\begin{align*}
\| q - \overline{q}_N\|_{\Gamma}^2 
\le \bar C \rho_N^\gamma \| q \|_{\Gamma} \| q\|_{\Gamma^2}.
\end{align*}
Putting \eqref{disp-000} and the bound above together, and following a similar procedure for $p$, we get the bounds
\begin{equation}\label{disp-222}
\begin{aligned}
    &\| q\|_\Gamma^2 - \| q^\star_N \|_\Gamma^2 \le \bar C \rho_N^\gamma \| q \|_{\Gamma} \| q\|_{\Gamma^2},\\
    &\| p\|_\Sigma^2 - \| p^\star_N \|_\Sigma^2 \le \bar C \rho_N^\sigma \| p \|_{\Sigma} \| p\|_{\Sigma^2},
\end{aligned}
\end{equation}
where by abuse of notation, we redefine $\bar C$ as the maximum of the constants for $q$ and $p$. Finally, the optimality of $q^\star_N, p^\star_N, H^\star_N$ gives 
\begin{equation*}\label{disp-3}
    \| H^\star_N \|^2_{\Psi} \le \| H \|_\Psi^2 + 
    \big[ \| q \|_\Gamma^2 - \| q^\star_N \|_\Gamma^2 \big] + \big[ \| p \|_\Sigma^2 - \| p^\star_N \|_\Sigma^2 \big],
\end{equation*}
and putting together \eqref{disp-222} with the bound above gives
\begin{align*}
\| H^\star_N \|^2_{\Psi} \le \| H \|_\Psi^2 + 
\bar C \big(\rho_N^\gamma \| q \|_\Gamma  \| q \|_{\Gamma^2} + \rho_N^\sigma \| p \|_\Sigma  \| p \|_{\Sigma^2}\big),
\end{align*}
which by using the embeddings 
$\mH_\Sigma^2 \subset \mH_\Sigma, \mH_\Gamma^2 \subset \mH_\Gamma$, the assumptions that $\rho_N < 1$ and $\gamma, \sigma >0$, and bounding all the constants by their maximum value and denoting it as $\bar C_{\mathrm{emb}}$ further gives the bound
\begin{equation}\label{disp-5}
\| H^\star_N \|^2_{\Psi} \le \bar C_{\mathrm{emb}} \left( \| H \|_\Psi^2 + 
\| q \|_{\Gamma^2}^2 + \| p \|_{\Sigma^2}^2\right). 
\end{equation}
Thus, we have an \textit{a priori} bound on the RKHS norm of the reconstruction $H_N^\star$ that depends only on the true functions and fill distance.
\\
\textbf{Controlling the error between $\nabla H$ and $\nabla H^\star_N$:}
Fix \(s\in S_{\mathrm{obs}}\). 
From equality constraints in \eqref{1-step-optimal-recovery-theory}, and from the Hamiltonian equations for true dynamics, we know that 
\begin{align*}
\dot y(s)=J\nabla H(y(s)),\qquad 
\dot y_N^\star(s)=J\nabla H_N^\star\big(y(s)\big), 
\end{align*}
where $\|J\|_2=1$, so
\begin{equation}\label{disp-12}
\begin{aligned}
\big\|\nabla H(y(s))-\nabla H_N^\star(y(s))\big\|_2
\le\; \big\|\dot y(s)- \dot y_N^\star(s)  \big \|_2 \le \big\|\dot y-\dot y_N^\star\big\|_{L^\infty([0,t_f])}.
\end{aligned}
\end{equation}
Now, note that by triangle inequality on the two-norms and taking sup, we have
\begin{align}\label{disp-333}
\|\dot y - \dot y_N^\star\|_{L^\infty([0,t_f])}
\;\le\;
\|\dot q - \dot q_N^\star\|_{L^\infty([0,t_f])}
\;+\;
\|\dot p - \dot p_N^\star\|_{L^\infty([0,t_f])}.
\end{align}
For $\gamma',\sigma'>3/2$, by Sobolev embedding theorem (with $\alpha, d = 1$ in theorem \ref{prop:sobolev-embedding}),
\begin{align*}
H^{\gamma'}([0,t_f]), H^{\sigma'}([0,t_f])\hookrightarrow C^1([0,t_f]).
\end{align*}
In particular, there exist $C_{\mathrm{emb,q}},C_{\mathrm{emb,p}}>0$ such that
\begin{equation}\label{eq:embed-Hr-C1}
\begin{aligned}
\|\dot q - \dot q_N^\star\|_{L^\infty([0,t_f])} \leq \| q -  q_N^\star\|_{C^1([0,t_f])}\;\le\; C_{\mathrm{emb,q}}\,\| q -  q_N^\star\|_{H^{\gamma'}([0,t_f])},\\
\|\dot p - \dot p_N^\star\|_{L^\infty([0,t_f])}\leq \| p -  p_N^\star\|_{C^1([0,t_f])}\;\le\; C_{\mathrm{emb,p}}\,\| p -  p_N^\star\|_{H^{\sigma'}([0,t_f])}.
\end{aligned}
\end{equation}
Denote $C_{\mathrm{emb}} = \max (C_{\mathrm{emb,q}}, C_{\mathrm{emb,p}})$. Putting equations \eqref{eq:rate-q},  \eqref{disp-333} and \eqref{eq:embed-Hr-C1} together and again using $(a + b)^2 \le 2 (a^2 + b^2)$, we get
\begin{align*}
    \|\dot y - \dot y_N^\star\|_{L^\infty([0,t_f])}^2
    \le 2\, C_{\mathrm{emb}}^2
    \left( C_q \, \rho_N^{2(\gamma -\gamma')} + C_p \, \rho_N^{2(\sigma -\sigma')}\right)\, \left( \|H\|_\Psi^2 + \| q \|_\Gamma^2 + \| p \|_\Sigma^2 \right).
\end{align*}
Putting together equations \eqref{disp-12} with the bound above, we get that for $s \in S_{\mathrm{obs}}$,
\begin{equation}\label{disp-17}
\begin{aligned}
\big\|(\nabla H-\nabla H_N^\star)\big(y(s)\big)\big\|_2^2  
\le 2\, C_{\mathrm{emb}}^2
\left( C_q \, \rho_N^{2(\gamma -\gamma')} + C_p \, \rho_N^{2(\sigma -\sigma')}\right)\, \left( \|H\|_\Psi^2 + \| q \|_\Gamma^2 + \| p \|_\Sigma^2 \right).
\end{aligned}
\end{equation}
We have thus shown that \(\nabla H\) and \(\nabla H_N^\star\) are close on the observed phase-space point cloud \(Z_{\mathrm{obs}}\).
\\
\textbf{Discrete to continuous:}
Assumption \ref{assumption_H}(ii) implies
$\nabla H \in H^{\eta-1}(\Omega;\mathbb{R}^{2m})$.
By the noisy Sobolev sampling inequality
\ref{prop:sobolev-sampling-inequality}(b), under the assumption that
$\varrho_N(B)\leq \varrho_0(B)$, we obtain
\begin{equation}\label{disp-19}
\|\nabla H-\nabla H_N^\star\|_{L^\infty(B)}
\leq
C_B^\star \varrho_N(B)^{\eta-1-m}
\|\nabla H-\nabla H_N^\star\|_{H^{\eta-1}(B)}
+
2\|(\nabla H-\nabla H_N^\star)|_{Z_{\mathrm{obs}}\cap B}\|_\infty .
\end{equation}
Using the fact that
\[
\nabla:H^\eta(\Omega)\to H^{\eta-1}(\Omega;\mathbb{R}^{2m})
\]
is a bounded linear operator, the continuous embedding assumption
\ref{assumption_H}(ii), and the triangle inequality, we can further write
\begin{align*}
\|\nabla H-\nabla H_N^\star\|_{H^{\eta-1}(B)}
\leq
C_{\Omega,\eta}\|H-H_N^\star\|_{H^\eta(\Omega)} 
\leq
\widehat C\left(\|H\|_\Psi+\|H_N^\star\|_\Psi\right),
\end{align*}
which, together with \eqref{disp-5} and
$(a+b)^2\leq 2(a^2+b^2)$, implies
\begin{equation}\label{disp-20}
\|\nabla H-\nabla H_N^\star\|_{H^{\eta-1}(B)}^2
\leq
\widetilde C_H
\left(
\|H\|_\Psi^2+
\|q\|_{\Gamma^2}^2+
\|p\|_{\Sigma^2}^2
\right).
\end{equation}
Putting \eqref{disp-17} and \eqref{disp-20} into
\eqref{disp-19}, using $(a+b)^2\leq 2(a^2+b^2)$ and the embeddings
$\mathcal H_\Sigma^2\subset\mathcal H_\Sigma$ and
$\mathcal H_\Gamma^2\subset\mathcal H_\Gamma$, results in
\begin{align*}
\|\nabla H-\nabla H_N^\star\|_{L^\infty(B)}^2
\leq
C_H(B)
\left(
\varrho_N(B)^{2(\eta-1-m)}
+\rho_N^{2(\gamma-\gamma')}
+\rho_N^{2(\sigma-\sigma')}
\right)
\left(
\|H\|_\Psi^2+
\|q\|_{\Gamma^2}^2+
\|p\|_{\Sigma^2}^2
\right).
\end{align*}
\end{enumerate}
\end{proof}

\begin{remark}
We can define
$\|H\|_{\Psi,0} \coloneqq \min_{c \, \in \, \mathbb{R}} \|H-c\|_{\Psi}.$
Since adding a constant does not affect the gradient, i.e.\ 
\(\nabla (H-c) = \nabla H\), all of the theoretical steps in the proof remain valid if we replace \(\|H\|_{\Psi}\) by \(\|H\|_{\Psi,0}\). 
In particular, working with \(\|H\|_{\Psi,0}\) yields a tighter bound.
\end{remark}

\section{Implementation}\label{sec:implementation}

In this section, we expand the methodology introduced in section \ref{sec:method}. For the 2-step method, we present the closed form solutions of optimization problems presented in \ref{sec:method} and, for the 1-step method, we explain in detail the reduction of the optimization problem and the numerical optimization algorithm implemented for its solution. Then, we discuss kernel and hyperparameter choices and end the discussion with remarks on the energy conservation and numerical integration properties of our methods.

\subsection{Kernels acting on linear functionals}\label{sec:kernels-on-functionals}

Let $\mX$ be any set. Let $\mH$ be the RKHS associated to the kernel $K:\mX \times \mX \to \R$. Then $\mH^*$ denotes the dual space of $\mH$, i.e. it is the space of all bounded linear functionals $\phi: \mH \to \R$.
For any $x\in\mX$, write the kernel section
$K_x := K(\cdot,x)\in\mH$. By abuse of notation, for a bounded linear functional $\phi\in\mH^*$, we define
\begin{align*}
K(\phi,x) \;:=\; \phi\big(K_x\big) \;\in\; \R.
\end{align*}
Then, naturally, for a vector of functionals $\phi=(\phi_1,\dots,\phi_N)\in(\mH^*)^N$, we have
\begin{align*}
K(\phi,x) \;:=\; \phi\big(K_x\big) = \big(\, \phi_1(K_x),\dots,\phi_N(K_x) \,\big)^\top \;\in\; \R^N.
\end{align*}
Then, for two vectors of functionals $\phi=(\phi_i)_{i=1}^N$ and $\psi=(\psi_j)_{j=1}^M$,
we have
\begin{align*}
K(\phi,\psi) \;\in\; \R^{N\times M},
\qquad
K(\phi,\psi)_{ij}
\;:=\;
\psi_j\big(K(\phi_i,\cdot)\big)
\;=\;
\psi_j\big( K_{\phi_i} \big),
\end{align*}
where $K(\phi_i,\cdot)\in\mH$ satisfies $\langle f, K(\phi_i,\cdot)\rangle_\mH=\phi_i(f)$ for all $f \in \mH$.

\begin{theorem}[generalized representer theorem {{\cite[Cor.~17.12]{owhadi2019operator}}}]\label{prop:generalized-rep-theorem}
Let $\mH$ be an RKHS associated to the kernel $K: \mX \times \mX \to \R$ and let $\phi=(\phi_1,\dots,\phi_N)\in(\mH^*)^N$ be a vector of bounded linear functionals on RKHS $\mH$. Consider
\begin{equation*}
f^\star := \argmin_{f \,\in \,\mH} \| f\|_{\mH} \quad \text{s.t.} \quad 
 \phi_i(f) = \xi_i, \qquad i=1, \dots, N, 
\end{equation*}
for $\xi = (\xi_1, \dots, \xi_N) \in \R^N$. If $K(\phi,\phi)$ is invertible, then every minimizer $f^\star$ has the form 
\begin{align*}
f^\star = K(\phi, \cdot)^\top K(\phi, \phi)^{-1} \xi.
\end{align*}
Moreover, it holds that
\begin{align*}
\|f^\star\|_\mH^2 + \tfrac{1}{\lambda}\|\phi f - \xi\|_2^2
\;=\; \xi^\top\big(K(\phi,\phi)+\lambda I\big)^{-1} \xi.
\end{align*}
\end{theorem}

We write $\delta_r$ for the bounded point-evaluation functional at $r$; i.e., for a function $f$, $\delta_r(f)=f(r)$. We define two vectors of bounded linear functionals used in the derivation of both methods.
Further, we assume $\Gamma,\Sigma\in C^1(\R\times\R)$ and
$\Psi\in C^2(\R^{2m}\times\R^{2m})$ so that the derivative evaluations below are bounded linear functionals on the corresponding RKHSs.

\subsection{Implementing 2-step method}
Denote 
\begin{align*}
\phi_{\mathrm{obs}} \;\coloneqq\; (\delta_{s_1},\dots,\delta_{s_{N_{\mathrm{obs}}}})^\top \;\in\;
\big(\mH_\Gamma^*\big)^{N_{\mathrm{obs}}} \text{or $\big(\mH_\Sigma^*\big)^{N_{\mathrm{obs}}}$ (depending on context).}
\end{align*}
By the generalized representer theorem (Theorem \ref{prop:generalized-rep-theorem}) applied with
$\mH=\mH_\Gamma$, $\phi=\phi_{\mathrm{obs}}$, $\xi=q(S)$ and with $\mH=\mH_\Sigma$, $\phi=\phi_{\mathrm{obs}}$, $\xi=p(S)$ to problems \eqref{q-kernel-interpolant} and \eqref{p-kernel-interpolant} respectively, 
their unique minimizers are given by
\begin{align}
q^\star(\tau) \;=\; \Gamma(\tau,S)\,\Gamma(S,S)^{-1}\,q(S),\label{interp_q}\\
p^\star(\tau) \;=\; \Sigma(\tau,S)\,\Sigma(S,S)^{-1}\,p(S).\label{interp_p}
\end{align}
Now, given the approximants $q^\star,p^\star$, in order to create input for $H$, since $\Gamma,\Sigma\in C^1$, we can explicitly differentiate these interpolants in time as follows
\begin{align*}
\dot q^\star(\tau) \;=\; \partial_\tau \Gamma(\tau,S)\,\Gamma(S,S)^{-1}\,q(S),\qquad
\dot p^\star(\tau) \;=\; \partial_\tau \Sigma(\tau,S)\,\Sigma(S,S)^{-1}\,p(S).
\end{align*}
We will use the stacked derivative vector at collocation times in the form
\begin{align*}
z \;=\; \begin{pmatrix}\dot q^\star(T)\\[2pt] -\,\dot p^\star(T)\end{pmatrix} \in \R^{2Nm},
\end{align*}
reflecting Hamilton’s equations $\dot q=\partial_p H$ and $\dot p=-\partial_q H$. Denote
\begin{align}\label{varphi}
\varphi \coloneqq \big(\,\varphi^{(p^\star)}_{1,1},\dots,\varphi^{(p^\star)}_{1,m},\,\dots,\,\varphi^{(p^\star)}_{N,1},\dots,\varphi^{(p^\star)}_{N,m},\,
\varphi^{(q^\star)}_{1,1},\dots,\varphi^{(q^\star)}_{N,m}\big)^\top \in (\mH_\Psi^*)^{2Nm},
\end{align}
where
\begin{align*}
\varphi^{(p^\star)}_{i,j} \coloneqq \delta_{y^\star(t_i)}\circ \partial_{p^\star_j},
\qquad
\varphi^{(q^\star)}_{i,j} \coloneqq \delta_{y^\star(t_i)}\circ \partial_{q^\star_j},
\qquad i=1,\dots,N, \qquad j = 1, \dots, m.
\end{align*}
Note that problem \eqref{H-kernel-interpolant} can now be rewritten in the equivalent form
\begin{equation}\label{eq:H-regression}
H^\star = \argmin_{\wh H\,\in \,\mH_{\Psi}}\;\; \|\wh H\|_\Psi^2
\;+\; \big\|\Lambda^{-1/2}\big(\, \varphi(\wh H) - z \,\big)\big\|_2^2,
\end{equation}
where $\Lambda \in \mathbb{R}^{2Nm\times 2Nm}$ is a block-diagonal matrix defined as
\begin{align*}
\Lambda = \begin{bmatrix}
\lambda_1 I_{Nm} & 0 \\
0 & \lambda_2 I_{Nm}
\end{bmatrix}.
\end{align*}
By the generalized representer theorem (Theorem \ref{prop:generalized-rep-theorem}) applied with
$\mH=\mH_\Psi$, $\phi=\varphi$, $\xi=z$ to problem \eqref{eq:H-regression}, its unique minimizer is given by
\begin{align}\label{interp_H}
H^\star(x) \;=\; \Psi(\varphi,x)^\top \big( \Psi(\varphi,\varphi) + \Lambda \big)^{-1} z,
\end{align}
where
\begin{align*}
\Psi\big(\varphi^{(p^\star)}_{i,j},x\big) \;=\; \partial_{p^\star_j}^{y}\,\Psi\big(y,x\big)\big|_{y=y(t_i)},\quad
\Psi\big(\varphi^{(q^\star)}_{i,j},x\big) \;=\; \partial_{q^\star_j}^{y}\,\Psi\big(y,x\big)\big|_{y=y(t_i)}, \quad i=1,\dots,N, \quad j = 1, \dots, m.
\end{align*}
thus
\begin{align*}
\Psi(\varphi, x)
=
\begin{pmatrix}
\partial_{p^\star}\Psi\big(y(T),\,x\big)\\[2pt]
\partial_{q^\star}\Psi\big(y(T),\,x\big)
\end{pmatrix} \in \R^{2Nm},
\end{align*}
and
\begin{align*}
\Psi(\varphi,\varphi)
&=
\begin{bmatrix}
\displaystyle \frac{\partial^2 \Psi}{\partial p^\star \,\partial {p^\star}'}(y(T),y(T))
& \displaystyle \frac{\partial^2 \Psi}{\partial p^\star \,\partial {q^\star}'}(y(T),y(T))\\[8pt]
\displaystyle \frac{\partial^2 \Psi}{\partial q^\star \,\partial {p^\star}'}(y(T),y(T))
& \displaystyle \frac{\partial^2 \Psi}{\partial q^\star \,\partial {q^\star}'}(y(T),y(T))
\end{bmatrix} \in\R^{2Nm\times 2Nm},
\end{align*}
where each block is an $Nm\times Nm$ matrix.

\subsubsection{Numerical implementation of the 2-step method}

The 2-step approach is the more convenient to implement as it only requires numerically solving the linear systems in \eqref{interp_q}, \eqref{interp_p} and \eqref{interp_H}.
Since the matrices $\Gamma(S,S)$ and $ \Sigma(S,S)$ are often ill-conditioned, in practice we regularize them by adding a small perturbation defined  as a multiple of the identity matrix and invert the perturbed matrices $\Gamma(S,S) + \lambda_q I_{{N_{\mathrm{obs}}}}$ and $ \Sigma(S,S) + \lambda_p I_{{N_{\mathrm{obs}}}}$  with regularization parameters $\lambda_q, \lambda_p >0$. The main computational cost is in the inversion of kernel matrices which can be handled using standard numerical solvers for small data sets. 

\subsection{Implementing 1-step method}\label{sec:1-step-implementation}

For the 1-step method, since we don't have data on the derivatives $\dot {\wh q}(T), \dot {\wh p}(T)$, we introduce the latent (slack) time-derivative functions $z_1,z_2: [0,t_f] \to \R^{m}$ and rewrite problem \eqref{1-step-optimal-recovery} as
\begin{align*}
 \min_{\wh q \,\in\, \mH_\Gamma,\,\wh p \,\in\, \mH_\Sigma,\,\wh H \, \in \,\mH_\Psi, \, z_1, \, z_2} &\|\wh H\|_{\Psi}^2 + \frac{1}{\lambda_1} \|\wh q\|_{\Gamma}^2 + \frac{1}{\lambda_2} \|\wh p\|_{\Sigma}^2
    \\
    &\text{subject to} \quad
    \begin{cases}
        \wh q(S), \dot {\wh q}(T) = q(S), z_1(T),\\
        \wh p(S), \dot {\wh p}(T) = p(S), z_2(T),\\
        z_1(T), z_2(T) = \partial_p \wh H((\wh q,\wh p)(T)), - \partial_q \wh H((\wh q,\wh p)(T)).
    \end{cases}
\end{align*}

We choose Tikhonov penalties for all constraints and obtain the unconstrained objective
\begin{equation}\label{unconstrained}
\begin{aligned}
\min_{\wh q \,\in\, \mH_\Gamma,\,\wh p \,\in\, \mH_\Sigma,\,\wh H \, \in \,\mH_\Psi, \, z_1, \, z_2}\;
\|\wh H\|_{\Psi}^2 + \frac{1}{\lambda_1} \|\wh q\|_{\Gamma}^2 + \frac{1}{\lambda_2} \|\wh p\|_{\Sigma}^2
&+\frac{1}{\lambda_1}\big\|\Phi\,\wh q - (q(S),z_1(T))\big\|_2^2\\
&+\frac{1}{\lambda_2}\big\|\Phi\,\wh p - (p(S),z_2(T))\big\|_2^2\\
&+\frac{1}{\lambda}\big\|\varphi \wh H - (z_1(T),-z_2(T))\big\|_2^2,
\end{aligned}
\end{equation}
where $\varphi$ is defined as before in \eqref{varphi} and 
\begin{align*}
\Phi \coloneqq
\big(\,\delta_{s_1},\dots,\delta_{s_{N_{\mathrm{obs}}}},\;\delta_{t_1}\circ\tfrac{\dif}{\dif t},\dots,\delta_{t_N}\circ\tfrac{\dif}{\dif t}\,\big)^\top
 \;\in\;
\big(\mH_\Gamma^*\big)^{N_{\mathrm{obs}}+N} \text{or $\big(\mH_\Sigma^*\big)^{N_{\mathrm{obs}}+N}$ (depending on context).}
\end{align*}
For fixed $z_1,z_2$, each term in \eqref{unconstrained} is a Tikhonov problem in an RKHS with a linear observation operator. By the generalized representer theorem (Theorem~\ref{prop:generalized-rep-theorem}) in its penalized form, minimizers $q^\star, p^\star, H^\star$ of problem \eqref{unconstrained} are of the form
\begin{align}
q^\star(\tau)
&= \Gamma\big(\tau,\Phi\big)\,\big(\Gamma(\Phi,\Phi)+\lambda_1 I\big)^{-1}
\begin{pmatrix} q(S)\\ z_1(T)\end{pmatrix},\label{eq:qstar-1step}\\[3pt]
p^\star(\tau)
&= \Sigma\big(\tau,\Phi\big)\,\big(\Sigma(\Phi,\Phi)+\lambda_2 I\big)^{-1}
\begin{pmatrix} p(S)\\ z_2(T)\end{pmatrix},\label{eq:pstar-1step}\\[3pt]
H^\star\big(x\big)
&= 
\Psi(x,\varphi) \left(\Psi(\varphi,\varphi)+\lambda I\right)^{-1}
\begin{pmatrix} z_1(T)\\ -\,z_2(T)\end{pmatrix}.\label{eq:Hstar-1step}
\end{align}
Substituting \eqref{eq:qstar-1step}-\eqref{eq:Hstar-1step} into \eqref{unconstrained} and using the norm identity in theorem \ref{prop:generalized-rep-theorem}, we obtain finite-dimensional reduced problem in $z_1,z_2$ as follows
\begin{equation}\label{eq:reduced-z}
\begin{aligned}
\min_{z_1,z_2}\quad
&\begin{pmatrix} q(S)\\ z_1(T)\end{pmatrix}^{\!\top}
\big(\Gamma(\Phi,\Phi)+\lambda_1 I\big)^{-1}
\begin{pmatrix} q(S)\\ z_1(T)\end{pmatrix}
\\[-1pt]
\quad+\;
&\begin{pmatrix} p(S)\\ z_2(T)\end{pmatrix}^{\!\top}
\big(\Sigma(\Phi,\Phi)+\lambda_2 I\big)^{-1}
\begin{pmatrix} p(S)\\ z_2(T)\end{pmatrix}
\\[-1pt]
\quad+\;
&\begin{pmatrix} z_1(T)\\ -\,z_2(T)\end{pmatrix}^{\!\top}
\left(\Psi(\varphi,\varphi)+\lambda I\right)^{-1}
\begin{pmatrix} z_1(T)\\ -\,z_2(T)\end{pmatrix}.
\end{aligned}
\end{equation}
Note that $(q^\star,p^\star)(T)$ depend on $(z_1,z_2)$ via \eqref{eq:qstar-1step}-\eqref{eq:pstar-1step}, so \eqref{eq:reduced-z} is generally nonconvex. Thus, we solve \eqref{eq:reduced-z} for $(z_1,z_2)$ using a smooth large-scale optimizer (e.g., L-BFGS) and warm-start with the outputs of the 2-step method. 

\subsubsection{Numerical implementation of the 1-step method}

The numerical implementation of the 1-step learning method is significantly more challenging as the objective function couples the estimation of the Hamiltonian $H$ with the dynamics in a single, joint optimization. The main bottleneck here lies in solving the resulting high-dimensional, nonlinear optimization problem over the function space $\mathcal{H}_\Psi$. 
To address this, we employ a quasi-Newton optimization algorithm, specifically L-BFGS, which is well-suited for large-scale smooth optimization problems. Moreover, to improve convergence and reduce sensitivity to poor initialization, we warm-start the 1-step solver using the output of the 2-step method. This strategy not only accelerates convergence but also improves the stability and reliability of the learned Hamiltonian.

\subsubsection{Problem-agnostic CGC framework 
\footref{fn:code}}\label{subsec:cgc}
While it is possible to directly implement the 1-step method described in section \ref{sec:1-step-implementation} specifically for Hamiltonian systems, we instead provide a more generic, problem-agnostic framework that implements the general setting of the CGC framework \cite{houman_cgc} explained in \ref{sec:CGC}. 
Denote the function representing the state of the system as 
$$
x = (\cdot, p, q, H(q,p)).
$$
Then, we construct our data matrix  $x(T) \in \R^{N \times 4}$ as the following: 
at time $s$, if the value of any of the evaluations is known either from data constraint or from physics constraint in problem \eqref{1-step-optimal-recovery}, then we enter the value; otherwise, we write $0$.
We initialize a candidate matrix $z(T)\in \R^{N \times 4}$ that approximates the $x(T)$ matrix. The optimization problem is then reduced to
$$
\argmin_z \|f\|^2 + \lambda_1\mathcal{L}_1(f, z) + \lambda_2\mathcal{L}_2(f, z) + \lambda_3\|z - x(T)\|_2^2,
$$
where
$$
|\|f||^2 \coloneqq \|q\|^2_\Sigma + \|p\|^2_\Gamma +  \|H\|^2_{\Psi}.
$$
The second term is the loss enforcing the dependencies between the unknown functions and the variables in $Z$:
$$
\mathcal{L}_1(f, Z) = |p(Z_{.,0}) - Z_{., 1}|^2 + |q(Z_{., 0}) - Z_{., 2}|^2 + |H(Z_{., 2}, Z_{., 1}) - Z_{., 3}|^2.
$$
The last term is the loss enforcing the Hamiltonian constraints between the unknown functions and the variables Z:
$$
\mathcal{L}_2(f, Z) = |\dot{p}(Z_{.,0}) + \partial_qH(Z_{., 2}, Z_{., 1})|^2 + |\dot{q}(Z_{., 0}) - \partial_pH(Z_{., 2}, Z_{., 1})|^2.
$$
The RKHS norms and the approximates for $p,q, H$ are defined using the representer theorems conditioned on the Z variables, and the optimization iterates until the the value of the loss is optimal. The Final functions estimates are then calculated with the representer theorems using the final value of $Z$.

The framework, written in Python, provides an object-oriented interface for writing the computational graph underlying any CGC problem. 
It is built on top of JAX \cite{jax2018github}, which allows the method to run on commercial CPUs or utilize the power of either Nvidia or AMD GPUs to boost the speed and performance of the CGC computation. For more technical explanation, see appendix \ref{sec:app_cgc}.

\subsection{Choosing kernels and hyperparameters}\label{sec:kernels}

While the 2-step and 1-step learning methods are generic and well-defined for any choice of kernels $\Gamma, \Sigma,\Psi$, the practical performance of these algorithms is closely tied to a good choice of kernels. In general, the choice of these kernels imposes constraints on our model
classes for the functions $q,p,H$. Moreover, in many applications we may have access to additional knowledge about the structure of $q,p,H$ and building that knowledge into the kernel will improve the accuracy of our approximation.

\subsubsection{Choosing $\Gamma, \Sigma$ for $q,p$}

In this work, we assume to have no \textit{a priori} knowledge on the functional form of the trajectories $q,p$. Therefore, in all of our experiments, we take the following Gaussian kernel for approximating $q,p$
\begin{align*}
    \Gamma = \Sigma = K, \quad K(t,t')= \exp(-\frac{|t-t'|^2}{2\theta^2}),
\end{align*}
with parameters $\theta = 1.0, \,\lambda = 1e-5.$

\subsubsection{Choosing $\Psi$ for $H$}

In a generic situation where there is no prior knowledge beyond smoothness about the structure of the Hamiltonian $H$, it is sensible to take a generic Gaussian kernel for approximating $H$ similarly to our approach for $q,p$. To this end, we considered the Gaussian kernel
\begin{align*}
        \Psi(q,p,q',p')= \exp(-\frac{\|q-q'\|^2}{2\theta^2}) \exp(-\frac{\|p-p'\|^2}{2\theta^2}),
    \end{align*}
with parameters $\theta = 1.0, \, \lambda = 0.001$.
However, many equations that arise in physical sciences have polynomial forms \cite[Sec.~1.2]{evans2022partial}. In particular, a lot of Hamiltonians are separable polynomials; that is, $H(q,p) = H_1(q) + H_2(p)$ where $H_1, H_2$ are polynomials. In these situations, we can incorporate that information by choosing the associated kernel to be a separable polynomial kernel 
    \begin{align*}
        \Psi(q,p,q',p')= (q q' + \theta)^d + (p p' + \theta)^d,
    \end{align*}
with parameters $\theta = 1.0, \, \lambda = 0.001,\, d=3.$ In cases when $H$ exhibits additive polynomial  and non-polynomial structure, we can take a mixed kernel of the form
\begin{align*}
        \Psi(q,p,q',p')= \exp(-\frac{\|q-q'\|^2}{2\theta^2}) + (p p' + \theta)^d.
    \end{align*}

\subsection{Optimization pipeline and hyperparameter choices}

All Hamiltonian experiments use the same two-stage optimization pipeline, with no system-specific hyperparameter tuning, grid search, or cross-validation.
\paragraph*{\textbf{Stage 1: two-step initialization and loss scaling.}}
We first construct an initial state $X_{\mathrm{init}}$ using the two-step method. The missing trajectory components $p$ and $q$ are reconstructed from the sparse observations by kernel regression, after which a Hamiltonian $H(q,p)$ is fitted from the reconstructed trajectories and their time derivatives. The resulting model is then integrated forward to provide the initialization for the one-step optimization.
\\
The same initialization is also used to determine the relative weights of the different terms in the CGC objective. Specifically, the unknown-function, constraint, and data-compliance penalties are rescaled so that their magnitudes are comparable to that of the RKHS regularization at $X_{\mathrm{init}}$.
\paragraph*{\textbf{Stage 2: one-step joint optimization.}}
Starting from $X_{\mathrm{init}}$, the variables $q$, $p$, and $H$ are learned simultaneously by minimizing the joint CGC objective with L-BFGS-B. The kernel hyperparameters are kept fixed throughout the optimization.

Across all experiments, we use Gaussian kernels for $q$ and $p$ with scale $1.0$ and regularization parameter $\alpha=10^{-5}$, and a separable polynomial kernel for $H$ with constant $1$, degree $3$, and regularization parameter $\alpha=10^{-3}$.
Thus, all systems are treated with the same initialization procedure, optimization method, kernel choices, and regularization parameters. No system-specific optimization or kernel hyperparameter tuning is performed.

\section{Numerical Experiments} \label{sec:numerics}

In this section, we numerically illustrate the effectiveness of the proposed methods using a range of benchmark Hamiltonian systems. The purpose of these experiments is to validate both the system identification and forecasting capabilities of the 1-step and 2-step kernel-based methods introduced in Sections \ref{twostep_method} and \ref{onestep_method}. In particular, we aim to demonstrate that the 1-step method outperforms the 2-step method in the scarce-data regime, due to its ability to jointly learn the dynamics and the governing Hamiltonian while enforcing the underlying physical structure. Moreover, we numerically compare the efficiency of different kernels for $H$.
We consider four representative Hamiltonian systems: the mass-spring system, the two-mass-three-spring system, the Hénon-Heiles system, and the nonlinear pendulum. These systems span a range of complexities, from linear to nonlinear, and from low-dimensional to higher-dimensional and chaotic regimes. For each system, we study the accuracy of trajectory reconstruction, as well as the accuracy of Hamiltonian recovery, and the fidelity of long-term forecasting. We end this section by presenting an experiment in the multiple-trajectory setting.

\subsection{Common structure}\label{sec:common}

For all the experiments, we follow a common structure for data generation and forecasting and use this common structure for both 2-step and 1-step methods. 

\subsubsection{Training data}

Given the explicit form of the true Hamiltonian function $H(y)$, we derive $\frac{\dif y}{\dif t}$ and along with the initial condition $y(t_0)$ and the time interval, numerically integrate using the \texttt{odeint} scheme from SciPy, which is based on variable-coefficient Adams methods for non-stiff problems, producing our collocation points of $q,p$ evaluated at the given time interval. 
The integration is performed on a set $T_{\mathrm{col}}$ of $N_{\mathrm{col}}=200$ uniformly sampled time points over a fixed time interval from $0$ to $t_f = 40$ (recall that, in vector form, this is represented as $T$ in the methodology). 
For sparsity factors $\alpha= 0, 0.5, 0.6, 0.7, 0.8, 0.9$, observed points are randomly subsampled with probability $1 - \alpha$ from the collocation set. The selection of observed points $S_{\mathrm{obs}}$ (denoted $S$ in vector form) is randomized at each experimental run; each experiment is repeated ten times with different random subsamplings, and the results are averaged.

\subsubsection{Testing data}

For testing, we use the recovered Hamiltonian $H^\star$ and the reconstructed initial state $y^\star(t_0)$ to simulate the system $\frac{\dif y^\star}{\dif t}$ on the previously seen interval (interpolation) and forward in time beyond the training window (forecasting).
Denote
\begin{align*}
    T_{\mathrm{test}} \coloneqq T_{\mathrm{int}} \cup T_{\mathrm{ext}},
\end{align*}
to be the \textit{testing set}, where $T_{\mathrm{int}}$ and $T_{\mathrm{ext}}$ are defined as following:
\begin{itemize}
    \item $T_{\mathrm{int}}$: first part of $T_{\mathrm{test}}$ defined as $T_{\mathrm{int}} \coloneqq T_{\mathrm{col}}\backslash T_{\mathrm{obs}}$, consisting of the unobserved collocation points ranging from $t=0$ to $t=40$, used to test the accuracy of interpolation;
    \item $T_{\mathrm{ext}}$: second part of $T_{\mathrm{test}}$, consisting of $200$ uniformly sampled time points taken from $t=40$ to $t=80$, used to test the accuracy of extrapolation (forecasting). 
\end{itemize}

\subsubsection{Evaluation}

Evaluation is performed using both qualitative and quantitative metrics on the testing set. We visualize the learned trajectories over time to compare them with the true dynamics. We also evaluate the relative errors in the recovered trajectories and Hamiltonians, as well as the deviation of the total energy over time. Specifically, the trajectory errors (state errors) are measured using relative error (RE), defined as
\begin{align*}
\mathrm{RE}_q
=
\frac{\|\,q - q^\star\,\|_{L^2(T_{\mathrm{test}})}}
     {\|\,q\,\|_{L^2(T_{\mathrm{test}})}} 
\approxeq
\frac{\left(\sum_{t_i \in T_{\mathrm{test}}} \|\,q(t_i)-q^\star(t_i)\,\|_2^2 \right)^{\frac12}}
     {\left(\sum_{t_i \in T_{\mathrm{test}}} \|\,q(t_i)\,\|_2^2 \right)^{\frac12}},
\end{align*}
for $q$ and similarly for $p$. 

\subsubsection{Energy conservation and numerical integration}

An important structural property of an autonomous Hamiltonian system is that its Hamiltonian is conserved along its exact flow. Since our learned vector field is generated from the learned Hamiltonian \(H^\star\) through Hamilton’s equations, \(H^\star\) is conserved along the exact continuous-time flow of the learned system. This property should be distinguished from conservation under numerical time integration: the numerical integrator may introduce energy drift. In our experiments, we use \texttt{odeint} with sufficiently small numerical tolerances to reduce such integration error, but we do not claim exact discrete conservation. 

\subsubsection{Presentation of results}
We considered two different regimes through which we compared the two methods numerically:  
(i) the Gaussian kernel for $H$, and  
(ii) the separable polynomial kernel for $H$.  
For all examples, we ran experiments at all sparsity factors mentioned above and report relative errors for both interpolation and extrapolation. We present the results in the following ways:

\begin{itemize}
    \item \textbf{Trajectory recovery plots:} For demonstration purposes, we only present these plots for sparsity factors $0.5$ and $0.8$, although results for all values are available in the code repository.
    
    \item \textbf{Error tables:} These tables report interpolation and extrapolation errors for the 2-step and 1-step methods with Gaussian and polynomial (or additive polynomial Gaussian) kernels, across all sparsity factors. 
    Note that the interpolation errors for the 2-step method (within a single system) are exactly the same regardless of the kernel used for $H$. The reason is that each of $q$ and $p$ is represented with the same Gaussian kernel, independent of the kernel chosen for $H$, as the choice of $H$’s kernel comes into play only in the second step. This is not the case for the 1-step method, where the choice of $H$’s kernel also affects interpolation, since the Hamiltonian constraints are enforced globally on all non-observed points. 

    \item \textbf{Error plots:} To quantify these observations, we compute the relative error (RE) for each variable as a function of the sparsity factor. 
    Each plot compares the 2-step and 1-step methods for a fixed kernel across all sparsity factors over the entire time period. 
    This process is performed separately for interpolation and extrapolation. 
    Extrapolation error plots are shown for each example in the main text, while the interpolation error plots are provided in Section~\ref{SM: plots}.
\end{itemize}

\subsection{Summary of results}

Across all systems, kernels, and sparsity levels, several clear patterns emerge.
\begin{itemize}
    \item \textbf{Interpolation vs.\ extrapolation:} 
    Interpolation errors are consistently much smaller than extrapolation errors. Both the 2-step and 1-step methods perform well when interpolating within known data regions, where the problem is well-posed. The approximation error remains bounded within the region covered by the data. In extrapolation settings, as expected from the theory of dynamical forecasting, accuracy decreases since the problem becomes ill-posed and errors amplify under time evolution. Nevertheless, the 1-step method consistently manages this instability better than the 2-step method.
    
    \item \textbf{2-step vs.\ 1-step method:} 
    For both interpolation and extrapolation, the 1-step method almost always achieves lower errors than the 2-step method. For extrapolation, even at very high sparsity levels ($0.7$ and above), where both methods face challenges, the 1-step method still yields the smaller errors. This demonstrates the advantage of learning the mapping in a single stage, where the 1-step method avoids error accumulation and simultaneously incorporates both the physics and the data of the problem.
    
    \item \textbf{Effect of sparsity:} 
    As sparsity increases, errors naturally grow for both interpolation and extrapolation. This trend is expected, since fewer data points reduce the effective constraints on the recovery problem. Importantly, across all sparsity levels, the 1-step method remains more accurate and stable than the 2-step method. The improved robustness comes from the fact that the 1-step formulation directly incorporates the dynamics, whereas the 2-step method loses accuracy through derivative approximations and decoupling of physics from data.
    
    \item \textbf{Kernel dependence:} 
    With the Gaussian kernel, errors increase faster at high sparsity compared to the separable or additive polynomial kernels. Polynomial-based kernels yield more stable extrapolation errors, and when combined with the 1-step method, provide particularly strong performance. This highlights not only the importance of designing appropriate kernels when prior knowledge is available, but also the ability of the 1-step method to consistently deliver superior results across different kernel choices.
\end{itemize}

We will now proceed to each individual system.

\subsection{Mass-spring system}
We first consider a harmonic oscillator with Hamiltonian
\begin{align*}
    H(q,p) = \frac{1}{2}(q^2 + p^2),
\end{align*}
with initial condition $y(t_0) = (0, 1)$.  
Being linear and one-dimensional, this classical system has periodic motion and serves as a simple test case for validating correctness. 
Since the Hamiltonian is symmetric with respect to $q$ and $p$, and their recovery behavior is very similar, we moved the recovery plots and the error table for $p$ to the appendix (see Section~\ref{SM: plots}).

\subsubsection{Trajectory recovery plots}
Figure~\ref{fig:recovery_ms_q} shows the recovered trajectories for $q$ over time, with sparsity factors $0.5$ and $0.8$. 
The ground truth trajectories are plotted in blue, while red crosses denote observed data points. 
Predictions from the 2-step and 1-step methods are shown with dashed and dash-dotted lines, respectively.  
At lower sparsity ($0.5$), both methods provide accurate approximations. 
At higher sparsity ($0.8$), the 1-step method remains stable and tracks the true dynamics more closely, 
whereas the 2-step method exhibits noticeable phase drift.

\begin{figure}[htbp]
    \centering
    \begin{subfigure}{0.36\textwidth}
        \includegraphics[width=\linewidth]{ 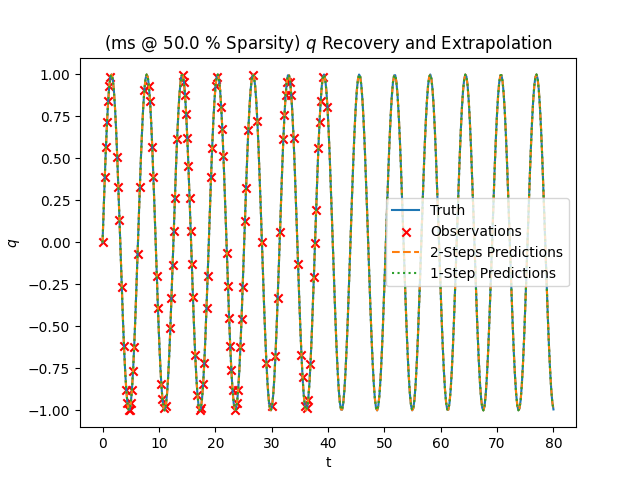}
        \caption*{Gaussian 0.5}
    \end{subfigure}
    \begin{subfigure}{0.36\textwidth}
        \includegraphics[width=\linewidth]{ 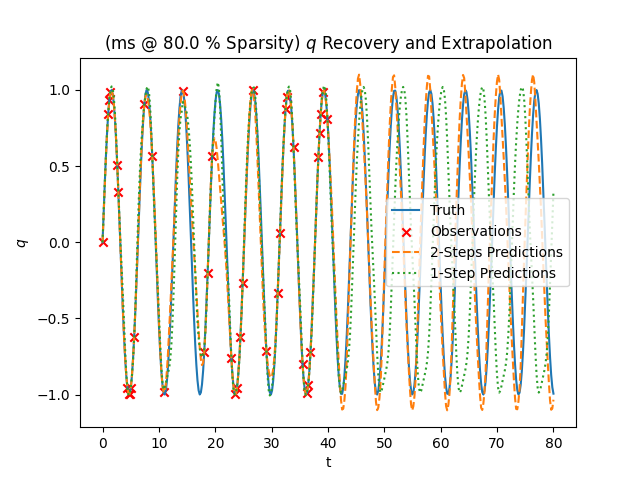}
        \caption*{Gaussian 0.8}
    \end{subfigure}
    \begin{subfigure}{0.36\textwidth}
        \includegraphics[width=\linewidth]{ 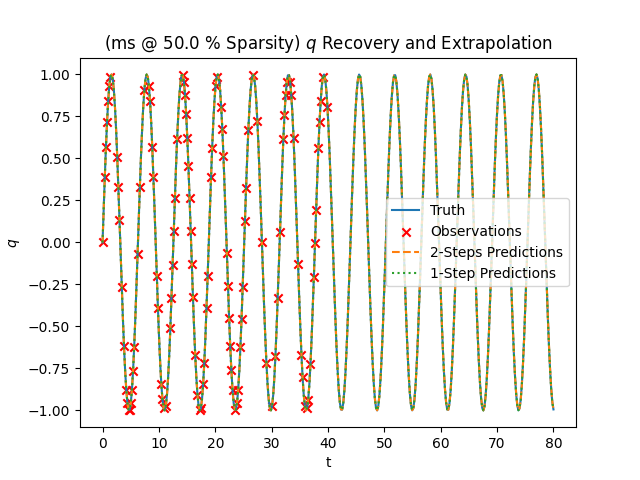}
        \caption*{Poly 0.5}
    \end{subfigure}
    \begin{subfigure}{0.36\textwidth}
        \includegraphics[width=\linewidth]{ 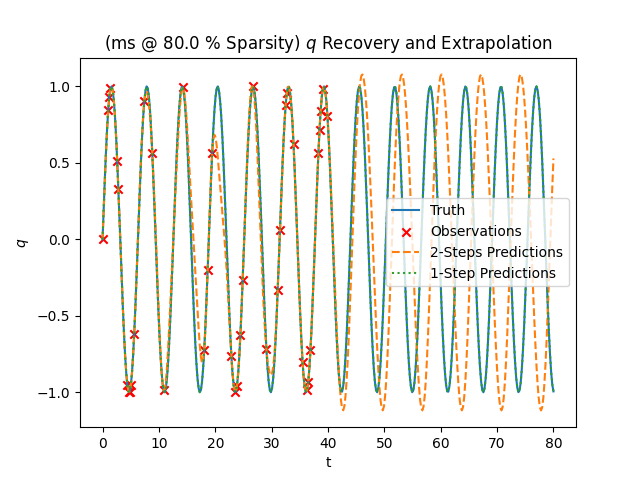}
        \caption*{Poly 0.8}
    \end{subfigure}
    \caption{Recovery of $q$ in the Mass-Spring system.}
    \label{fig:recovery_ms_q}
\end{figure}

\subsubsection{Error tables}
Table~\ref{tab:ms-q} reports the relative errors for $q$.  
Two main observations can be drawn from these results:  
\begin{enumerate}[label=(\roman*)]
    \item 1-step interpolation errors remain close to zero up to moderate sparsity ($0.6$). Extrapolation errors grow rapidly with increasing sparsity, but the 1-step method consistently outperforms the 2-step method. Even at high sparsity levels, where both methods deteriorate, the 1-step method maintains lower error magnitudes, highlighting the effectiveness of the 1-step approach under sparse data.  
    \item Within the same method, the separable polynomial kernel yields smaller errors than the Gaussian kernel, emphasizing the importance of kernel engineering when prior knowledge is available.  
\end{enumerate}
The corresponding table for $p$ is provided in Section~\ref{SM: plots}.

\begin{table}[ht]
\centering
\setlength{\tabcolsep}{12pt} 
\renewcommand{\arraystretch}{1.2} 
\scriptsize

\begin{tabular}{c} 

\begin{subtable}{\textwidth}
\centering
\begin{tabular}{lcccc}
\toprule
\multirow{2}{*}{\textbf{Sparsity}} & \multicolumn{2}{c}{\textbf{Two-Steps Method}} & \multicolumn{2}{c}{\textbf{One-Step Method}} \\
\cmidrule(lr){2-3}\cmidrule(lr){4-5}
& Interpolation   & Extrapolation   & Interpolation   & Extrapolation   \\
\midrule
0.0 & \textbf{0.000 $\pm$ 0.000} & 0.222 $\pm$ 0.000 & 0.005 $\pm$ 0.000 & \textbf{0.036 $\pm$ 0.000} \\
0.5 & 0.281 $\pm$ 0.297 & 1.493 $\pm$ 2.218 & \textbf{0.025 $\pm$ 0.027} & \textbf{0.138 $\pm$ 0.084} \\
0.6 & 0.815 $\pm$ 0.711 & 8.951 $\pm$ 13.545 & \textbf{0.195 $\pm$ 0.337} & \textbf{2.545 $\pm$ 5.074} \\
0.7 & 7.048 $\pm$ 7.192 & 59.604 $\pm$ 62.958 & \textbf{6.282 $\pm$ 11.262} & \textbf{34.315 $\pm$ 66.308} \\
0.8 & 23.992 $\pm$ 14.060 & 104.471 $\pm$ 43.668 & \textbf{26.322 $\pm$ 22.815} & \textbf{93.133 $\pm$ 54.790} \\
0.9 & 73.202 $\pm$ 19.451 & 141.407 $\pm$ 14.438 & \textbf{53.002 $\pm$ 32.089} & \textbf{90.448 $\pm$ 56.183} \\
\bottomrule
\end{tabular}
\caption*{Gaussian kernel}
\end{subtable}\\[1.5ex]
\begin{subtable}{\textwidth}
\centering
\begin{tabular}{lcccc}
\toprule
\multirow{2}{*}{\textbf{Sparsity}} & \multicolumn{2}{c}{\textbf{Two-Steps Method}} & \multicolumn{2}{c}{\textbf{One-Step Method}} \\
\cmidrule(lr){2-3}\cmidrule(lr){4-5}
& Interpolation   & Extrapolation   & Interpolation   & Extrapolation   \\
\midrule
0.0 & \textbf{0.000 $\pm$ 0.000} & 0.192 $\pm$ 0.000 & 0.005 $\pm$ 0.000 & \textbf{0.063 $\pm$ 0.000} \\
0.5 & 0.281 $\pm$ 0.297 & 1.243 $\pm$ 1.106 & \textbf{0.018 $\pm$ 0.013} & \textbf{0.160 $\pm$ 0.186} \\
0.6 & 0.815 $\pm$ 0.711 & 3.173 $\pm$ 3.762 & \textbf{0.037 $\pm$ 0.046} & \textbf{0.138 $\pm$ 0.099} \\
0.7 & 7.048 $\pm$ 7.192 & 17.616 $\pm$ 18.914 & \textbf{0.051 $\pm$ 0.034} & \textbf{0.269 $\pm$ 0.217} \\
0.8 & 23.992 $\pm$ 14.060 & 79.531 $\pm$ 42.974 & \textbf{0.108 $\pm$ 0.070} & \textbf{0.286 $\pm$ 0.243} \\
0.9 & 73.202 $\pm$ 19.451 & 147.251 $\pm$ 8.506 & \textbf{0.077 $\pm$ 0.026} & \textbf{0.170 $\pm$ 0.087} \\
\bottomrule
\end{tabular}
\caption*{Separable polynomial kernel}
\end{subtable}
\end{tabular}

\caption{Relative errors (mean $\pm$ std) for $q$ in the Mass-Spring system.}
\label{tab:ms-q}
\end{table}

\subsubsection{Error plots}
Figure~\ref{fig:for_error_ms_all} shows the forecasting relative errors for $q$ and $p$, together with shaded confidence bands representing one standard deviation.  
For both $q$ and $p$, the 1-step method maintains low and stable relative error across all sparsity levels.  
In contrast, the 2-step method exhibits rapidly growing errors beyond a sparsity factor of $0.6$, 
indicating its sensitivity to data sparsity.

\begin{figure}[htbp]
\centering

\begin{subfigure}{\textwidth}
    \centering
    \begin{subfigure}{0.36\textwidth}
        \includegraphics[width=\linewidth]{ 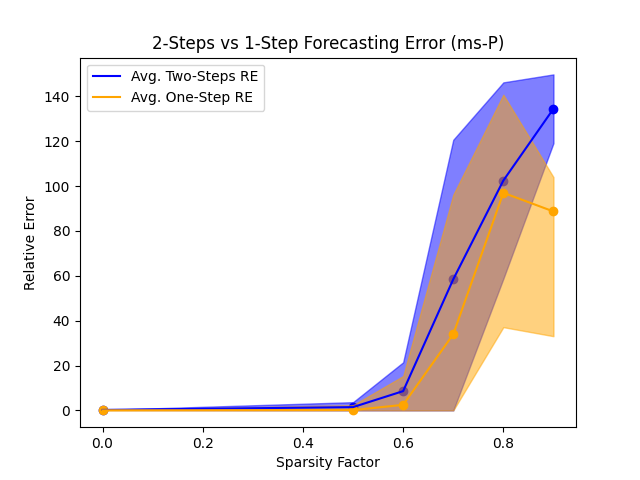}
        \caption*{$p$-error, Gaussian}
    \end{subfigure}
    \begin{subfigure}{0.36\textwidth}
        \includegraphics[width=\linewidth]{ 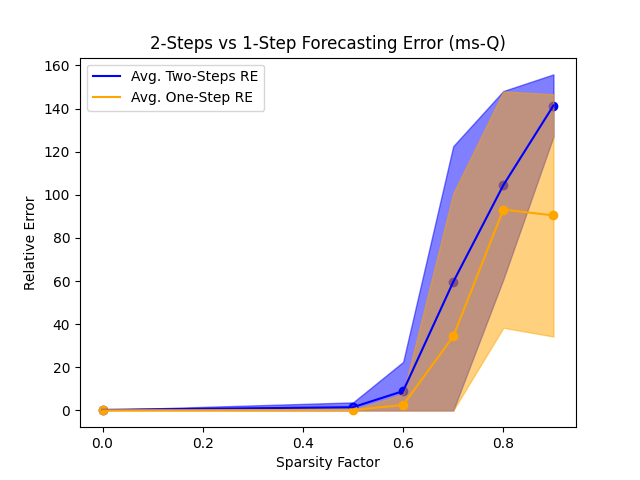}
        \caption*{$q$-error, Gaussian}
    \end{subfigure}
\end{subfigure}

\begin{subfigure}{\textwidth}
    \centering
    \begin{subfigure}{0.36\textwidth}
        \includegraphics[width=\linewidth]{ 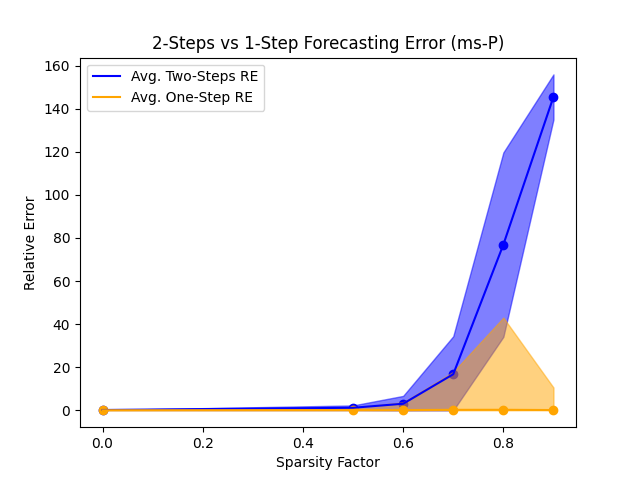}
        \caption*{$p$-error, Separable Polynomial}
    \end{subfigure}
    \begin{subfigure}{0.36\textwidth}
        \includegraphics[width=\linewidth]{ 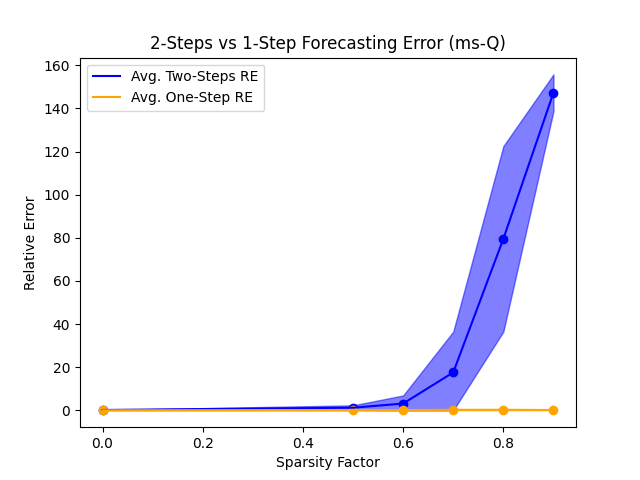}
        \caption*{$q$-error, Separable Polynomial}
    \end{subfigure}
\end{subfigure}

\caption{Forecasting relative error  for the Mass-Spring System.}
\label{fig:for_error_ms_all}
\end{figure}

\subsection{Two-Mass-Three-Spring System}

We now consider a two-dimensional mechanical system composed of two unit masses connected by three springs. Its Hamiltonian is given by
\begin{align*}
    H(q,p) = \frac{1}{2}(p_1^2 + p_2^2) + \frac{1}{2}(q_1^2 + q_2^2 + (q_2 - q_1)^2).
\end{align*}
The initial condition is fixed at $y(t_0) = ( 0.2, -0.1, 0.1, -0.1)$.  
The system exhibits coupled oscillatory behavior and is a standard multi-variable benchmark. 

\subsubsection{Trajectory recovery plots}
Similar to the previous system, in Figure~\ref{fig:recovery_m2s3_q1} we observe good recovery at sparsity $0.5$ from both methods, but in the $0.8$ sparsity regime, the 1-step method outperforms the 2-step method, especially in the forecasting region.  
Since the Hamiltonian is symmetric with respect to $q_1$ and $q_2$, and it is an order-two polynomial in $q_1, q_2, p_1, p_2$, we only report the trajectory plots for $q_1$ here and provide the rest in Section~\ref{SM: plots}. 

\begin{figure}[htbp]
\centering
\begin{subfigure}{0.36\textwidth}
\includegraphics[width=\linewidth]{ 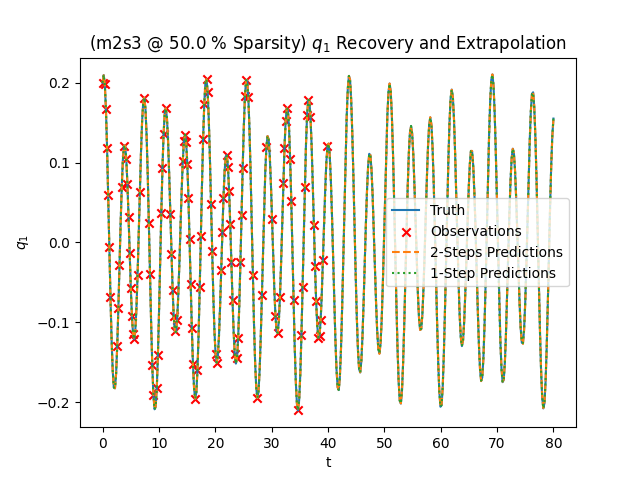}
\caption*{Gaussian 0.5}
\end{subfigure}
\begin{subfigure}{0.36\textwidth}
\includegraphics[width=\linewidth]{ 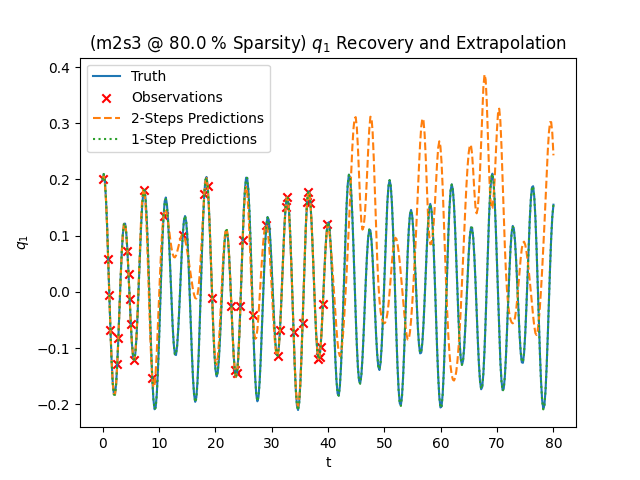}
\caption*{Gaussian 0.8}
\end{subfigure}
\begin{subfigure}{0.36\textwidth}
\includegraphics[width=\linewidth]{ 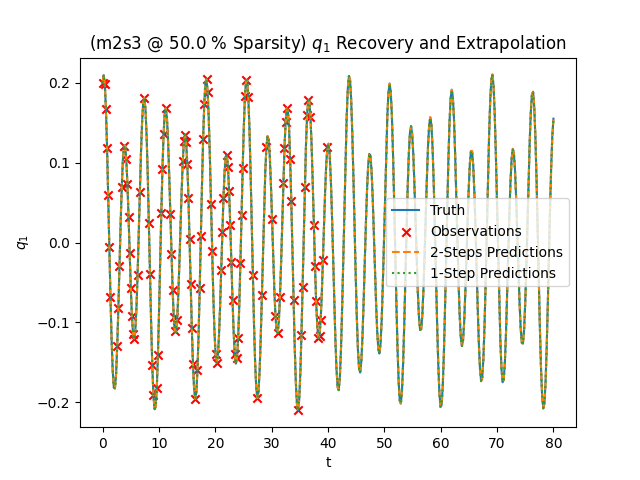}
\caption*{Poly 0.5}
\end{subfigure}
\begin{subfigure}{0.36\textwidth}
\includegraphics[width=\linewidth]{ 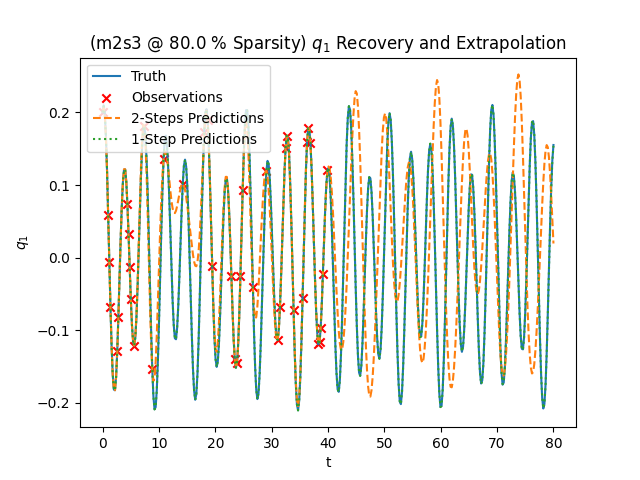}
\caption*{Poly 0.8}
\end{subfigure}
\caption{Recovery of $q_1$ in the Two-Mass-Three-Spring system.}
\label{fig:recovery_m2s3_q1}
\end{figure}

\subsubsection{Error tables}

\begin{table}[ht]
\centering
\setlength{\tabcolsep}{12pt} 
\renewcommand{\arraystretch}{1.2} 
\scriptsize

\begin{tabular}{c} 
\begin{subtable}{\textwidth}
\centering
\begin{tabular}{lcccc}
\toprule
\multirow{2}{*}{\textbf{Sparsity}} & \multicolumn{2}{c}{\textbf{Two-Steps Method}} & \multicolumn{2}{c}{\textbf{One-Step Method}} \\
\cmidrule(lr){2-3}\cmidrule(lr){4-5}
& Interpolation   & Extrapolation   & Interpolation   & Extrapolation   \\
\midrule
0.0 & \textbf{0.000 $\pm$ 0.000} & 0.203 $\pm$ 0.000 & 0.168 $\pm$ 0.000 & \textbf{0.224 $\pm$ 0.000} \\
0.5 & 0.360 $\pm$ 0.285 & 1.451 $\pm$ 1.290 & \textbf{0.072 $\pm$ 0.069} & \textbf{0.214 $\pm$ 0.082} \\
0.6 & 0.920 $\pm$ 0.796 & 5.396 $\pm$ 7.568 & \textbf{0.037 $\pm$ 0.010} & \textbf{0.194 $\pm$ 0.081} \\
0.7 & 15.480 $\pm$ 12.938 & 54.705 $\pm$ 57.873 & \textbf{1.287 $\pm$ 2.688} & \textbf{5.146 $\pm$ 11.018} \\
0.8 & 42.990 $\pm$ 20.771 & 122.531 $\pm$ 43.043 & \textbf{7.603 $\pm$ 9.835} & \textbf{17.316 $\pm$ 23.083} \\
0.9 & 103.826 $\pm$ 22.026 & 130.073 $\pm$ 18.070 & \textbf{58.210 $\pm$ 32.230} & \textbf{79.855 $\pm$ 25.081} \\
\bottomrule
\end{tabular}
\caption*{Gaussian kernel}
\end{subtable} \\[1.5ex]

\begin{subtable}{\textwidth}
\centering
\begin{tabular}{lcccc}
\toprule
\multirow{2}{*}{\textbf{Sparsity}} & \multicolumn{2}{c}{\textbf{Two-Steps Method}} & \multicolumn{2}{c}{\textbf{One-Step Method}} \\
\cmidrule(lr){2-3}\cmidrule(lr){4-5}
& Interpolation   & Extrapolation   & Interpolation   & Extrapolation   \\
\midrule
0.0 & \textbf{0.000 $\pm$ 0.000} & 0.199 $\pm$ 0.000 & 0.011 $\pm$ 0.000 & \textbf{0.157 $\pm$ 0.000} \\
0.5 & 0.360 $\pm$ 0.285 & 1.545 $\pm$ 1.548 & \textbf{0.030 $\pm$ 0.016} & \textbf{0.218 $\pm$ 0.126} \\
0.6 & 0.920 $\pm$ 0.796 & 5.811 $\pm$ 8.604 & \textbf{0.061 $\pm$ 0.052} & \textbf{0.300 $\pm$ 0.211} \\
0.7 & 15.480 $\pm$ 12.938 & 71.540 $\pm$ 57.892 & \textbf{0.104 $\pm$ 0.053} & \textbf{0.356 $\pm$ 0.181} \\
0.8 & 42.990 $\pm$ 20.771 & 134.576 $\pm$ 35.819 & \textbf{0.156 $\pm$ 0.072} & \textbf{0.401 $\pm$ 0.231} \\
\bottomrule
\end{tabular}
\caption*{Separable polynomial kernel}
\end{subtable}
\end{tabular}

\caption{Relative errors (mean $\pm$ std) for $q$ in the Two-Mass-Three-Spring system.}
\label{tab:m2s3-q}
\end{table}

As in the previous example (Table~\ref{tab:m2s3-q}), especially in the forecasting regime, increasing sparsity shows that the 1-step method outperforms the 2-step method.  
Due to numerical instabilities at sparsity $0.9$ with the polynomial kernel (specifically matrix inversion issues), the tables only report up to sparsity $0.8$ for this kernel. 

\subsubsection{Error plots}

The forecasting error plots in Figure~\ref{fig:for_error_m2s3_all} show a trend similar to the previous system: starting from sparsity $0.6$, the 2-step method’s error grows significantly compared to the 1-step method.

\begin{figure}[htbp]
\centering
\begin{subfigure}{\textwidth}
    \centering
    \begin{subfigure}{0.36\textwidth}
        \includegraphics[width=\linewidth]{ 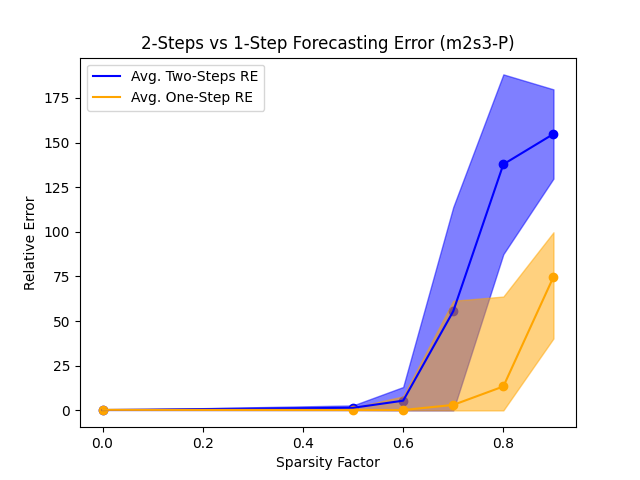}
        \caption*{$p$-error, Gaussian}
    \end{subfigure}
    \begin{subfigure}{0.36\textwidth}
        \includegraphics[width=\linewidth]{ 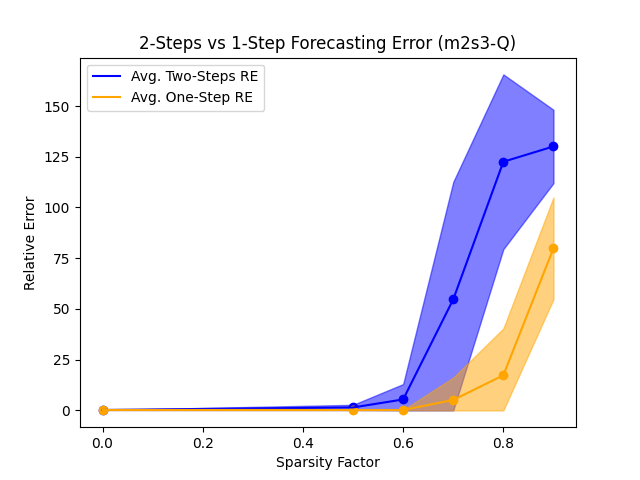}
        \caption*{$q$-error, Gaussian}
    \end{subfigure}
\end{subfigure}

\begin{subfigure}{\textwidth}
    \centering
    \begin{subfigure}{0.36\textwidth}
        \includegraphics[width=\linewidth]{ 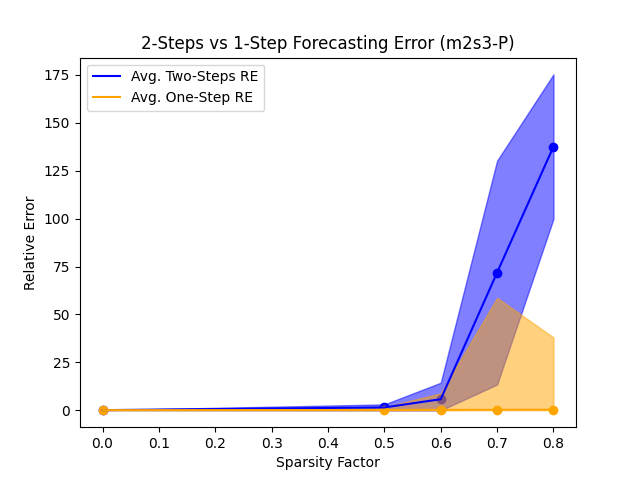}
        \caption*{$p$-error, Separable Polynomial}
    \end{subfigure}
    \begin{subfigure}{0.36\textwidth}
        \includegraphics[width=\linewidth]{ 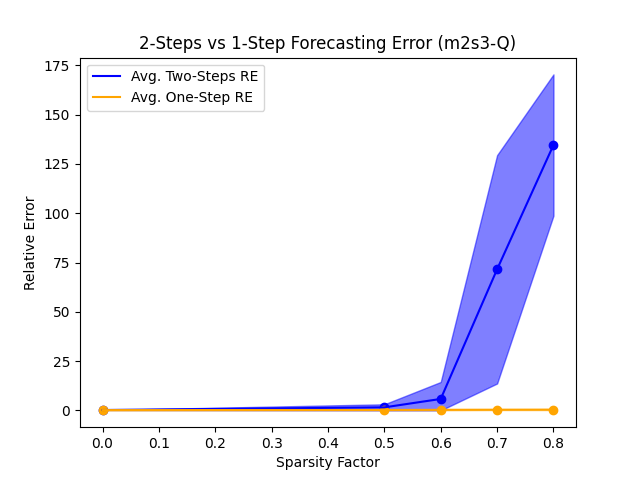}
        \caption*{$q$-error, Separable Polynomial}
    \end{subfigure}
\end{subfigure}

\caption{Forecasting relative errors for the Two-Mass-Three-Spring system.}
\label{fig:for_error_m2s3_all}
\end{figure}

\subsection{Hénon-Heiles System}

For our next example, we consider the Hénon-Heiles system, a well-known two-dimensional nonlinear chaotic Hamiltonian system. Its Hamiltonian is given by
\begin{align*}
H(q, p) = \frac{1}{2}(p_1^2 + p_2^2) + \frac{1}{2}(q_1^2 + q_2^2) + q_1^2 q_2 - \frac{1}{3}q_2^3.  
\end{align*}
We fix the initial condition as $y(t_0) = (0.2, -0.1, 0.1, -0.1)$.  
The 1-step method is able to recover the dynamics and forecast future behavior with reasonable accuracy, even in the presence of chaotic features. Naturally, the prediction horizon is limited due to the system’s sensitivity to initial conditions.

\subsubsection{Trajectory recovery plots}
Figures~\ref{fig:recovery_hh_q1} and~\ref{fig:recovery_hh_q2} display the recovered trajectories for $q_1$ and $q_2$, respectively, under sparsity factors $0.5$ and $0.8$. We show these two as $q_2$ is of order 3 and report the plots for the rest of variables in \ref{SM: plots}.
At moderate sparsity ($0.5$), both methods perform reasonably well.  
At higher sparsity ($0.8$), the 1-step method continues to follow the true dynamics more closely, while the 2-step method deviates significantly, particularly in the forecasting region.

\begin{figure}[htbp]
\centering
\begin{subfigure}{0.36\textwidth}
\includegraphics[width=\linewidth]{ 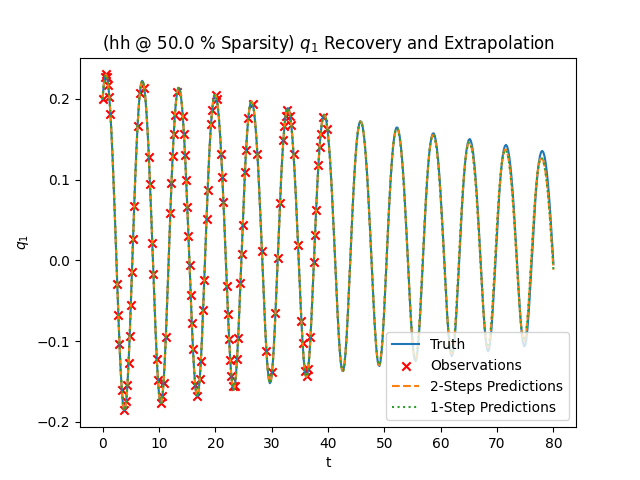}
\caption*{Gaussian 0.5}
\end{subfigure}
\begin{subfigure}{0.36\textwidth}
\includegraphics[width=\linewidth]{ 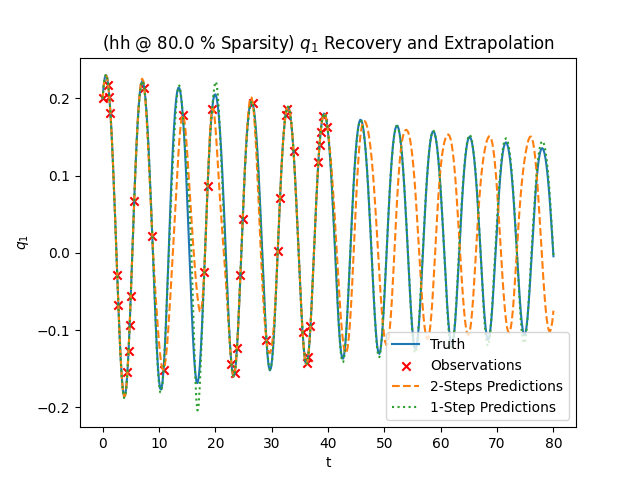}
\caption*{Gaussian 0.8}
\end{subfigure}
\begin{subfigure}{0.36\textwidth}
\includegraphics[width=\linewidth]{ 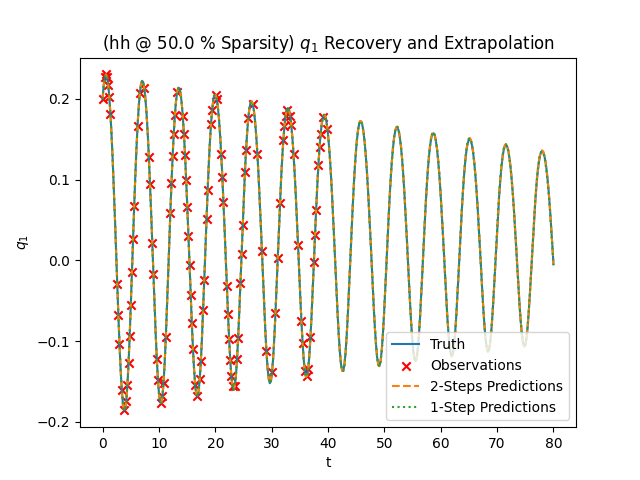}
\caption*{Poly 0.5}
\end{subfigure}
\begin{subfigure}{0.36\textwidth}
\includegraphics[width=\linewidth]{ 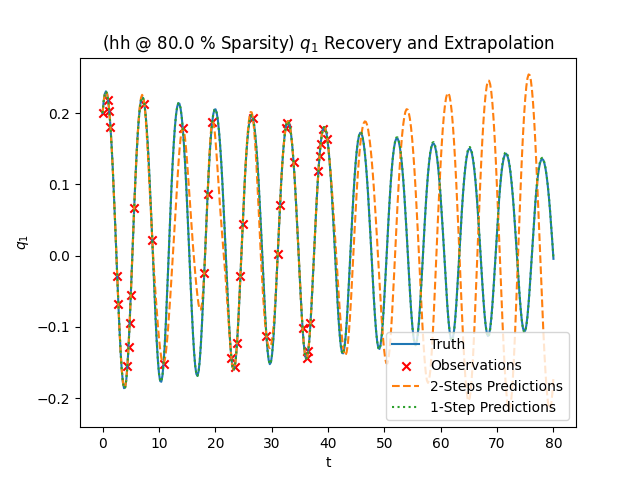}
\caption*{Poly 0.8}
\end{subfigure}
\caption{Recovery of $q_1$ in the Hénon-Heiles system.}
\label{fig:recovery_hh_q1}
\end{figure}

\begin{figure}[htbp]
\centering
\begin{subfigure}{0.36\textwidth}
\includegraphics[width=\linewidth]{ 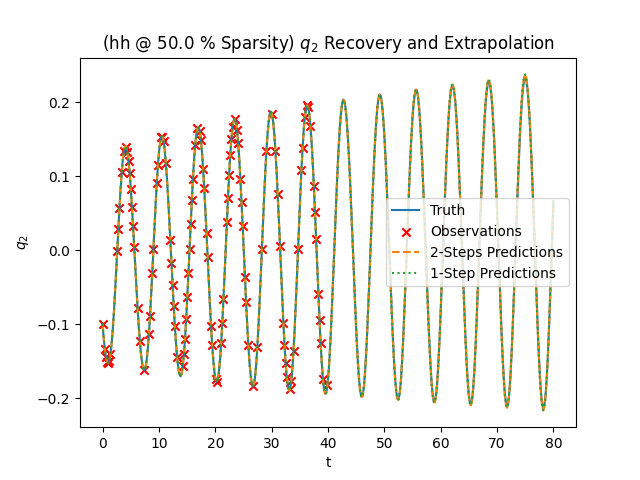}
\caption*{Gaussian 0.5}
\end{subfigure}
\begin{subfigure}{0.36\textwidth}
\includegraphics[width=\linewidth]{ 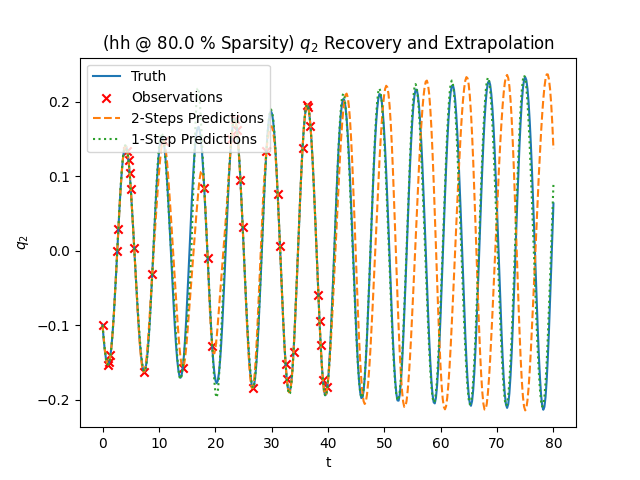}
\caption*{Gaussian 0.8}
\end{subfigure}
\begin{subfigure}{0.36\textwidth}
\includegraphics[width=\linewidth]{ 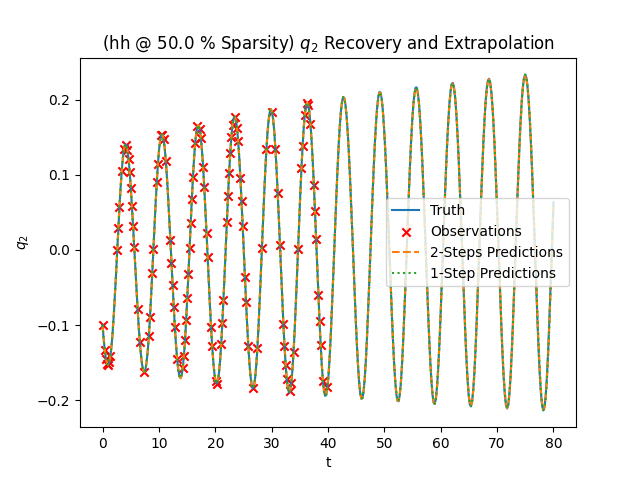}
\caption*{Poly 0.5}
\end{subfigure}
\begin{subfigure}{0.36\textwidth}
\includegraphics[width=\linewidth]{ 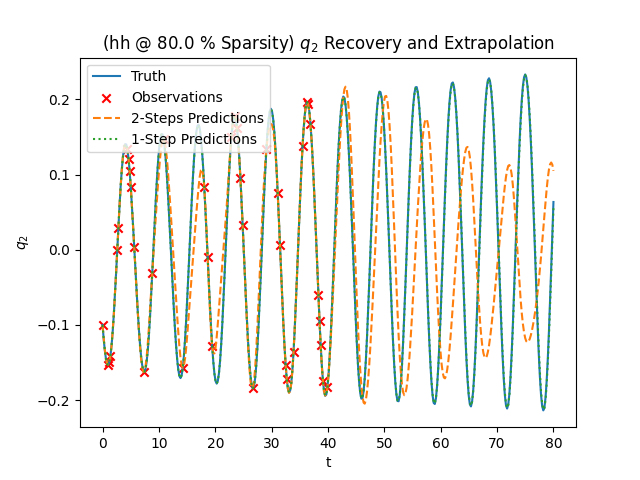}
\caption*{Poly 0.8}
\end{subfigure}
\caption{Recovery of $q_2$ in the Hénon-Heiles system.}
\label{fig:recovery_hh_q2}
\end{figure}

\subsubsection{Error tables}
Table~\ref{tab:hh-q} reports the relative errors for $q$.  
Interpolation with the 1-step method is significantly more accurate than with the 2-step method, especially for sparsities $0.5$-$0.7$.  
Extrapolation shows the same trend: the 1-step method consistently yields smaller errors, though both methods exhibit large errors at high sparsity.  
Polynomial kernels mitigate extrapolation blow-up more effectively than Gaussian kernels.

\begin{table}[ht]
\centering
\setlength{\tabcolsep}{12pt} 
\renewcommand{\arraystretch}{1.2} 
\scriptsize

\begin{tabular}{c} 
\begin{subtable}{\textwidth}
\centering
\begin{tabular}{lcccc}
\toprule
\multirow{2}{*}{\textbf{Sparsity}} & \multicolumn{2}{c}{\textbf{Two-Steps Method}} & \multicolumn{2}{c}{\textbf{One-Step Method}} \\
\cmidrule(lr){2-3}\cmidrule(lr){4-5}
& Interpolation   & Extrapolation   & Interpolation   & Extrapolation   \\
\midrule
0.0 & \textbf{0.000 $\pm$ 0.000} & 2.977 $\pm$ 0.000 & 0.002 $\pm$ 0.000 & \textbf{2.988 $\pm$ 0.000} \\
0.5 & 0.103 $\pm$ 0.097 & 2.842 $\pm$ 0.371 & \textbf{0.022 $\pm$ 0.015} & 3.139 $\pm$ 0.158 \\
0.6 & 0.303 $\pm$ 0.197 & 3.600 $\pm$ 1.511 & \textbf{0.040 $\pm$ 0.022} & \textbf{3.141 $\pm$ 0.333} \\
0.7 & 5.123 $\pm$ 4.006 & 17.517 $\pm$ 14.945 & \textbf{1.486 $\pm$ 3.676} & \textbf{8.332 $\pm$ 10.159} \\
0.8 & 22.769 $\pm$ 14.063 & 64.279 $\pm$ 40.605 & \textbf{7.272 $\pm$ 8.807} & \textbf{21.508 $\pm$ 27.138} \\
0.9 & 69.056 $\pm$ 20.458 & 135.651 $\pm$ 13.299 & \textbf{28.934 $\pm$ 20.882} & \textbf{46.749 $\pm$ 29.029} \\
\bottomrule
\end{tabular}
\caption*{Gaussian kernel}
\end{subtable}\\[1.5ex]

\begin{subtable}{\textwidth}
\centering
\begin{tabular}{lcccc}
\toprule
\multirow{2}{*}{\textbf{Sparsity}} & \multicolumn{2}{c}{\textbf{Two-Steps Method}} & \multicolumn{2}{c}{\textbf{One-Step Method}} \\
\cmidrule(lr){2-3}\cmidrule(lr){4-5}
& Interpolation   & Extrapolation   & Interpolation   & Extrapolation   \\
\midrule
0.0 & \textbf{0.000 $\pm$ 0.000} & 1.145 $\pm$ 0.000 & 0.002 $\pm$ 0.000 & \textbf{1.145 $\pm$ 0.000} \\
0.5 & 0.103 $\pm$ 0.097 & 1.790 $\pm$ 0.747 & \textbf{0.013 $\pm$ 0.007} & \textbf{1.252 $\pm$ 0.108} \\
0.6 & 0.303 $\pm$ 0.197 & 3.190 $\pm$ 2.567 & \textbf{0.028 $\pm$ 0.024} & \textbf{1.197 $\pm$ 0.057} \\
0.7 & 5.123 $\pm$ 4.006 & 17.339 $\pm$ 17.574 & \textbf{0.036 $\pm$ 0.022} & \textbf{1.360 $\pm$ 0.199} \\
0.8 & 22.769 $\pm$ 14.063 & 74.536 $\pm$ 43.707 & \textbf{0.059 $\pm$ 0.031} & \textbf{1.434 $\pm$ 0.217} \\
0.9 & 69.056 $\pm$ 20.458 & 144.810 $\pm$ 12.041 & \textbf{0.117 $\pm$ 0.095} & \textbf{1.856 $\pm$ 0.528} \\
\bottomrule
\end{tabular}
\caption*{Separable polynomial kernel}
\end{subtable}
\end{tabular}

\caption{Relative errors (mean $\pm$ std) for $q$ in the Hénon-Heiles system.}
\label{tab:hh-q}
\end{table}

\subsubsection{Error plots}
Figure~\ref{fig:for_error_hh_all} shows the forecasting relative errors for both $q$ and $p$.  
Once again, the 1-step method maintains lower errors across sparsity levels, while the 2-step method degrades rapidly beyond sparsity $0.6$.  
The separable polynomial kernel provides more stable performance in the extrapolation regime compared to the Gaussian kernel.

\begin{figure}[htbp!]
\centering

\begin{subfigure}{\textwidth}
    \centering
    \begin{subfigure}{0.36\textwidth}
        \includegraphics[width=\linewidth]{ 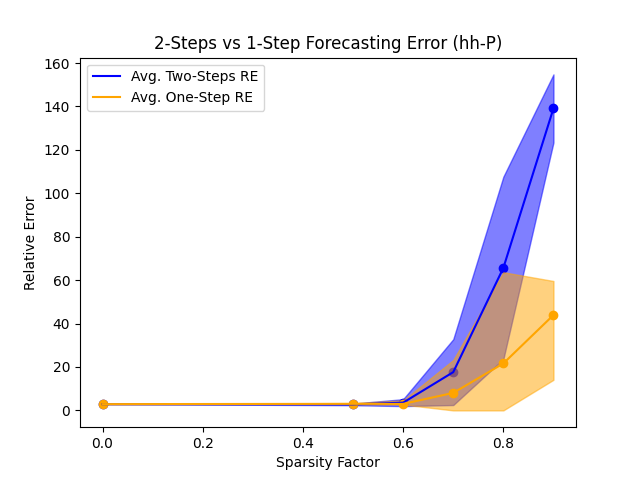}
        \caption*{$p$-error, Gaussian}
    \end{subfigure}
    \begin{subfigure}{0.36\textwidth}
        \includegraphics[width=\linewidth]{ 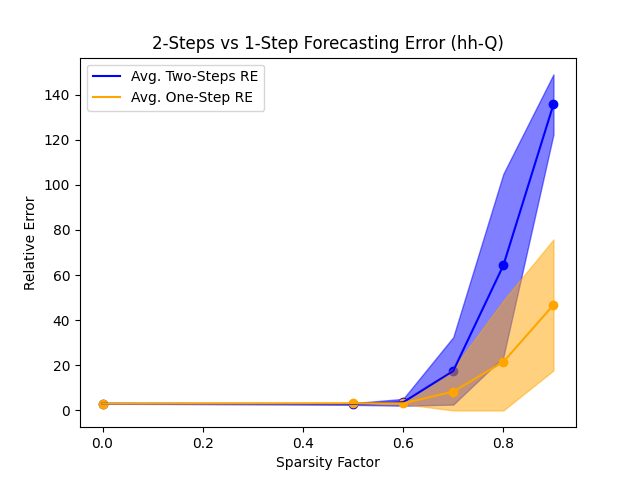}
        \caption*{$q$-error, Gaussian}
    \end{subfigure}
\end{subfigure}

\begin{subfigure}{\textwidth}
    \centering
    \begin{subfigure}{0.36\textwidth}
        \includegraphics[width=\linewidth]{ 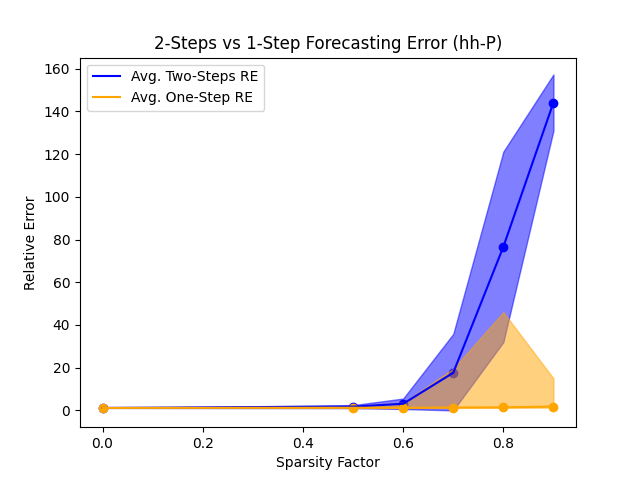}
        \caption*{$p$-error, Separable Polynomial}
    \end{subfigure}
    \begin{subfigure}{0.36\textwidth}
        \includegraphics[width=\linewidth]{ 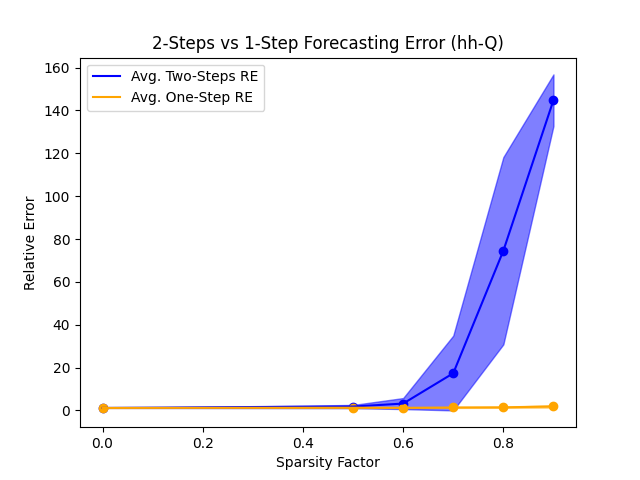}
        \caption*{$q$-error, Separable Polynomial}
    \end{subfigure}
\end{subfigure}

\caption{Forecasting relative errors for the Hénon-Heiles System.}
\label{fig:for_error_hh_all}
\end{figure}

\subsection{Nonlinear Pendulum}

For our next example, we consider a one-dimensional nonlinear system governed by the Hamiltonian
\begin{align*}
    H(q,p) = \frac{1}{2}p^2 - \cos(q),
\end{align*}
with initial condition $y(t_0) = (0.95\pi, 0.0)$.  
This system exhibits nonlinear oscillations and highlights the capability of our method to handle non-polynomial Hamiltonians. 

\subsubsection{Trajectory recovery plots}
Unlike the previous examples, here the recovery behavior differs between $q$ and $p$. Therefore we report both variables in the main body (Figures~\ref{fig:recovery_np_q} and~\ref{fig:recovery_np_p}).  
Despite the difference, once again, at moderate sparsity ($0.5$), both the 1-step and 2-step methods provide good recovery for $q$ and $p$; at higher sparsity ($0.8$), however, the 1-step method remains stable and continues to track the true dynamics more accurately, while the 2-step method shows larger deviations, particularly in the forecasting region.  

\begin{figure}[htbp]
    \centering
    \begin{subfigure}{0.36\textwidth}
        \includegraphics[width=\linewidth]{ 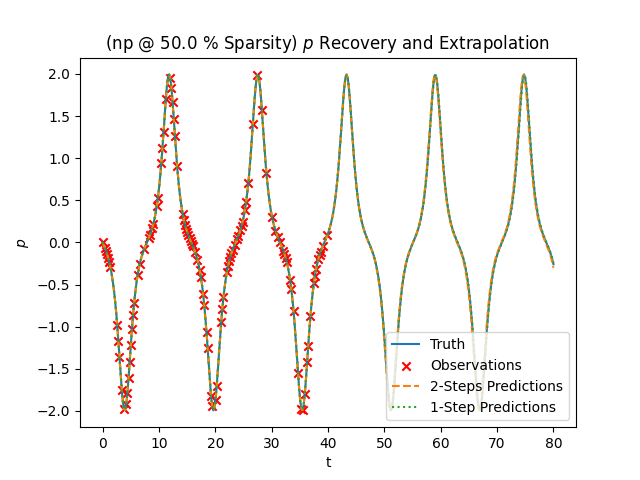}
        \caption*{Gaussian 0.5}
    \end{subfigure}
    \begin{subfigure}{0.36\textwidth}
        \includegraphics[width=\linewidth]{ 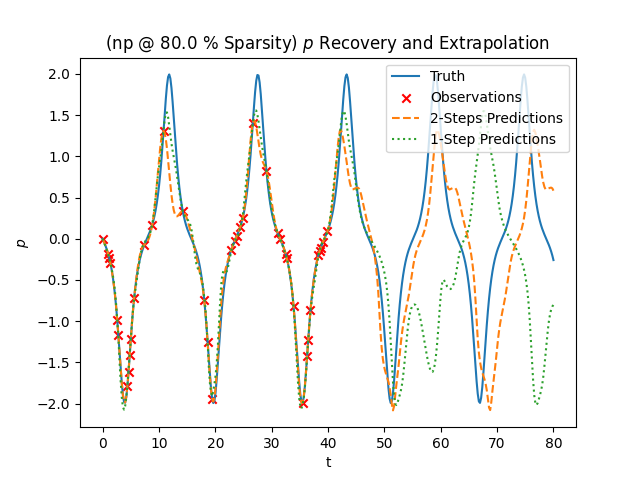}
        \caption*{Gaussian 0.8}
    \end{subfigure}
    \begin{subfigure}{0.36\textwidth}
        \includegraphics[width=\linewidth]{ 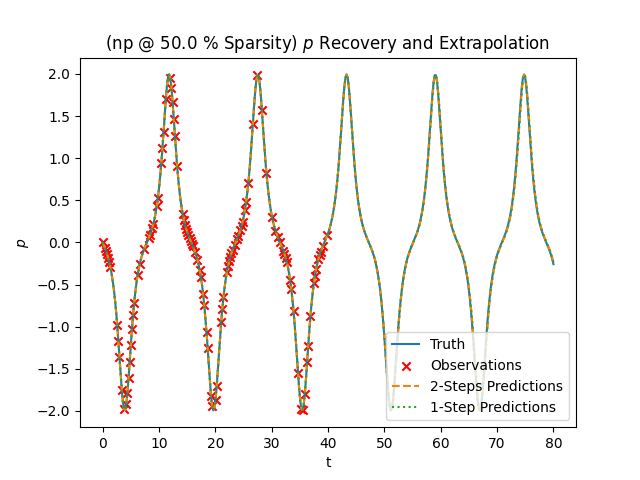}
        \caption*{Additive 0.5}
    \end{subfigure}
    \begin{subfigure}{0.36\textwidth}
        \includegraphics[width=\linewidth]{ 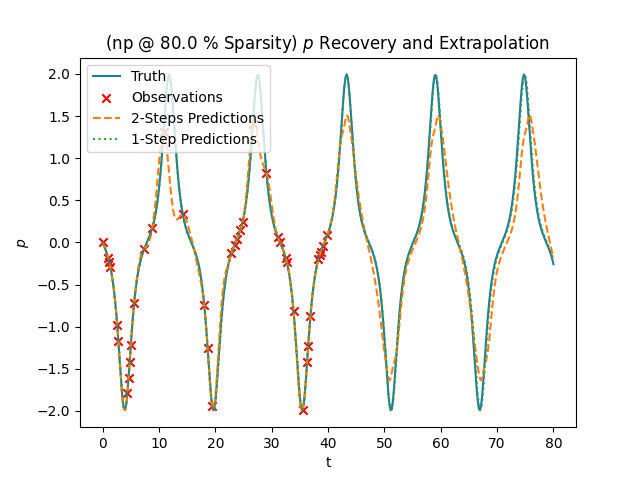}
        \caption*{Additive 0.8}
    \end{subfigure}
    \caption{Recovery of $p$ in the Nonlinear Pendulum system.}
    \label{fig:recovery_np_p}
\end{figure}

\begin{figure}[htbp]
    \centering
    \begin{subfigure}{0.36\textwidth}
        \includegraphics[width=\linewidth]{ 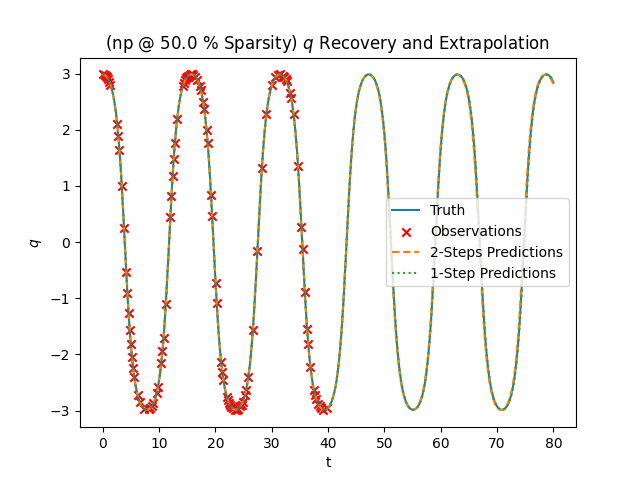}
        \caption*{Gaussian 0.5}
    \end{subfigure}
    \begin{subfigure}{0.36\textwidth}
        \includegraphics[width=\linewidth]{ 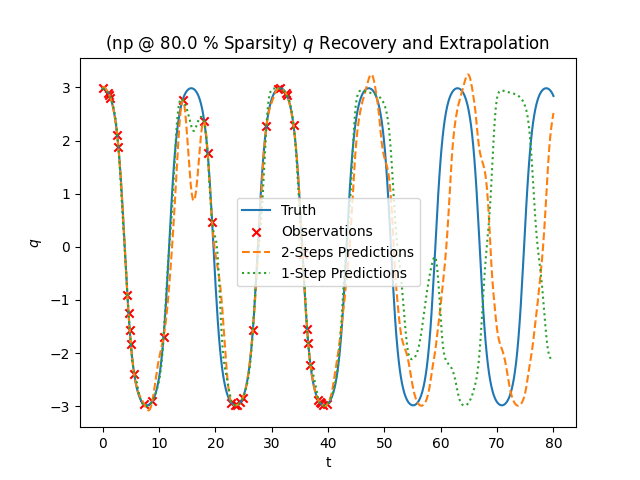}
        \caption*{Gaussian 0.8}
    \end{subfigure}
    
    \begin{subfigure}{0.36\textwidth}
        \includegraphics[width=\linewidth]{ 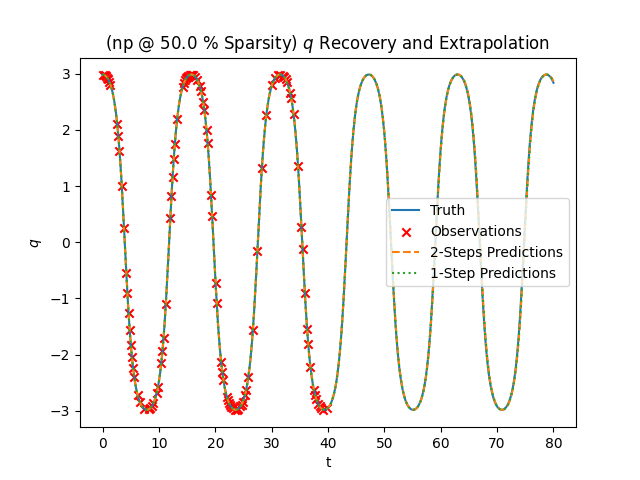}
        \caption*{Additive 0.5}
    \end{subfigure}
    \begin{subfigure}{0.36\textwidth}
        \includegraphics[width=\linewidth]{ 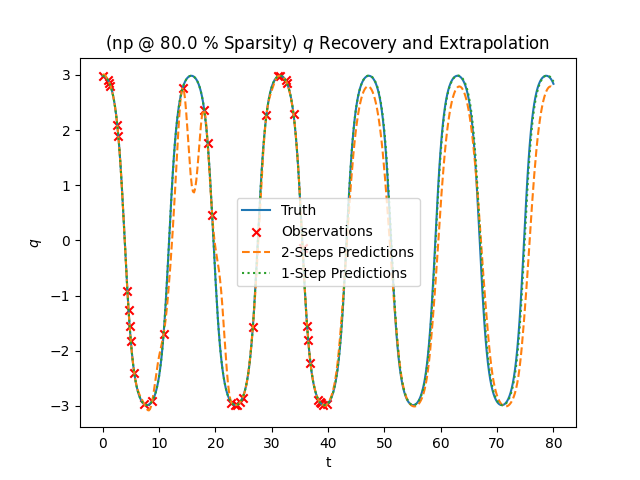}
        \caption*{Additive 0.8}
    \end{subfigure}

    \caption{Recovery of $q$ in the Nonlinear Pendulum system.}
    \label{fig:recovery_np_q}
\end{figure}

\subsubsection{Error tables}
Tables~\ref{tab:np-q} and~\ref{tab:np-p} report the relative errors for $q$ and $p$, respectively.  
In both interpolation and extrapolation, the 1-step method consistently outperforms the 2-step method.  
For the Gaussian kernel, extrapolation errors increase sharply at high sparsity.  
In contrast, the additive polynomial and Gaussian kernel provides more stable performance in the extrapolation regime.  

\begin{table}[ht]
\centering
\setlength{\tabcolsep}{12pt} 
\renewcommand{\arraystretch}{1.2} 
\scriptsize

\begin{tabular}{c} 

\begin{subtable}{\textwidth}
\centering
\begin{tabular}{lcccc}
\toprule
\multirow{2}{*}{\textbf{Sparsity}} & \multicolumn{2}{c}{\textbf{Two-Steps Method}} & \multicolumn{2}{c}{\textbf{One-Step Method}} \\
\cmidrule(lr){2-3}\cmidrule(lr){4-5}
& Interpolation   & Extrapolation   & Interpolation   & Extrapolation   \\
\midrule
0.0 & \textbf{0.000 $\pm$ 0.000} & 0.234 $\pm$ 0.000 & 0.008 $\pm$ 0.000 & \textbf{1.872 $\pm$ 0.000} \\
0.5 & 0.182 $\pm$ 0.148 & 2.549 $\pm$ 3.069 & \textbf{0.035 $\pm$ 0.046} & \textbf{1.136 $\pm$ 0.980} \\
0.6 & 0.599 $\pm$ 0.575 & 5.061 $\pm$ 5.205 & \textbf{0.085 $\pm$ 0.120} & \textbf{2.062 $\pm$ 1.948} \\
0.7 & 5.481 $\pm$ 4.974 & 45.986 $\pm$ 48.646 & \textbf{1.013 $\pm$ 1.528} & \textbf{14.267 $\pm$ 38.004} \\
0.8 & 15.302 $\pm$ 10.281 & 102.954 $\pm$ 47.515 & \textbf{9.592 $\pm$ 11.503} & \textbf{74.506 $\pm$ 67.847} \\
0.9 & 59.719 $\pm$ 16.933 & 127.130 $\pm$ 22.565 & \textbf{33.419 $\pm$ 16.313} & \textbf{97.625 $\pm$ 54.474} \\
\bottomrule
\end{tabular}
\caption*{Gaussian kernel}
\end{subtable}\\[1.5ex]

\begin{subtable}{\textwidth}
\centering
\begin{tabular}{lcccc}
\toprule
\multirow{2}{*}{\textbf{Sparsity}} & \multicolumn{2}{c}{\textbf{Two-Steps Method}} & \multicolumn{2}{c}{\textbf{One-Step Method}} \\
\cmidrule(lr){2-3}\cmidrule(lr){4-5}
& Interpolation   & Extrapolation   & Interpolation   & Extrapolation   \\
\midrule
0.0 & \textbf{0.000 $\pm$ 0.000} & 0.228 $\pm$ 0.000 & 0.011 $\pm$ 0.000 & \textbf{1.890 $\pm$ 0.000} \\
0.5 & 0.182 $\pm$ 0.148 & 8.016 $\pm$ 7.455 & \textbf{0.013 $\pm$ 0.003} & \textbf{1.046 $\pm$ 0.529} \\
0.6 & 0.599 $\pm$ 0.575 & 32.859 $\pm$ 20.962 & \textbf{0.014 $\pm$ 0.004} & \textbf{1.337 $\pm$ 0.827} \\
0.7 & 5.481 $\pm$ 4.974 & 81.800 $\pm$ 46.374 & \textbf{0.025 $\pm$ 0.024} & \textbf{2.192 $\pm$ 2.705} \\
0.8 & 15.302 $\pm$ 10.281 & 80.487 $\pm$ 41.416 & \textbf{0.051 $\pm$ 0.069} & \textbf{3.056 $\pm$ 2.150} \\
0.9 & 59.719 $\pm$ 16.933 & 122.773 $\pm$ 29.940 & \textbf{25.260 $\pm$ 15.747} & \textbf{102.912 $\pm$ 53.001} \\
\bottomrule
\end{tabular}
\caption*{Additive polynomial + Gaussian kernel}
\end{subtable}
\end{tabular}

\caption{Relative errors (mean $\pm$ std) for $q$ in the Nonlinear Pendulum system.}
\label{tab:np-q}
\end{table}

\begin{table}[ht]
\centering
\setlength{\tabcolsep}{12pt} 
\renewcommand{\arraystretch}{1.2} 
\scriptsize

\begin{tabular}{c} 
\begin{subtable}{\textwidth}
\centering
\begin{tabular}{lcccc}
\toprule
\multirow{2}{*}{\textbf{Sparsity}} & \multicolumn{2}{c}{\textbf{Two-Steps Method}} & \multicolumn{2}{c}{\textbf{One-Step Method}} \\
\cmidrule(lr){2-3}\cmidrule(lr){4-5}
& Interpolation & Extrapolation & Interpolation   & Extrapolation \\
\midrule
0.0 & \textbf{0.000 $\pm$ 0.000} & 0.336 $\pm$ 0.000 & 0.014 $\pm$ 0.000 & \textbf{2.655 $\pm$ 0.000} \\
0.5 & 0.267 $\pm$ 0.188 & 3.583 $\pm$ 4.310 & \textbf{0.101 $\pm$ 0.091} & \textbf{1.610 $\pm$ 1.390} \\
0.6 & 0.824 $\pm$ 0.627 & 7.179 $\pm$ 7.248 & \textbf{0.144 $\pm$ 0.120} & \textbf{2.922 $\pm$ 2.757} \\
0.7 & 6.008 $\pm$ 5.373 & 48.595 $\pm$ 46.310 & \textbf{1.412 $\pm$ 2.069} & \textbf{14.468 $\pm$ 36.524} \\
0.8 & 23.172 $\pm$ 22.818 & 100.743 $\pm$ 43.798 & \textbf{17.445 $\pm$ 17.640} & \textbf{71.424 $\pm$ 60.946} \\
0.9 & 65.326 $\pm$ 36.965 & 239.710 $\pm$ 329.837 & \textbf{46.574 $\pm$ 18.174} & \textbf{105.368 $\pm$ 40.820} \\
\bottomrule
\end{tabular}
\caption*{Gaussian kernel}
\end{subtable} \\[1.5ex]

\begin{subtable}{\textwidth}
\centering
\begin{tabular}{lcccc}
\toprule
\multirow{2}{*}{\textbf{Sparsity}} & \multicolumn{2}{c}{\textbf{Two-Steps Method}} & \multicolumn{2}{c}{\textbf{One-Step Method}} \\
\cmidrule(lr){2-3}\cmidrule(lr){4-5}
& Interpolation & Extrapolation & Interpolation   & Extrapolation \\
\midrule
0.0 & \textbf{0.000 $\pm$ 0.000} & 0.324 $\pm$ 0.000 & 0.024 $\pm$ 0.000 & \textbf{2.681 $\pm$ 0.000} \\
0.5 & 0.267 $\pm$ 0.188 & 11.231 $\pm$ 10.373 & \textbf{0.049 $\pm$ 0.017} & \textbf{1.484 $\pm$ 0.752} \\
0.6 & 0.824 $\pm$ 0.627 & 42.634 $\pm$ 25.174 & \textbf{0.065 $\pm$ 0.020} & \textbf{1.898 $\pm$ 1.174} \\
0.7 & 6.008 $\pm$ 5.373 & 85.578 $\pm$ 41.625 & \textbf{0.093 $\pm$ 0.046} & \textbf{3.108 $\pm$ 3.831} \\
0.8 & 23.172 $\pm$ 22.818 & 88.899 $\pm$ 39.267 & \textbf{0.182 $\pm$ 0.171} & \textbf{4.336 $\pm$ 3.048} \\
0.9 & 65.326 $\pm$ 36.965 & 140.867 $\pm$ 32.271 & \textbf{37.796 $\pm$ 23.027} & \textbf{102.901 $\pm$ 43.392} \\
\bottomrule
\end{tabular}
\caption*{Additive polynomial + Gaussian kernel}
\end{subtable}

\end{tabular}

\caption{Relative errors (mean $\pm$ std) for $p$ in the Nonlinear Pendulum system.}
\label{tab:np-p}
\end{table}

\subsubsection{Error plots}
Figure~\ref{fig:for_error_np_all} shows the forecasting relative errors for $q$ and $p$. Here again, the 1-step method maintains low errors across all sparsity levels, whereas the 2-step method shows rapidly increasing errors beyond sparsity $0.6$.  
This further highlights the robustness of the 1-step approach under sparse data conditions.

\begin{figure}[htbp]
\centering

\begin{subfigure}{\textwidth}
    \centering
    \begin{subfigure}{0.36\textwidth}
        \includegraphics[width=\linewidth]{ 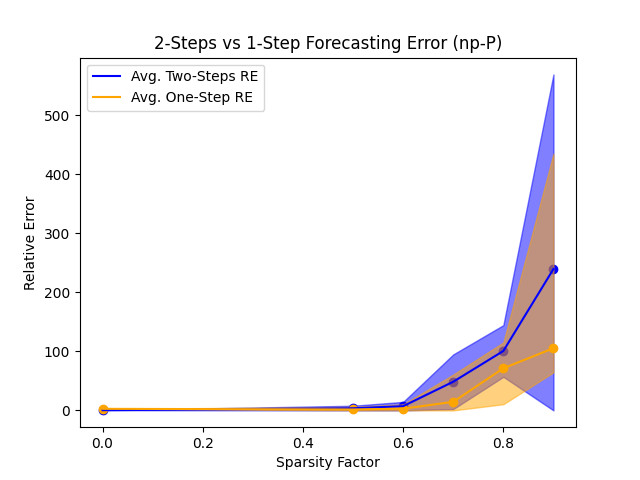}
        \caption*{$p$-error, Gaussian}
    \end{subfigure}
    \begin{subfigure}{0.36\textwidth}
        \includegraphics[width=\linewidth]{ 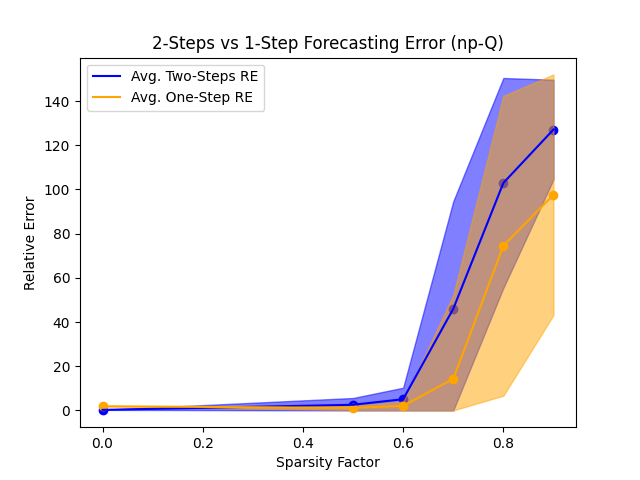}
        \caption*{$q$-error, Gaussian}
    \end{subfigure}
\end{subfigure}

\begin{subfigure}{\textwidth}
    \centering
    \begin{subfigure}{0.36\textwidth}
        \includegraphics[width=\linewidth]{ 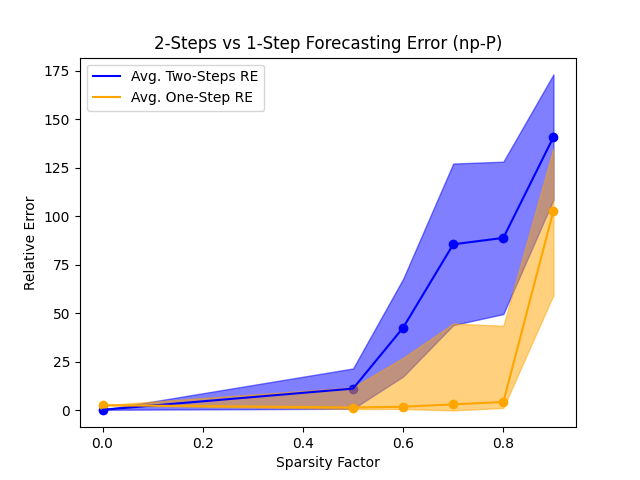}
        \caption*{$p$-error, Separable Polynomial}
    \end{subfigure}
    \begin{subfigure}{0.36\textwidth}
        \includegraphics[width=\linewidth]{ 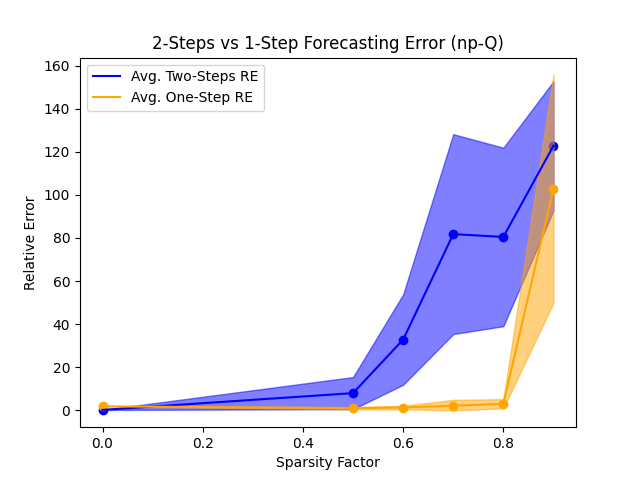}
        \caption*{$q$-error, Separable Polynomial}
    \end{subfigure}
\end{subfigure}

\caption{Forecasting relative error  for the Nonlinear Pendulum System.}
\label{fig:for_error_np_all}
\end{figure}

\subsection{Multiple-Trajectory Experiment}
\label{sec:multiple-trajectories}

We now consider partial observations of multiple trajectories generated
by the same Hamiltonian system. Let $M\in\mathbb N$ denote the number
of training trajectories and write
\[
    y^{r}(t)
    :=
    \begin{pmatrix}
        q^{r}(t)\\
        p^{r}(t)
    \end{pmatrix},
    \qquad
    r=1,\ldots,M.
\]
The trajectories correspond to different initial conditions but satisfy
the same Hamiltonian equations,
\[
    \dot q^{r}(t)
    =
    \partial_p H
    \bigl(q^{r}(t),p^{r}(t)\bigr),
    \qquad
    \dot p^{r}(t)
    =
    -\partial_q H
    \bigl(q^{r}(t),p^{r}(t)\bigr).
\]
Thus, each trajectory has its own position and momentum functions, while
the Hamiltonian $H$ is common to all trajectories.

\subsubsection{Multiple-trajectory formulation}

Let
\[
T := (t_1,\ldots,t_N)^\top
\]
denote the common vector of collocation times. For trajectory $r$, let
\[
S^{r}
:=
\bigl(s_1^{r},\ldots,s_{N_{\mathrm{obs},r}}^{r}\bigr)^\top
\subseteq T
\]
denote its observation times. The available data are
\[
\left\{
q^{r}(S^{r}),p^{r}(S^{r})
\right\}_{r=1}^{M}.
\]
We collect the trajectory functions as
\[
\mathbf q
:=
\bigl(q^{1},\ldots,q^{M}\bigr)
\in\mathcal H_\Gamma^{M},
\qquad
\mathbf p
:=
\bigl(p^{1},\ldots,p^{M}\bigr)
\in\mathcal H_\Sigma^{M}.
\]
Here $\mathcal H_\Gamma^{M}$ and $\mathcal H_\Sigma^{M}$ are finite
Cartesian products of the original trajectory RKHSs.

For the multiple-trajectory 2-step method, the trajectory reconstruction
problems are
\[
\begin{aligned}
\mathbf q_{M}^\star
&:=
\arg\min_{\widetilde{\mathbf q}\in\mathcal H_\Gamma^{M}}
\frac{1}{M}
\sum_{r=1}^{M}
\|\widetilde q^{r}\|_\Gamma^2
\\
&\hspace{1cm}\text{subject to}\quad
\widetilde q^{r}(S^{r})=q^{r}(S^{r}),
\qquad r=1,\ldots,M,
\end{aligned}
\]
and
\[
\begin{aligned}
\mathbf p_{M}^\star
&:=
\arg\min_{\widetilde{\mathbf p}\in\mathcal H_\Sigma^{M}}
\frac{1}{M}
\sum_{r=1}^{M}
\|\widetilde p^{r}\|_\Sigma^2
\\
&\hspace{1cm}\text{subject to}\quad
\widetilde p^{r}(S^{r})=p^{r}(S^{r}),
\qquad r=1,\ldots,M.
\end{aligned}
\]
These problems separate into $M$ independent minimum-norm
interpolations using the same kernels $\Gamma$ and $\Sigma$.
Once the trajectories have been reconstructed and differentiated, one
common Hamiltonian is fit using all of the reconstructed phase-space
points
\[
\begin{aligned}
H_{M}^\star
:=
\arg\min_{\widetilde H\in\mathcal H_\Psi}
\quad&
\|\widetilde H\|_\Psi^2
\\
&+
\frac{1}{M\lambda_1}
\sum_{r=1}^{M}
\left\|
\partial_p\widetilde H
\left(
q_{M}^{r,\star}(T),
p_{M}^{r,\star}(T)
\right)
-
\dot q_{M}^{r,\star}(T)
\right\|_2^2
\\
&+
\frac{1}{M\lambda_2}
\sum_{r=1}^{M}
\left\|
\partial_q\widetilde H
\left(
q_{M}^{r,\star}(T),
p_{M}^{r,\star}(T)
\right)
+
\dot p_{M}^{r,\star}(T)
\right\|_2^2.
\end{aligned}
\]

For the 1-step method, let
\[
\widetilde{\mathbf q}
:=
(\widetilde q^{1},\ldots,\widetilde q^{M})
\in\mathcal H_\Gamma^{M},
\qquad
\widetilde{\mathbf p}
:=
(\widetilde p^{1},\ldots,\widetilde p^{M})
\in\mathcal H_\Sigma^{M}.
\]
The multiple-trajectory problem is
\[
\begin{aligned}
\min_{
\widetilde{\mathbf q}\in\mathcal H_\Gamma^{M},\,\,
\widetilde{\mathbf p}\in\mathcal H_\Sigma^{M},\,\,
\widetilde H\in\mathcal H_\Psi
}
\quad&
\|\widetilde H\|_\Psi^2
+
\frac{1}{M\lambda_1}
\sum_{r=1}^{M}
\|\widetilde q^{r}\|_\Gamma^2
+
\frac{1}{M\lambda_2}
\sum_{r=1}^{M}
\|\widetilde p^{r}\|_\Sigma^2
\\
\text{subject to}\quad&
\widetilde q^{r}(S^{r})=q^{r}(S^{r}),
\\
&
\widetilde p^{r}(S^{r})=p^{r}(S^{r}),
\\
&
\dot{\widetilde q}^{r}(T)
=
\partial_p\widetilde H
\left(
\widetilde q^{r}(T),\widetilde p^{r}(T)
\right),
\\
&
\dot{\widetilde p}^{r}(T)
=
-\partial_q\widetilde H
\left(
\widetilde q^{r}(T),\widetilde p^{r}(T)
\right),
\qquad r=1,\ldots,M.
\end{aligned}
\]
The normalization by $M$ keeps the relative weight of the trajectory
regularization terms fixed as the number of trajectories changes. 

\subsubsection{Stacked implementation}

The multiple-trajectory problem is implemented by stacking the single-trajectory variables and constraints. For each trajectory $r$, the observation and derivative functionals $\Phi^r$ and latent derivative vectors $z_1^r(T),z_2^r(T)$ are defined exactly as in the single-trajectory case, so the representer formulas for $q^{r,\star}$ and $p^{r,\star}$ apply independently to each trajectory. Consequently, the trajectory kernel matrices are block diagonal, reflecting that each trajectory has its own position and momentum functions. In contrast, the Hamiltonian is shared across all trajectories: the corresponding derivative-evaluation functionals and latent derivatives are pooled into $\varphi_M$ and $z_M$, yielding
\[
H_M^\star(x)
=
\Psi(x,\varphi_M)
\bigl(\Psi(\varphi_M,\varphi_M)+\lambda I\bigr)^{-1}z_M.
\]
Here $\Psi(\varphi_M,\varphi_M)$ is generally not block diagonal, since its off-diagonal blocks couple phase-space points from different trajectories through the common Hamiltonian. Thus, the reduced problem simply combines the trajectory-specific terms with a single Hamiltonian term built from all phase-space samples, and the same optimizer as in the single-trajectory case is applied to the resulting stacked problem.

\subsubsection{Extension of the theoretical analysis}
The proof of theorem \ref{thm:rates} extends directly to any fixed finite collection of trajectories. The trajectory sampling argument is applied separately to each trajectory, since each $q^r$ and $p^r$ remains a function of time on $[0,t_f]$, while the Hamiltonian sampling argument is applied to the union of the phase-space samples from all trajectories. Thus, the trajectory error bounds have the same form as in theorem \ref{thm:rates}, with the temporal fill distance replaced by the worst-case fill distance across trajectories, $\rho_{M,N}:=\max_{1\le r\le M}\rho_{r,N}$, and the corresponding trajectory norms averaged over $r$. For $\nabla H$, the relevant fill distance is instead that of the pooled phase-space point cloud
\(
Z_{M,N}:=\bigcup_{r=1}^M\{y^r(s):s\in S_N^r\}.
\)
Consequently, additional trajectories do not change the temporal approximation argument, but may improve the phase-space coverage and hence the regional estimate of $\nabla H$. As in the single-trajectory case, the result is restricted to data-covered regions, and $H$ is determined only up to an additive constant. For a related treatment of the multiple-solution setting in the PDE case, we refer the reader to \cite{jalalian2025keql}.

\subsubsection{Experimental setup}

We illustrate the multiple-trajectory formulation for the
Two-Mass-Three-Spring system introduced earlier, whose
Hamiltonian is
\[
    H(q,p)
    =
    \frac{1}{2}
    \left(
        p_1^2+p_2^2
    \right)
    +
    \frac{1}{2}
    \left(
        q_1^2+q_2^2+(q_2-q_1)^2
    \right).
\]
We generate a fixed collection of trajectories from distinct initial
conditions and vary the number used for training according to
\[
    M\in\{1,\ldots,10\}.
\]
For each value of $M$, both methods are supplied with the same
trajectories and the same observation sets.

Each trajectory is generated on the temporal grid described in
section \ref{sec:common}, consisting of $N_{\mathrm{col}}=200$ uniformly spaced
collocation times over the training interval $[0,40]$. The learned
Hamiltonian system is then simulated from each corresponding initial
condition and evaluated using the testing procedure of
section \ref{sec:common}.

We consider the sparsity factors
\(
    \alpha\in\{0,0.5,0.7\}.
\)
At sparsity $\alpha$, a fraction $\alpha$ of the collocation values is
unobserved along each trajectory. For $\alpha=0.5$ and
$\alpha=0.7$, the random subsampling procedure is repeated ten times.
The plotted curves show the mean relative errors and the shaded regions
show one standard deviation over these repetitions. At $\alpha=0$, all
collocation values are observed and the corresponding standard
deviations are zero.

Consistent with the stacked formulation, the true and predicted values
from all $M$ trajectories are concatenated before applying the
relative-error metric of section
\ref{sec:common}. In particular, given
\[
\mathbf q(T_{\mathrm{test}})
:=
\bigl(
q^{1}(T_{\mathrm{test}})^{\top},
\ldots,
q^{M}(T_{\mathrm{test}})^{\top}
\bigr)^{\top},
\qquad
\mathbf q^{\star}(T_{\mathrm{test}})
:=
\bigl(
q^{1,\star}(T_{\mathrm{test}})^{\top},
\ldots,
q^{M,\star}(T_{\mathrm{test}})^{\top}
\bigr)^{\top},
\]
the multiple-trajectory relative error for $q$ is
\[
    \operatorname{RE}_{\mathbf q}^{M}
    :=
    \frac{
        \left\|
            \mathbf q(T_{\mathrm{test}})
            -
            \mathbf q^\star(T_{\mathrm{test}})
        \right\|_2
    }{
        \left\|
            \mathbf q(T_{\mathrm{test}})
        \right\|_2
    },
\]
and $\operatorname{RE}_{\mathbf p}^{M}$ is defined analogously. All kernel
choices, regularization parameters, initialization procedures, and
optimization settings are kept fixed as $M$ varies.

\begin{figure}[t]
    \centering
    \setlength{\tabcolsep}{2pt}
    \begin{tabular}{@{}cc@{}}
        \includegraphics[width=0.4\textwidth]
            {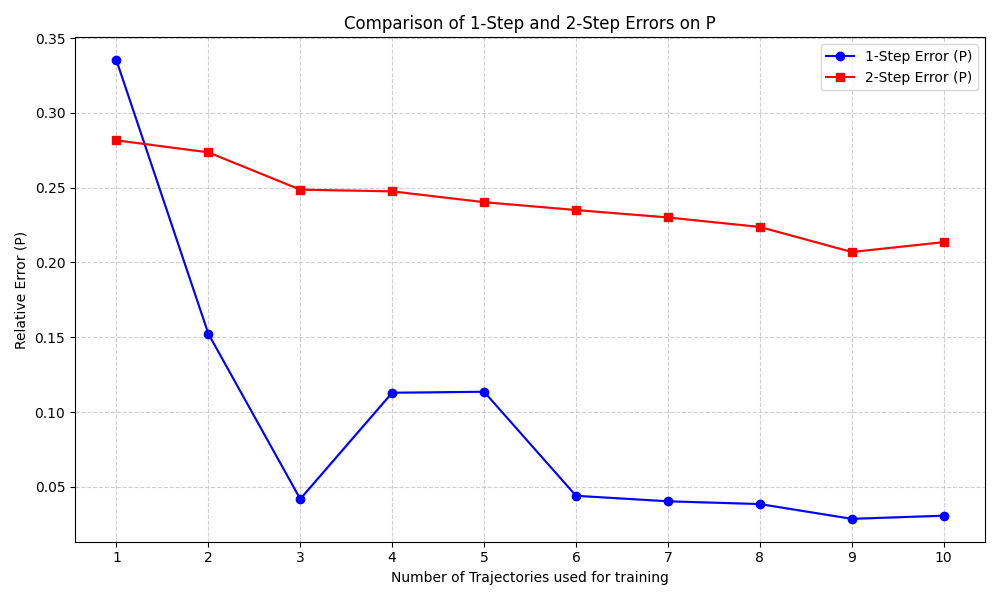}
        &
        \includegraphics[width=0.4\textwidth]
            {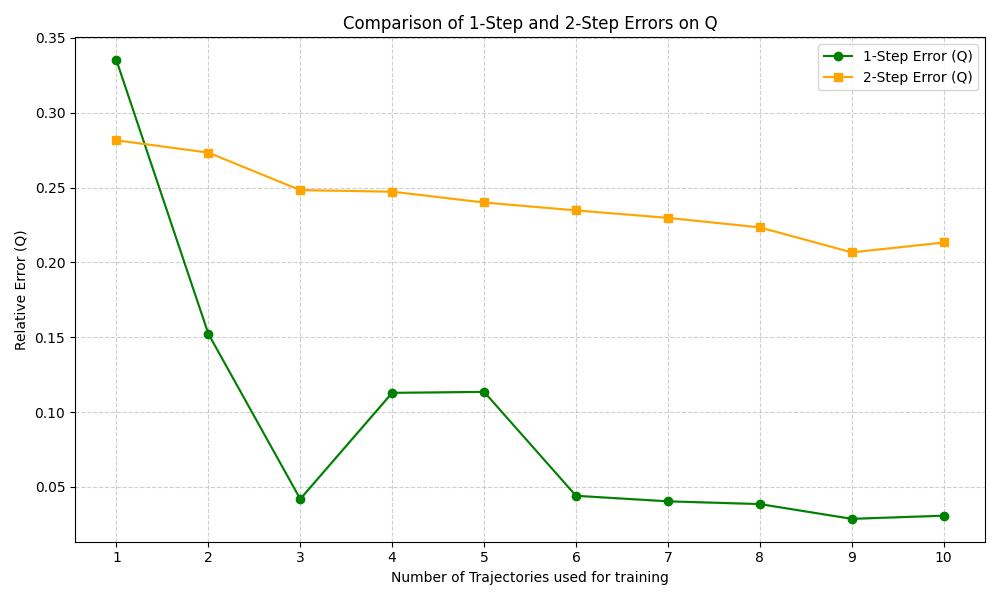}
        \\[3pt]
        \includegraphics[width=0.4\textwidth]
            {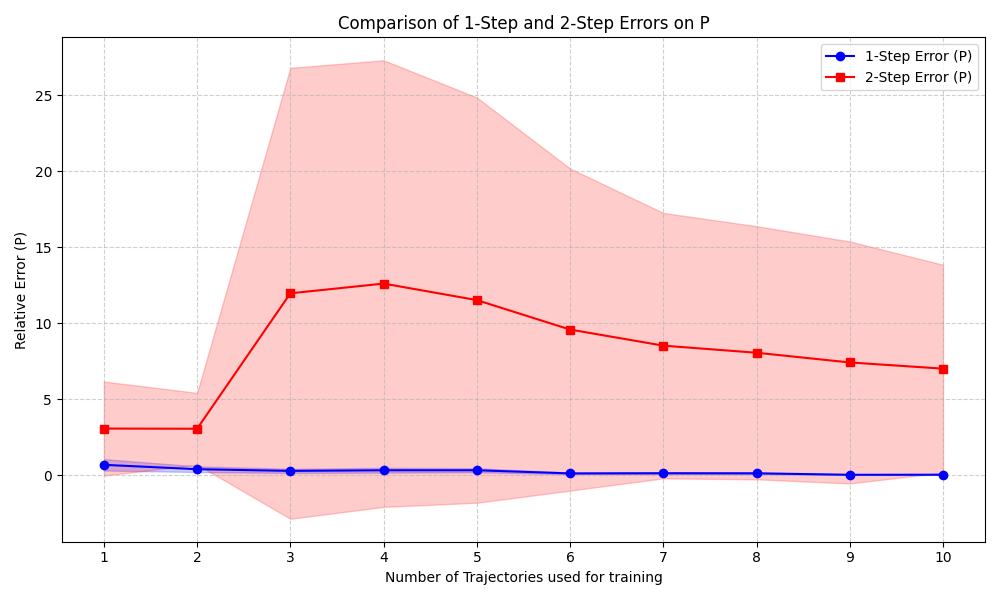}
        &
        \includegraphics[width=0.4\textwidth]
            {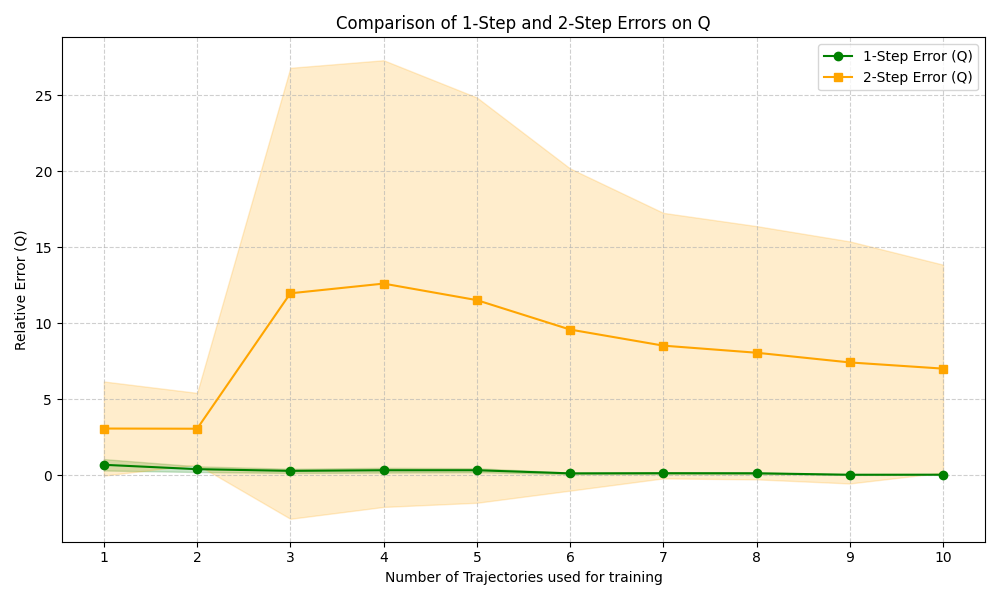}
        \\[3pt]
        \includegraphics[width=0.4\textwidth]
            {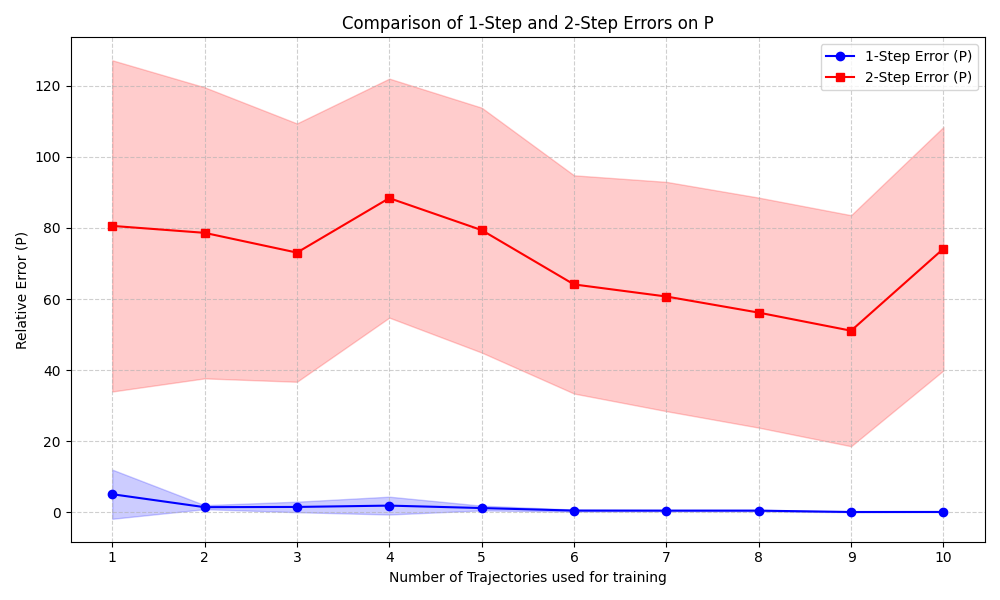}
        &
        \includegraphics[width=0.4\textwidth]
            {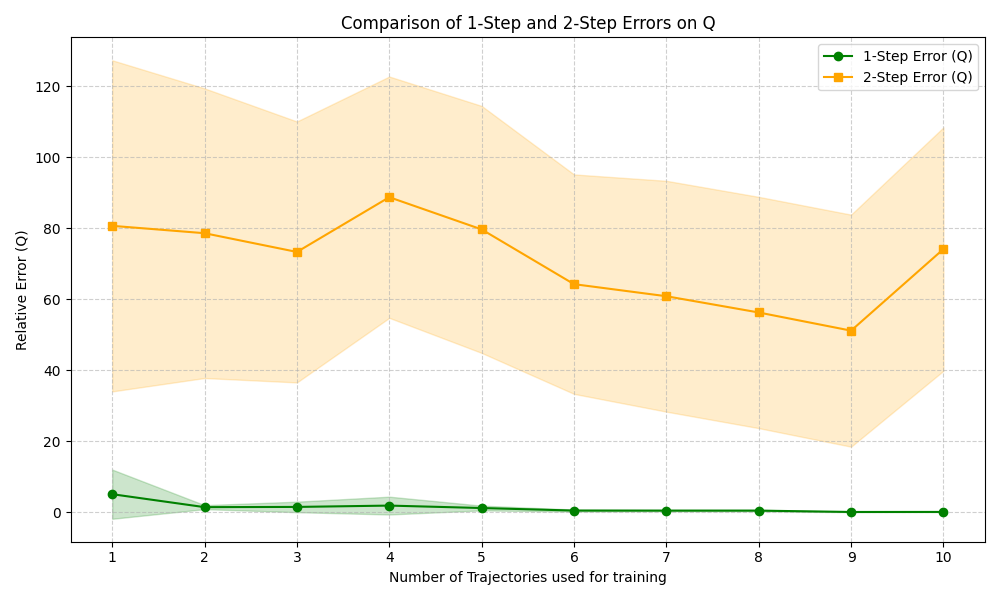}
    \end{tabular}
    \caption{Relative errors as functions of the number of training
    trajectories for the Two-Mass-Three-Spring system. The left and
    right columns show the errors for $p$ and $q$, respectively. From
    top to bottom, the rows correspond to sparsity factors $\alpha=0$,
    $\alpha=0.5$, and $\alpha=0.7$. For $\alpha=0.5$ and
    $\alpha=0.7$, the curves show the mean relative errors and the
    shaded regions show one standard deviation over ten random
    subsampling trials. At $\alpha=0$, the standard deviations are
    zero.}
    \label{fig:multiple-trajectory-errors}
\end{figure}

\subsubsection{Results}

Figure~\ref{fig:multiple-trajectory-errors} reports the relative errors
for $p$ and $q$ as functions of the number of trajectories included in
the training problem. The errors for $p$ and $q$ are nearly identical
at all three sparsity levels, so we discuss their common qualitative
behavior.

In the fully observed setting, $\alpha=0$, the 2-step method has the
smaller error when only one trajectory is used. For $p$, the errors with
one trajectory are approximately $0.282$ for the 2-step method and
$0.335$ for the 1-step method. Beginning with two trajectories,
however, the 1-step method has the smaller error. With ten trajectories,
the 1-step error is approximately $0.031$, whereas the 2-step error is
approximately $0.214$. Thus, although the dependence on $M$ is not
monotone at every intermediate value, the 1-step method exhibits a
substantial overall improvement as additional fully observed
trajectories are included.

At sparsity $\alpha=0.5$, the difference between the methods becomes
substantially larger. With one trajectory, the mean relative error for
$p$ is approximately $0.688$ for the 1-step method and $3.074$ for the
2-step method. With ten trajectories, the mean 1-step error decreases
to approximately $0.037$, whereas the mean 2-step error remains
approximately $7.02$. The 2-step method also exhibits considerably
wider standard-deviation bands, especially for intermediate values of
$M$.

At sparsity $\alpha=0.7$, the contrast is more pronounced. With one
trajectory, the mean relative errors for $p$ are approximately $5.08$
for the 1-step method and $80.6$ for the 2-step method. The mean
1-step error decreases overall as additional trajectories are included,
reaching approximately $0.074$ with nine trajectories and $0.087$ with
ten trajectories. The 2-step error remains above $50$ throughout the
displayed range and is approximately $74.2$ with ten trajectories. Its
standard-deviation bands also remain large.

The errors do not decrease monotonically at every intermediate value of
$M$. In the sparse experiments, the particular observed portions of the
trajectories vary across the random subsampling trials. In addition,
the reduced 1-step problem is generally nonconvex. These effects can
produce fluctuations as trajectories are added, and the relevant
comparison is therefore the overall behavior across the full range of
$M$.

The distinction between the methods follows from how the shared
Hamiltonian enters the reconstruction. In the 2-step method, the $M$
trajectories are first reconstructed and differentiated independently,
and the common Hamiltonian is introduced only afterward. Errors
produced by sparse interpolation are therefore already present in the
derivative data used to fit $H$. In the 1-step method, all $M$
trajectory reconstructions are constrained simultaneously by the same
Hamiltonian. Information from the different trajectories is consequently
combined through the common Hamiltonian vector field during the
trajectory reconstruction itself.

This coupling occurs only through the Hamiltonian. The product-space
formulation does not directly transfer the time-dependent position or
momentum values of one trajectory to another. Each trajectory retains
its own functions in $\mathcal H_\Gamma$ and $\mathcal H_\Sigma$,
while the pooled phase-space samples jointly constrain one element of
$\mathcal H_\Psi$.

\section{Concluding remarks}

In this work, we developed kernel-based methodologies, inspired by the CGC and KEqL frameworks \cite{houman_cgc, jalalian2025keql}, for learning Hamiltonian dynamical systems from time-series data, addressing both system identification and forecasting tasks. We introduced a principled approach to approximate unknown trajectories and Hamiltonians using RKHS theory. Our results show that the simultaneous learning method (our 1-step method) consistently outperforms the standard kernel based approach (our 2-step method), particularly in data-scarce regimes, while preserving the Hamiltonian structure, so that the exact flow generated by the learned Hamiltonian conserves the learned Hamiltonian. We also established theoretical error rates that ensure the reliability of the learned models. In addition, the implementation we provide is general and can be adapted to other CGC \cite{houman_cgc} problems beyond Hamiltonian dynamics. Taken together, the theoretical guarantees, and structure-preserving properties of our methods make them a compelling alternative to existing machine learning approaches. 

Future directions include (i) extending the framework to larger-scale and higher-dimensional Hamiltonian systems, (ii) developing adaptive kernel designs that account for the specific characteristics of the underlying system to improve data efficiency and scalability, such as the data-adapted kernel learning approach of \cite{lengyel2024adapted}, and (iii) developing a non-asymptotic error analysis that is better suited to sparse-data regimes.

\section{Acknowledgment}

The authors acknowledge support from Beyond Limits (Learning Optimal Models). 
YJ and HO acknowledge support from the Air Force Office of Scientific Research under MURI awards number FA9550-20-1-0358 (Machine Learning and Physics-Based Modeling and Simulation), FOA-AFRL-AFOSR-2023-0004 (Mathematics of Digital Twins), the Department of Energy under award number DE-SC0023163 (SEA-CROGS: Scalable, Efficient, and Accelerated Causal Reasoning Operators, Graphs and Spikes for Earth and Embedded Systems), the National Science Foundations under award number 2425909 (Discovering the Law of Stress Transfer and Earthquake Dynamics in a Fault Network using a Computational Graph Discovery Approach), and the Jet Propulsion Laboratory 
PDRDF 24AW0133 (Jupiter’s Radiation Environment: Assimilating Data with ML-driven Approaches). Additionally, HO acknowledges support from the DoD Vannevar Bush Faculty Fellowship Program (ONR N000142512035).

\bibliographystyle{plain}
\bibliography{references}

\appendix

\section{Background Results for the Theory}\label{app:theory}
In this section, we provide results used in the proof of theorem \ref{thm:rates}, which also justify the need for some of the assumptions needed for the theorem to hold in section \ref{sec:theory}. We assume the reader is familiar with classical results in kernel methods and RKHS theory, as well as classical results of functional analysis. The results in this appendix are classical in approximation theory with kernel methods, but the formulations are adapted from \cite{jalalian2025keql}.

\subsection{Spectral Characterization of RKHS}

\begin{proposition}\label{prop:Mercers-thm}
    Suppose $\Omega \subset \R^q$ is bounded and let $K$ be a positive-definite  kernel that is continuous in both of its arguments on $\overline{\Omega}$. 
    Then there exists an orthonormal set of continuous eigenfunctions
    $\{ e_i \}_{i=1}^\infty \subset L^2(\Omega)$ and 
    decreasing eigenvalues
     $\{\lambda_i \}_{i=1}^\infty$, $\lambda_1 \ge \lambda_2 \ge \dots$, 
    such that 
       $ K(x, x') = \sum_{i=1}^\infty \lambda_i e_i(x) e_i(x'),$ and 
  the RKHS $\mH$ can be characterized as 
    \begin{equation}\label{Mercer-RKHS-characterization}
        \mH = \left\{ f: \Omega \to \R \Bigg| f(x) 
        = \sum_{i \in \{ i \mid \lambda_i \neq 0\}} c_i(f)  e_i(x), \quad 
        \sum_{i \in \{i \mid \lambda_i \neq 0 \}} \lambda_i^{-1}c_i(f)^2  < +\infty
        \right\}
    \end{equation}
    and for any pair $f,f' \in \mH$ we have  
    $\langle f, f' \rangle_{\mH} =  \sum_{i \in \{ i \mid \lambda_i \neq 0\} } \lambda_i^{-1} c_i(f) c_i(f')$. 
\end{proposition}

\begin{definition}
Given  \eqref{Mercer-RKHS-characterization}, we further define the nested ladder of RKHS spaces 
\begin{equation*}
    \mH^\gamma := \left\{ f: \Omega \to \R \Bigg| f(x) = \sum_{i=1}^\infty c_i(f) 
     e_i(x) , \quad \sum_{i=1}^\infty \lambda_i^{-\gamma} c_i(f)^2 < + \infty\right\},
\end{equation*}
for $\gamma \ge 1$.
Naturally, larger 
values of $\gamma$ imply more "smoothness", in particular we have 
the inclusion $\mH^{\gamma_2} \subset \mH^{\gamma_1}$ for $1 \le \gamma_1 < \gamma_2$.
\end{definition}

\begin{lemma}\label{lem:H-gamma-bound}
Suppose $f\in \mH^{2\gamma}$ and $f' \in \mH^\gamma$. Then
\begin{align*}
\langle f', f \rangle_{\mH^\gamma} \le 
\| f'\|_{L^2(\Omega)} \| f\|_{\mH^{2\gamma}}.   
\end{align*}
\end{lemma}

\subsection{Sobolev Embedding \& Sampling Inequalities}
\begin{theorem}[Sobolev embedding theorem {\cite[Thm. 4.12]{adams2003sobolev}}]\label{prop:sobolev-embedding}
Suppose $\Omega \subset \R^d$ is a bounded set with Lipschitz boundary and that for 
$\alpha \in \mathbb{N}$ it holds that $\gamma > d/2 + \alpha$. Then $H^\gamma(\Omega)$ 
is continuously embedded in $C^\alpha(\Omega)$ and it holds that 
\begin{align*}
\| u \|_{C^\alpha(\Omega)} \le C_\Omega \| u \|_{H^\gamma(\Omega)}  
\end{align*}
for an embedding constant $C_\Omega \ge 0$ that depends only on $\Omega$.
\end{theorem}

\begin{theorem}[Sobolev sampling inequality {\cite[Prop. 3]{jalalian2025keql}}]\label{prop:sobolev-sampling-inequality}
Suppose $\Omega \subset \R^d$ is a bounded set with Lipschitz 
boundary and consider a set of points
 $X = \{ x_1, \dots, x_N \} \subset \overline{\Omega}$
with 
fill distance  $h_X:= sup_{x \in \Omega} \inf_{x' \in X} \| x - x' \|_2$.
Let $u|_X$ denote the restriction of $u$ to the set $X$.
Further consider  $\gamma > d/2$ and $0 \le \eta \le \gamma$ and 
let $u \in H^\gamma(\Omega)$. 
\begin{enumerate}[label=(\alph*)]
    \item (Noiseless) Suppose $u|_X = 0$. 
Then there exists $h_0>0$ so that whenever $h_X \le h_0$ we have 
\begin{align*}
\| u \|_{H^\eta(\Omega)} \le C_\Omega h_X^{\gamma - \eta} \| u \|_{H^\gamma(\Omega)},  
\end{align*}
where $C_\Omega >0$ is a constant that depends only on $\Omega$.
\item (Noisy) Suppose $u|_X \neq 0$.
Then there exists $h_0>0$ so that whenever $h_X \le h_0$ we have the inequality 
\begin{align*}
\| u \|_{L^\infty(\Omega)} 
\le C_\Omega h_X^{\gamma - d/2} \| u \|_{H^\gamma(\Omega)} 
+ 2 \| u|_X \|_\infty,
\end{align*}
where $C_\Omega >0$ is a constant that depends only on $\Omega$.
\end{enumerate}
\end{theorem}

\section{Generic CGC framework}\label{sec:app_cgc}

As discussed in \ref{subsec:cgc}, for the 1-step method, we implemented a generic code following the computational graph completion (CGC) method introduced in \cite{houman_cgc}.
The process starts by creating an instance of the \texttt{ComputationalGraph} class which can be used to build the full graph using the following functions:
\begin{itemize}
    \item \texttt{add\_observable}: Takes an observable name (like \texttt{"t"} for time) and adds a node for it in the graph. An observable node is considered a root node that doesn't have any incoming edges.
    \item \texttt{add\_unknown\_fn}: Adds a new node that represents the quantity calculated by the function along with a directed edge form the given source node that new quantity node. It also takes a kernel definition along with its parameters values to build the Gaussian process that will estimate the unknown dependency between the source and target nodes. This Gaussian process can be conditioned on any linear transformation of its outputs, in case no value of the target quantity is observed. This feature allows us to condition the Gaussian process for $H$ here on the $\dot{p}$ and $\dot{q}$, which are the results of a linear transformation (the gradient) of $H$.
    \item \texttt{add\_known\_fn}: Adds a new node the represents a quantity calculated by a known function that takes another source node quantity.
    \item \texttt{add\_aggregator}: Aggregates the quantities of two or more nodes into a single new node that can be used as the source of further computation.
    \item \texttt{add\_constraint}: Adds a node that represents a constraint function on a given source node quantity. The output of that constraint function should typically be zero. in an optimized state, hence the value of the constraint functions will be used as a penalty in the optimization process. In the Hamiltonian systems problem here, the constraints function would be the difference between the gradient of $H$ and the values of $\dot{q}$ and $\dot{p}$.
\end{itemize}

Using these methods, we can transform the 1-step computational graph in figure \ref{fig:comp_graph} into the Python code in listing \ref{code:1-d-hamiltonina} for one-dimensional Hamiltonian systems. The completion process itself can be done by calling the \texttt{complete} method of the \texttt{ComputationalGraph} object.  This method takes two Numpy arrays:
\begin{itemize}
    \item \texttt{X}: which contains the observed values of all the quantities  in the graph and some initialization for the non-observed values. The order of the columns of this matrix should follow the observables order given in the initialization of the \texttt{ComputationalGraph} itself.
    \item \texttt{M}: which is a boolean array indicating what values are observed and what values are not.
\end{itemize}

\begin{lstlisting}[language=Python, frame=single, caption=Implementation of one-dimensional Hamiltonian computational graph, label=code:1-d-hamiltonina]
import jax
import cgc
from cgc.utils import KernelParameter as KP

graph = cgc.ComputationalGraph(observables_order=["t", "p", "q", "H"])

graph.add_observable("t")
graph.add_unknown_fn(
    "t", "q", 
    kernel="gaussian", kernel_parameters={"scale": KP(1.0)}
)
graph.add_unknown_fn(
    "t", "p", 
    kernel="gaussian", kernel_parameters="scale": KP(1.0)}
)

graph.add_known_fn("p", "p_dot", cgc.grad)
graph.add_known_fn("q", "q_dot", cgc.grad)
graph.add_known_fn("p_dot", "-p_dot", lambda p_dot: -p_dot)

graph.add_aggregator(["q_dot", "-p_dot"], "qp_dot")

graph.add_aggregator(["p", "q"], "pq")
graph.add_unknown_fn(
    "pq", "H", 
    linear_functional=jax.jacobian, observations="qp_dot", 
    kernel="gaussian", kernel_parameters="scale": KP(1.0)}
)
graph.add_known_fn("H", "grad_H", cgc.grad)

graph.add_aggregator(["q_dot", "grad_H"], "(q_dot, grad_H)")
graph.add_aggregator(["p_dot", "grad_H"], "(p_dot, grad_H)")

def p_dot_constraint(p_dot_grad_H):
    p_dot, grad_H = p_dot_grad_H[:, 0], p_dot_grad_H[:, 1:]
    return p_dot + grad_H[:, 1]

def q_dot_constraint(q_dot_grad_H):
    q_dot, grad_H = q_dot_grad_H[:, 0], q_dot_grad_H[:, 1:]
    return q_dot - grad_H[:, 0]

graph.add_constraint("(p_dot, grad_H)", "W1", p_dot_constraint)
graph.add_constraint("(q_dot, grad_H)", "W2", q_dot_constraint)
\end{lstlisting}

The \texttt{complete} method runs the L-BFGS optimization algorithm on the CGC loss described in \cite{houman_cgc} and returns the \texttt{Z} array which is a copy of the input \texttt{X} array with the non-observed values filled in and completed. This new \texttt{Z} matrix can be used to retrieve any unknown function in the graph and use it for further inference and computations. We use that to retrieve the Hamiltonian function and externally integrate it with an integrator in the experiments reported in section \ref{sec:numerics}.

\begin{remark}
    Since the 2-step method provides closed-form solutions for recovering $q,p,H$, or equivalently, in the 2-step computational graph (figure \ref{fig:comp_graph}), the edges are replaced by closed-form kernel regression solutions, thus there is no need for a \texttt{complete} run. Consequently, no optimization is performed for the 2-step method, and computationally, only the nodes of the graph are considered.
\end{remark}

\section{Supplementary Tables \& Plots}\label{SM: plots}

\begin{figure}[htbp]
    \centering
    \begin{subfigure}{0.36\textwidth}
        \includegraphics[width=\linewidth]{ 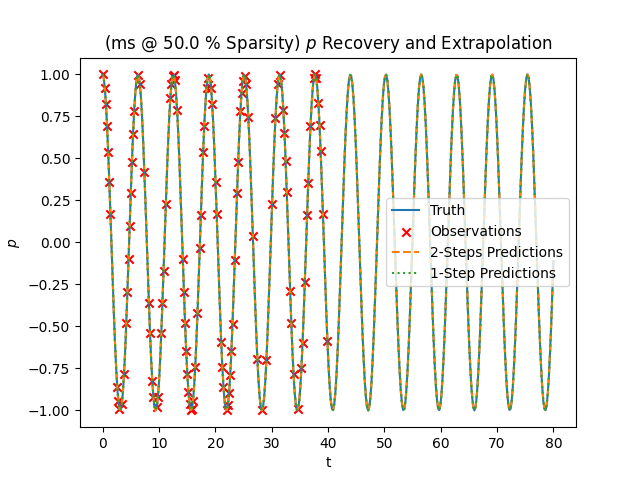}
        \caption*{Gaussian 0.5}
    \end{subfigure}
    \begin{subfigure}{0.36\textwidth}
        \includegraphics[width=\linewidth]{ 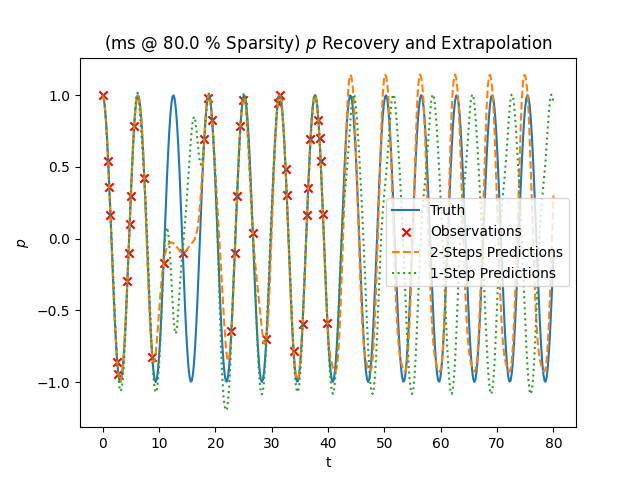}
        \caption*{Gaussian 0.8}
    \end{subfigure}
    \begin{subfigure}{0.36\textwidth}
        \includegraphics[width=\linewidth]{ 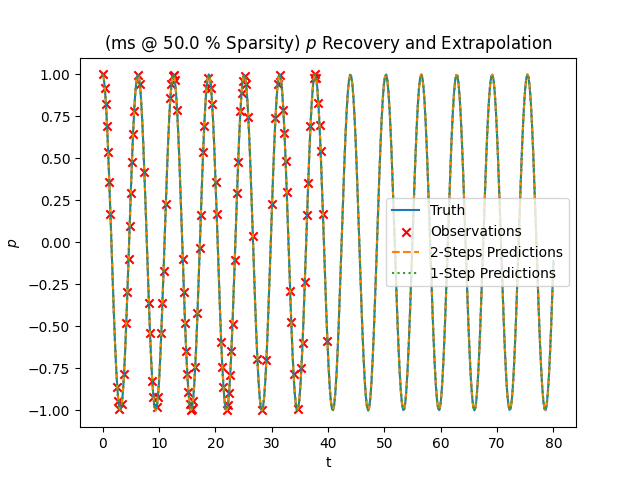}
        \caption*{Poly 0.5}
    \end{subfigure}
    \begin{subfigure}{0.36\textwidth}
        \includegraphics[width=\linewidth]{ 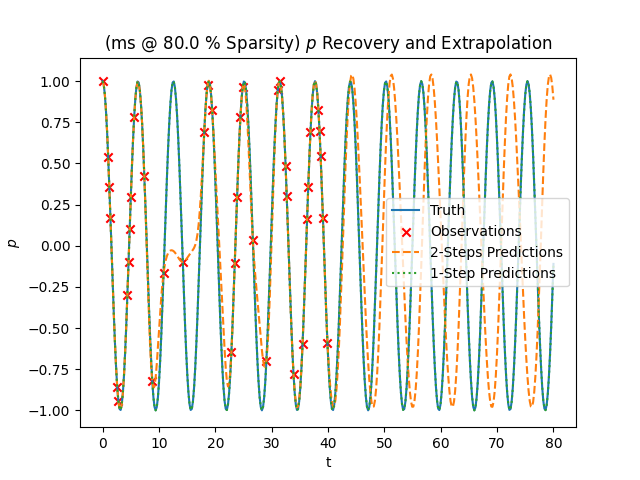}
        \caption*{Poly 0.8}
    \end{subfigure}

    \caption{Recovery of $p$ in the Mass-Spring system.}
    \label{fig:recovery_ms_p}
\end{figure}

\begin{figure}[htbp]
\centering

\begin{subfigure}{0.36\textwidth}
\includegraphics[width=\linewidth]{ 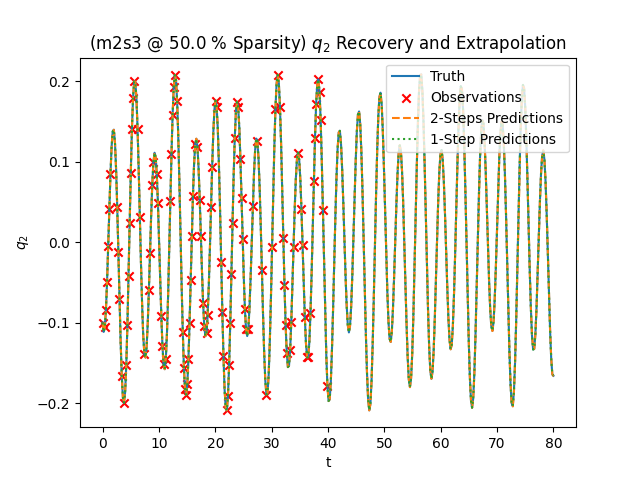}
\caption*{Gaussian 0.5}
\end{subfigure}
\begin{subfigure}{0.36\textwidth}
\includegraphics[width=\linewidth]{ 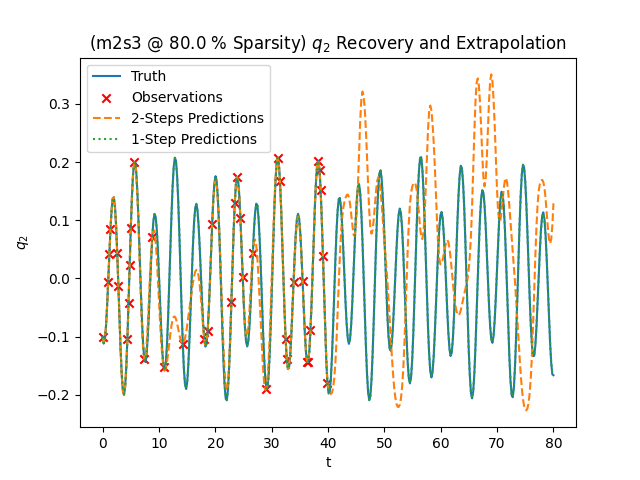}
\caption*{Gaussian 0.8}
\end{subfigure}
\begin{subfigure}{0.36\textwidth}
\includegraphics[width=\linewidth]{ 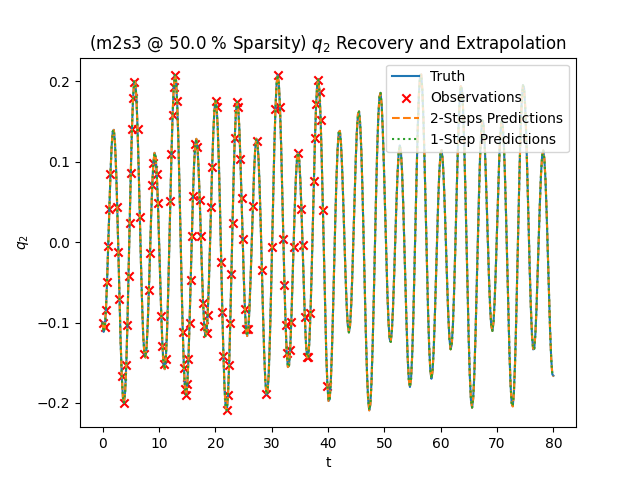}
\caption*{Poly 0.5}
\end{subfigure}
\begin{subfigure}{0.36\textwidth}
\includegraphics[width=\linewidth]{ 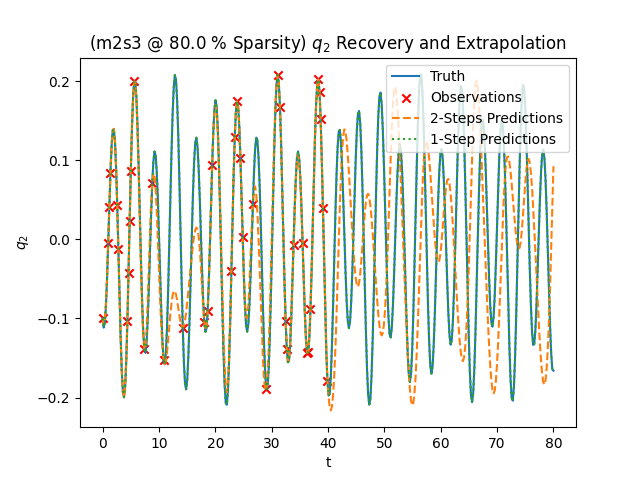}
\caption*{Poly 0.8}
\end{subfigure}
\caption{Recovery of $q_2$ in the Two-Mass-Three-Spring system.}
\end{figure}

\begin{figure}[htbp]
\centering
\begin{subfigure}{0.36\textwidth}
\includegraphics[width=\linewidth]{ 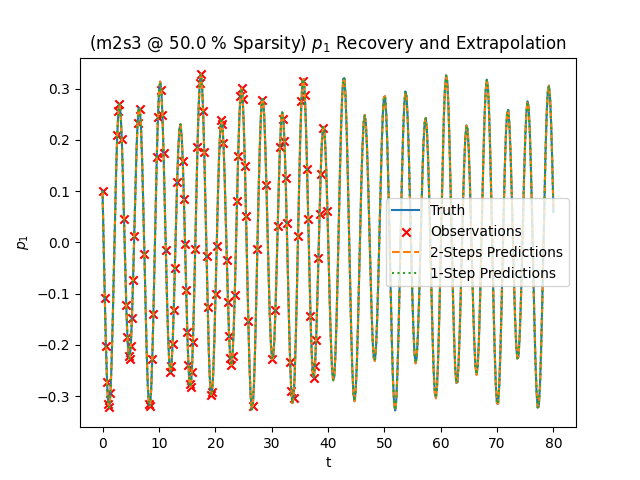}
\caption*{Gaussian 0.5}
\end{subfigure}
\begin{subfigure}{0.36\textwidth}
\includegraphics[width=\linewidth]{ 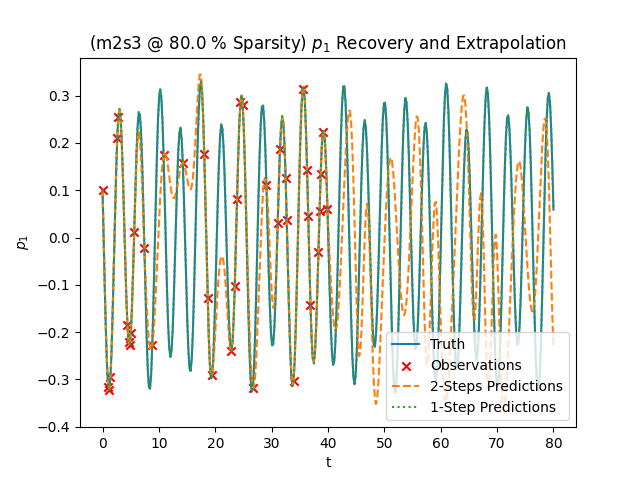}
\caption*{Gaussian 0.8}
\end{subfigure}

\begin{subfigure}{0.36\textwidth}
\includegraphics[width=\linewidth]{ 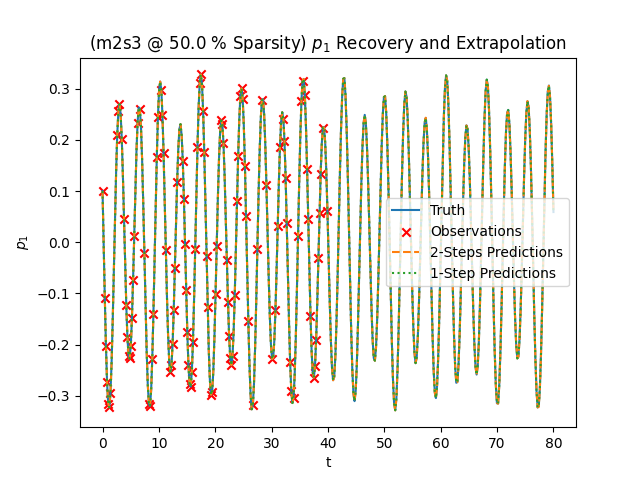}
\caption*{Poly 0.5}
\end{subfigure}
\begin{subfigure}{0.36\textwidth}
\includegraphics[width=\linewidth]{ 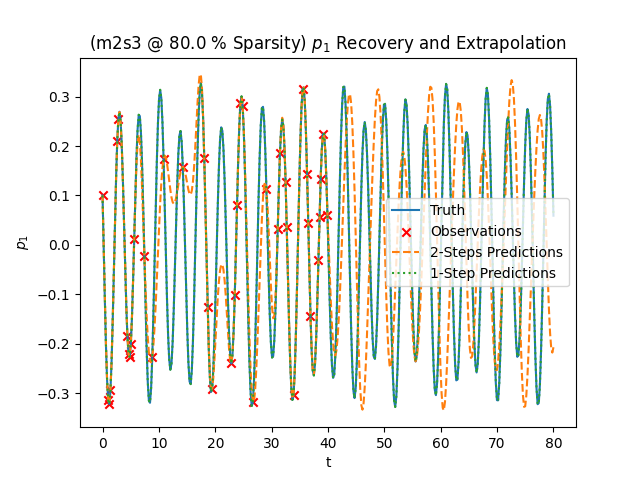}
\caption*{Poly 0.8}
\end{subfigure}
\caption{Recovery of $p_1$ in the Two-Mass-Three-Spring system.}
\end{figure}

\begin{figure}[htbp]
\centering

\begin{subfigure}{0.36\textwidth}
\includegraphics[width=\linewidth]{ 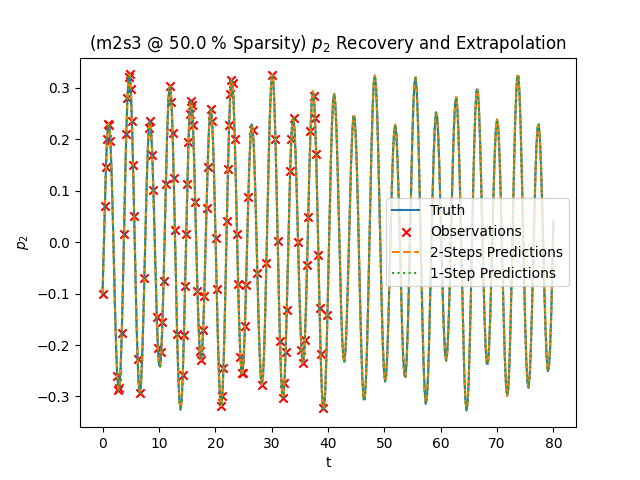}
\caption*{Gaussian 0.5}
\end{subfigure}
\begin{subfigure}{0.36\textwidth}
\includegraphics[width=\linewidth]{ 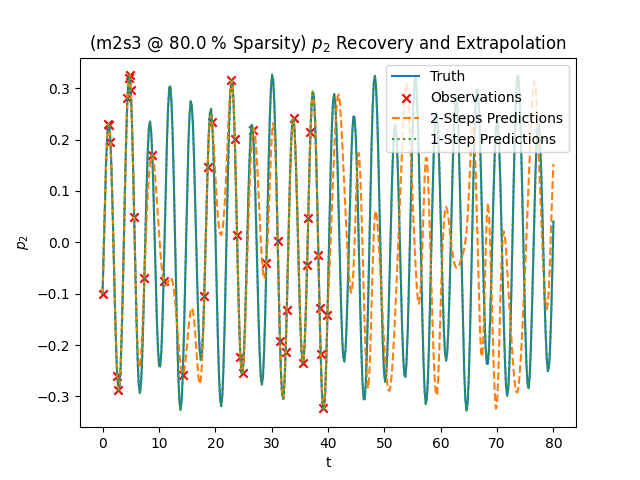}
\caption*{Gaussian 0.8}
\end{subfigure}

\begin{subfigure}{0.36\textwidth}
\includegraphics[width=\linewidth]{ 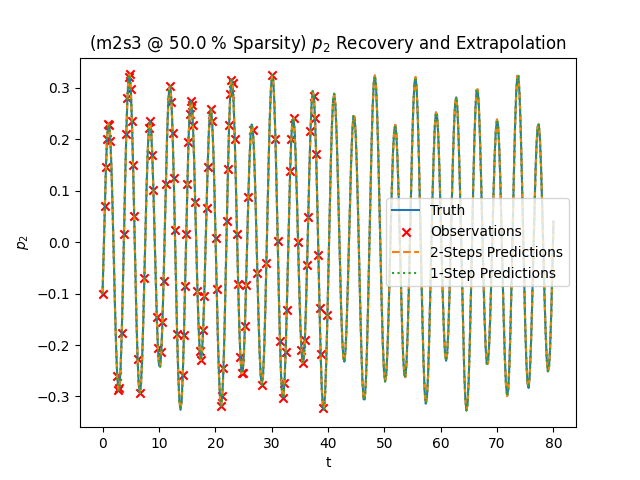}
\caption*{Poly 0.5}
\end{subfigure}
\begin{subfigure}{0.36\textwidth}
\includegraphics[width=\linewidth]{ 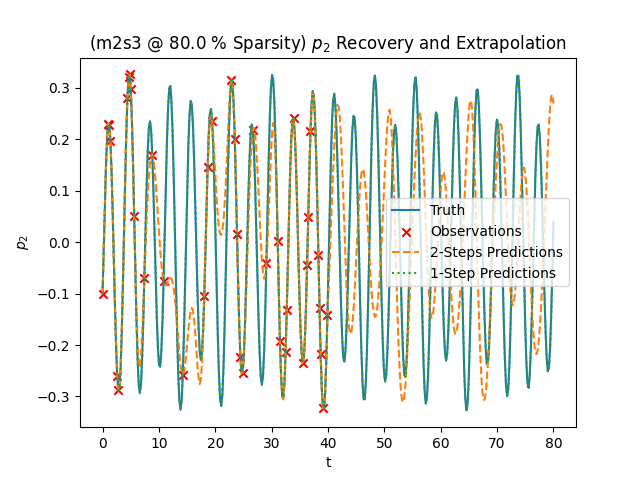}
\caption*{Poly 0.8}
\end{subfigure}

\caption{Recovery of $p_2$ in the Two-Mass-Three-Spring system.}
\label{fig:recovery_m2s3}
\end{figure}

\begin{figure}[htbp]
\centering
\begin{subfigure}{0.36\textwidth}
\includegraphics[width=\linewidth]{ 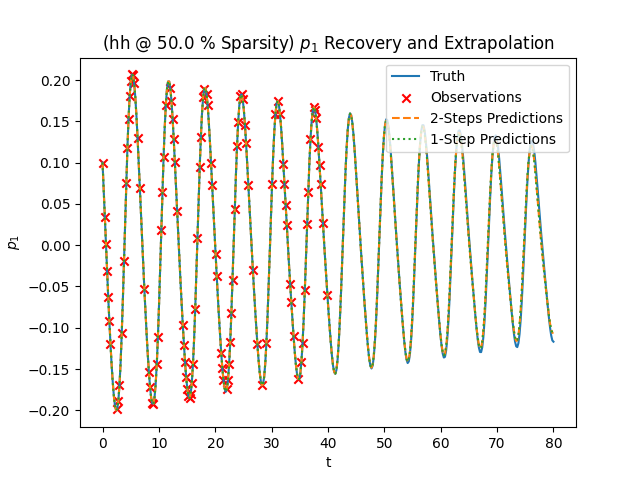}
\caption*{Gaussian 0.5}
\end{subfigure}
\begin{subfigure}{0.36\textwidth}
\includegraphics[width=\linewidth]{ 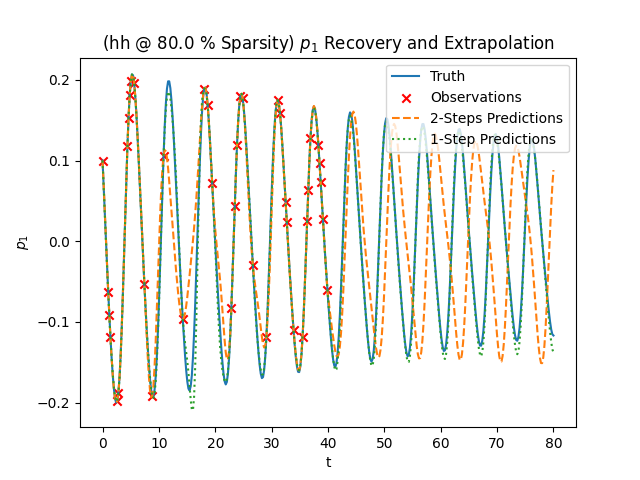}
\caption*{Gaussian 0.8}
\end{subfigure}

\begin{subfigure}{0.36\textwidth}
\includegraphics[width=\linewidth]{ 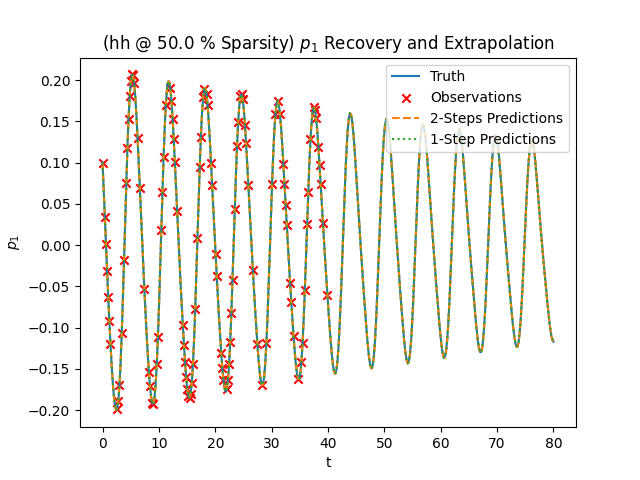}
\caption*{Poly 0.5}
\end{subfigure}
\begin{subfigure}{0.36\textwidth}
\includegraphics[width=\linewidth]{ 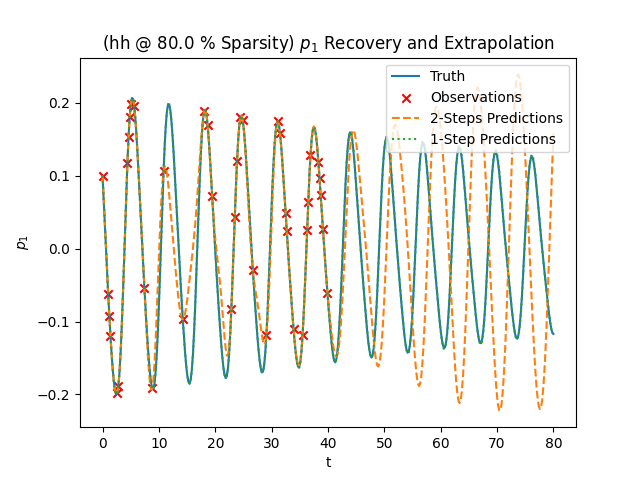}
\caption*{Poly 0.8}
\end{subfigure}
\caption{Recovery of $p_1$ in the Hénon-Heiles system.}
\label{fig:recovery_hh}
\end{figure}

\begin{figure}[htbp]
\centering
\begin{subfigure}{0.36\textwidth}
\includegraphics[width=\linewidth]{ 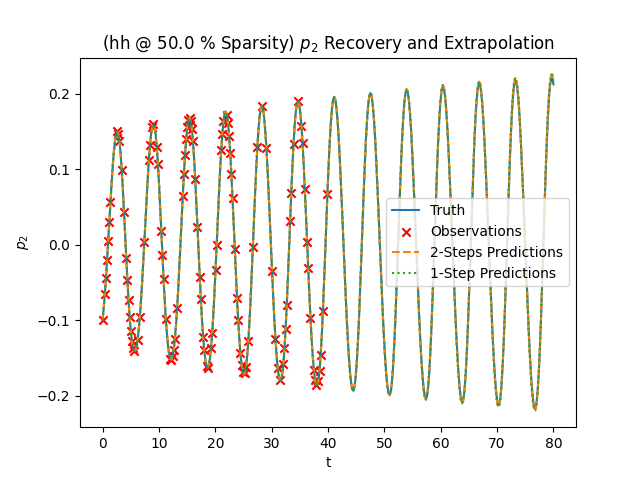}
\caption*{Gaussian 0.5}
\end{subfigure}
\begin{subfigure}{0.36\textwidth}
\includegraphics[width=\linewidth]{ 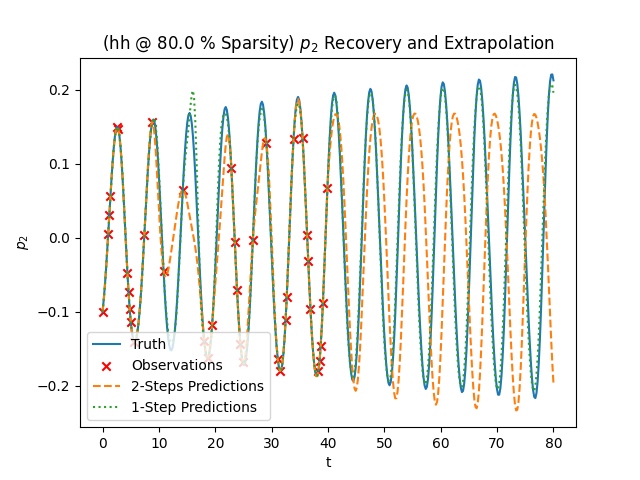}
\caption*{Gaussian 0.8}
\end{subfigure}
\begin{subfigure}{0.36\textwidth}
\includegraphics[width=\linewidth]{ 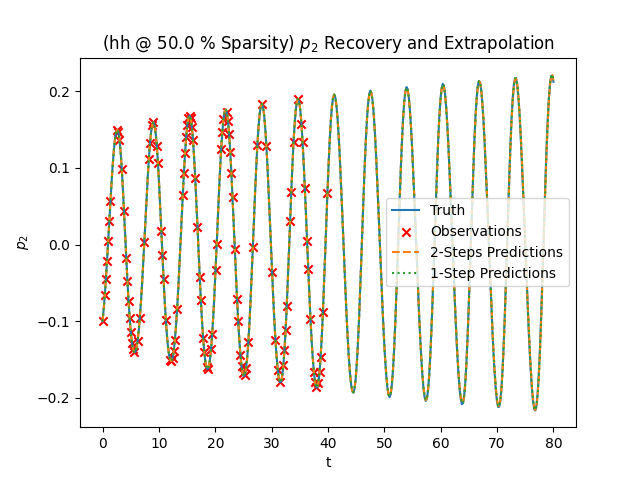}
\caption*{Poly 0.5}
\end{subfigure}
\begin{subfigure}{0.36\textwidth}
\includegraphics[width=\linewidth]{ 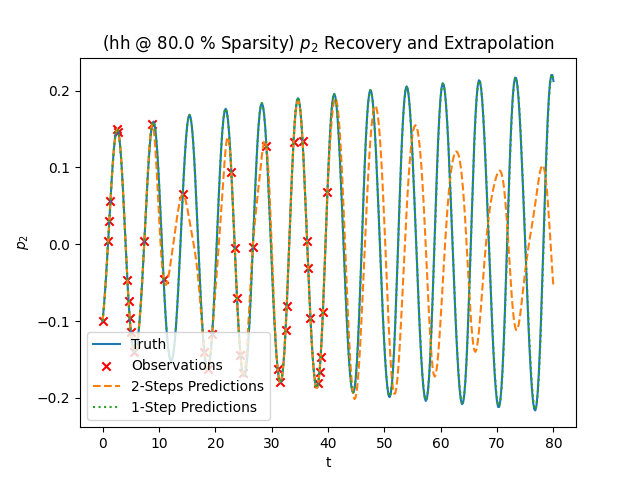}
\caption*{Poly 0.8}
\end{subfigure}
\caption{Recovery of $p_2$ in the Hénon-Heiles system.}
\end{figure}

\begin{table}[ht]
\centering
\setlength{\tabcolsep}{12pt} 
\renewcommand{\arraystretch}{1.2} 
\scriptsize

\begin{tabular}{c} 

\begin{subtable}{\textwidth}
\centering
\begin{tabular}{lcccc}
\toprule
\multirow{2}{*}{\textbf{Sparsity}} & \multicolumn{2}{c}{\textbf{Two-Steps Method}} & \multicolumn{2}{c}{\textbf{One-Step Method}} \\
\cmidrule(lr){2-3}\cmidrule(lr){4-5}
& Interpolation   & Extrapolation   & Interpolation   & Extrapolation   \\
\midrule
0.0 & \textbf{0.000 $\pm$ 0.000} & 0.214 $\pm$ 0.000 & 0.003 $\pm$ 0.000 & \textbf{0.038 $\pm$ 0.000} \\
0.5 & 0.126 $\pm$ 0.132 & 1.466 $\pm$ 2.201 & \textbf{0.019 $\pm$ 0.031} & \textbf{0.138 $\pm$ 0.086} \\
0.6 & 0.666 $\pm$ 0.504 & 8.598 $\pm$ 12.837 & \textbf{0.153 $\pm$ 0.263} & \textbf{2.477 $\pm$ 4.921} \\
0.7 & 8.368 $\pm$ 6.274 & 58.426 $\pm$ 62.178 & \textbf{4.585 $\pm$ 7.798} & \textbf{33.944 $\pm$ 65.119} \\
0.8 & 26.476 $\pm$ 15.010 & 102.456 $\pm$ 43.744 & \textbf{20.859 $\pm$ 19.867} & \textbf{96.970 $\pm$ 59.877} \\
0.9 & 75.576 $\pm$ 15.855 & 134.439 $\pm$ 15.315 & \textbf{53.391 $\pm$ 33.219} & \textbf{88.668 $\pm$ 55.595} \\
\bottomrule
\end{tabular}
\caption*{Gaussian kernel}
\end{subtable} \\[1.5ex]

\begin{subtable}{\textwidth}
\centering
\begin{tabular}{lcccc}
\toprule
\multirow{2}{*}{\textbf{Sparsity}} & \multicolumn{2}{c}{\textbf{Two-Steps Method}} & \multicolumn{2}{c}{\textbf{One-Step Method}} \\
\cmidrule(lr){2-3}\cmidrule(lr){4-5}
& Interpolation   & Extrapolation   & Interpolation   & Extrapolation   \\
\midrule
0.0 & \textbf{0.000 $\pm$ 0.000} & 0.189 $\pm$ 0.000 & 0.003 $\pm$ 0.000 & \textbf{0.061 $\pm$ 0.000} \\
0.5 & 0.126 $\pm$ 0.132 & 1.221 $\pm$ 1.085 & \textbf{0.009 $\pm$ 0.003} & \textbf{0.157 $\pm$ 0.183} \\
0.6 & 0.666 $\pm$ 0.504 & 3.099 $\pm$ 3.629 & \textbf{0.019 $\pm$ 0.013} & \textbf{0.136 $\pm$ 0.095} \\
0.7 & 8.368 $\pm$ 6.274 & 16.802 $\pm$ 17.617 & \textbf{0.037 $\pm$ 0.024} & \textbf{0.264 $\pm$ 0.212} \\
0.8 & 26.476 $\pm$ 15.010 & 76.919 $\pm$ 42.785 & \textbf{0.095 $\pm$ 0.068} & \textbf{0.283 $\pm$ 0.233} \\
0.9 & 75.576 $\pm$ 15.855 & 145.455 $\pm$ 10.461 & \textbf{0.069 $\pm$ 0.022} & \textbf{0.167 $\pm$ 0.081} \\
\bottomrule
\end{tabular}
\caption*{Separable polynomial kernel}
\end{subtable} 

\end{tabular}

\caption{Relative errors (mean $\pm$ std) for $p$ in the Mass-Spring system.}
\label{tab:ms-p}
\end{table}

\begin{table}[ht]
\centering
\setlength{\tabcolsep}{12pt} 
\renewcommand{\arraystretch}{1.2} 
\scriptsize

\begin{tabular}{c} 
\begin{subtable}{\textwidth}
\centering
\begin{tabular}{lcccc}
\toprule
\multirow{2}{*}{\textbf{Sparsity}} & \multicolumn{2}{c}{\textbf{Two-Steps Method}} & \multicolumn{2}{c}{\textbf{One-Step Method}} \\
\cmidrule(lr){2-3}\cmidrule(lr){4-5}
& Interpolation   & Extrapolation   & Interpolation   & Extrapolation   \\
\midrule
0.0 & \textbf{0.000 $\pm$ 0.000} & 0.209 $\pm$ 0.000 & 0.155 $\pm$ 0.000 & \textbf{0.314 $\pm$ 0.000} \\
0.5 & 0.334 $\pm$ 0.349 & 1.469 $\pm$ 1.315 & \textbf{0.060 $\pm$ 0.054} & \textbf{0.242 $\pm$ 0.110} \\
0.6 & 1.553 $\pm$ 1.417 & 5.443 $\pm$ 7.648 & \textbf{0.031 $\pm$ 0.011} & \textbf{0.196 $\pm$ 0.099} \\
0.7 & 16.758 $\pm$ 10.896 & 55.543 $\pm$ 58.337 & \textbf{2.077 $\pm$ 5.072} & \textbf{3.072 $\pm$ 6.465} \\
0.8 & 42.585 $\pm$ 18.315 & 137.835 $\pm$ 50.347 & \textbf{6.504 $\pm$ 8.376} & \textbf{13.412 $\pm$ 18.145} \\
0.9 & 103.994 $\pm$ 12.912 & 154.808 $\pm$ 24.992 & \textbf{48.381 $\pm$ 28.216} & \textbf{74.761 $\pm$ 34.479} \\
\bottomrule
\end{tabular}
\caption*{Gaussian kernel}
\end{subtable} \\[1.5ex] 

\begin{subtable}{\textwidth}
\centering
\begin{tabular}{lcccc}
\toprule
\multirow{2}{*}{\textbf{Sparsity}} & \multicolumn{2}{c}{\textbf{Two-Steps Method}} & \multicolumn{2}{c}{\textbf{One-Step Method}} \\
\cmidrule(lr){2-3}\cmidrule(lr){4-5}
& Interpolation   & Extrapolation   & Interpolation   & Extrapolation   \\
\midrule
0.0 & \textbf{0.000 $\pm$ 0.000} & 0.140 $\pm$ 0.000 & 0.004 $\pm$ 0.000 & \textbf{0.127 $\pm$ 0.000} \\
0.5 & 0.334 $\pm$ 0.349 & 1.549 $\pm$ 1.572 & \textbf{0.017 $\pm$ 0.006} & \textbf{0.183 $\pm$ 0.143} \\
0.6 & 1.553 $\pm$ 1.417 & 5.830 $\pm$ 8.598 & \textbf{0.042 $\pm$ 0.046} & \textbf{0.272 $\pm$ 0.231} \\
0.7 & 16.758 $\pm$ 10.896 & 71.835 $\pm$ 58.325 & \textbf{0.084 $\pm$ 0.051} & \textbf{0.345 $\pm$ 0.195} \\
0.8 & 42.585 $\pm$ 18.315 & 137.481 $\pm$ 37.661 & \textbf{0.145 $\pm$ 0.091} & \textbf{0.399 $\pm$ 0.245} \\
\bottomrule
\end{tabular}
\caption*{Separable polynomial kernel}
\end{subtable}
\end{tabular}

\caption{Relative errors (mean $\pm$ std) for $p$ in the Two-Mass-Three-Spring system.}
\label{tab:m2s3-p}
\end{table}

\begin{table}[ht]
\centering
\setlength{\tabcolsep}{12pt} 
\renewcommand{\arraystretch}{1.2} 
\scriptsize

\begin{tabular}{c} 
\begin{subtable}{\textwidth}
\centering
\begin{tabular}{lcccc}
\toprule
\multirow{2}{*}{\textbf{Sparsity}} & \multicolumn{2}{c}{\textbf{Two-Steps Method}} & \multicolumn{2}{c}{\textbf{One-Step Method}} \\
\cmidrule(lr){2-3}\cmidrule(lr){4-5}
& Interpolation   & Extrapolation   & Interpolation   & Extrapolation   \\
\midrule
0.0 & \textbf{0.000 $\pm$ 0.000} & 2.882 $\pm$ 0.000 & 0.002 $\pm$ 0.000 & \textbf{2.892 $\pm$ 0.000} \\
0.5 & 0.306 $\pm$ 0.277 & 2.764 $\pm$ 0.369 & \textbf{0.030 $\pm$ 0.024} & 3.017 $\pm$ 0.150 \\
0.6 & 0.995 $\pm$ 0.813 & 3.571 $\pm$ 1.597 & \textbf{0.054 $\pm$ 0.033} & \textbf{3.001 $\pm$ 0.328} \\
0.7 & 9.745 $\pm$ 8.983 & 17.634 $\pm$ 15.158 & \textbf{2.255 $\pm$ 5.360} & \textbf{8.057 $\pm$ 9.954} \\
0.8 & 25.900 $\pm$ 13.807 & 65.628 $\pm$ 42.015 & \textbf{6.284 $\pm$ 8.435} & \textbf{21.636 $\pm$ 27.447} \\
0.9 & 77.538 $\pm$ 19.790 & 139.084 $\pm$ 15.678 & \textbf{24.388 $\pm$ 18.010} & \textbf{43.922 $\pm$ 29.839} \\
\bottomrule
\end{tabular}
\caption*{Gaussian kernel}
\end{subtable} \\[1.5ex] 

\begin{subtable}{\textwidth}
\centering
\begin{tabular}{lcccc}
\toprule
\multirow{2}{*}{\textbf{Sparsity}} & \multicolumn{2}{c}{\textbf{Two-Steps Method}} & \multicolumn{2}{c}{\textbf{One-Step Method}} \\
\cmidrule(lr){2-3}\cmidrule(lr){4-5}
& Interpolation   & Extrapolation   & Interpolation   & Extrapolation   \\
\midrule
0.0 & \textbf{0.000 $\pm$ 0.000} & 1.149 $\pm$ 0.000 & 0.002 $\pm$ 0.000 & \textbf{1.149 $\pm$ 0.000} \\
0.5 & 0.306 $\pm$ 0.277 & 1.766 $\pm$ 0.712 & \textbf{0.025 $\pm$ 0.017} & \textbf{1.251 $\pm$ 0.104} \\
0.6 & 0.995 $\pm$ 0.813 & 3.100 $\pm$ 2.452 & \textbf{0.046 $\pm$ 0.041} & \textbf{1.197 $\pm$ 0.054} \\
0.7 & 9.745 $\pm$ 8.983 & 17.621 $\pm$ 18.179 & \textbf{0.061 $\pm$ 0.040} & \textbf{1.355 $\pm$ 0.191} \\
0.8 & 25.900 $\pm$ 13.807 & 76.579 $\pm$ 44.581 & \textbf{0.106 $\pm$ 0.063} & \textbf{1.425 $\pm$ 0.203} \\
0.9 & 77.538 $\pm$ 19.790 & 144.148 $\pm$ 13.160 & \textbf{0.146 $\pm$ 0.094} & \textbf{1.822 $\pm$ 0.503} \\
\bottomrule
\end{tabular}
\caption*{Separable polynomial kernel}
\end{subtable}
\end{tabular}

\caption{Relative errors (mean $\pm$ std) for $p$ in the Hénon-Heiles system.}
\label{tab:hh-p}
\end{table}

\begin{figure}[htbp]
\centering
\begin{subfigure}{\textwidth}
    \centering
    \begin{subfigure}{0.36\textwidth}
        \includegraphics[width=\linewidth]{ 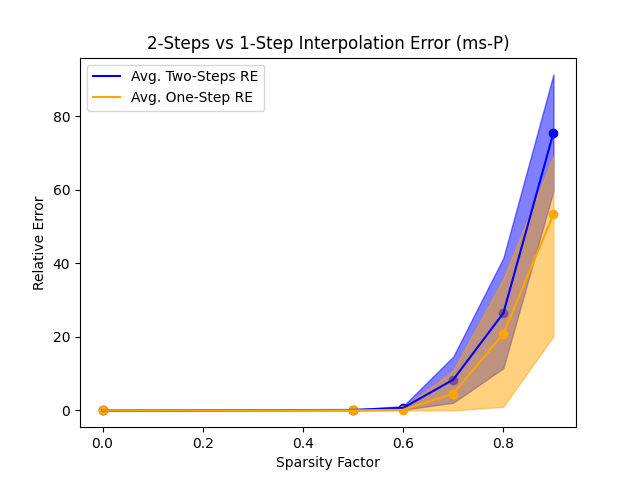}
        \caption*{$p$-error, Gaussian}
    \end{subfigure}
    \begin{subfigure}{0.36\textwidth}
        \includegraphics[width=\linewidth]{ 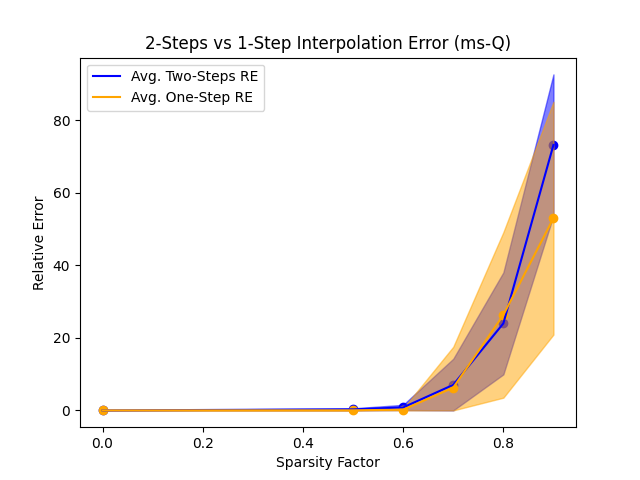}
        \caption*{$q$-error, Gaussian}
    \end{subfigure}
\end{subfigure}

\begin{subfigure}{\textwidth}
    \centering
    \begin{subfigure}{0.36\textwidth}
        \includegraphics[width=\linewidth]{ 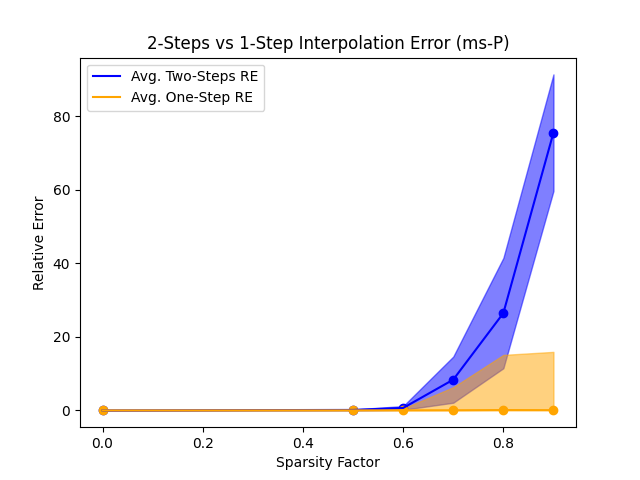}
        \caption*{$p$-error, Separable Polynomial}
    \end{subfigure}
    \begin{subfigure}{0.36\textwidth}
        \includegraphics[width=\linewidth]{ 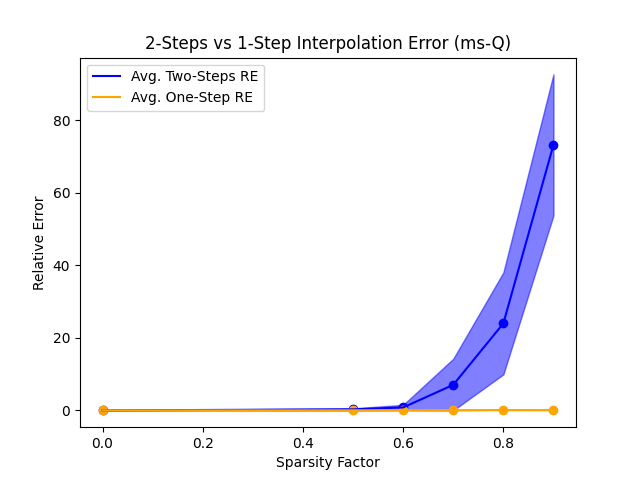}
        \caption*{$q$-error, Separable Polynomial}
    \end{subfigure}
\end{subfigure}

\caption{Interpolation relative errors for the Mass-Spring System.}
\label{fig:error_ms_all}
\end{figure}

\begin{figure}[htbp]
\centering

\begin{subfigure}{\textwidth}
    \centering
    \begin{subfigure}{0.36\textwidth}
        \includegraphics[width=\linewidth]{ 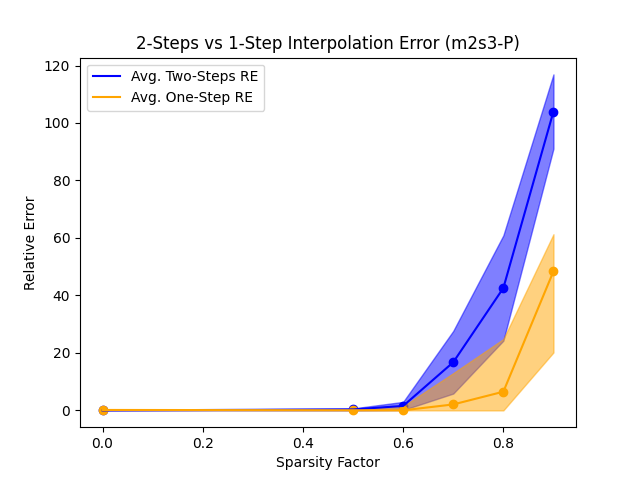}
        \caption*{$p$-error, Gaussian}
    \end{subfigure}
    \begin{subfigure}{0.36\textwidth}
        \includegraphics[width=\linewidth]{ 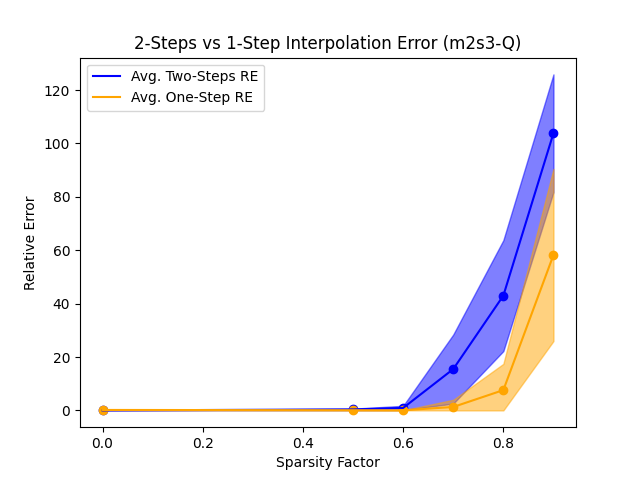}
        \caption*{$q$-error, Gaussian}
    \end{subfigure}
\end{subfigure}

\begin{subfigure}{\textwidth}
    \centering
    \begin{subfigure}{0.36\textwidth}
        \includegraphics[width=\linewidth]{ 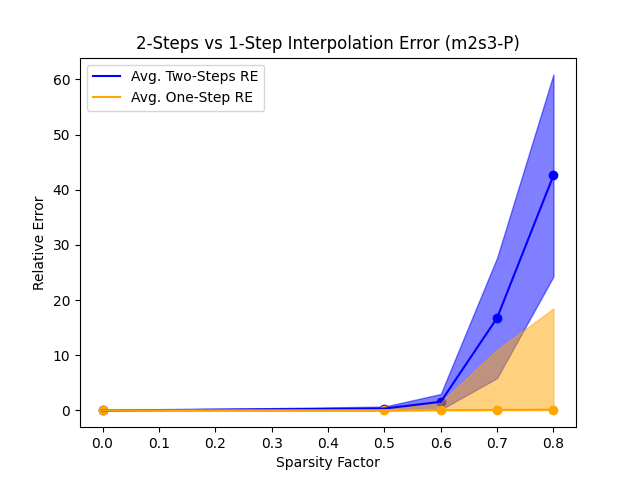}
        \caption*{$p$-error, Separable Polynomial}
    \end{subfigure}
    \begin{subfigure}{0.36\textwidth}
        \includegraphics[width=\linewidth]{ 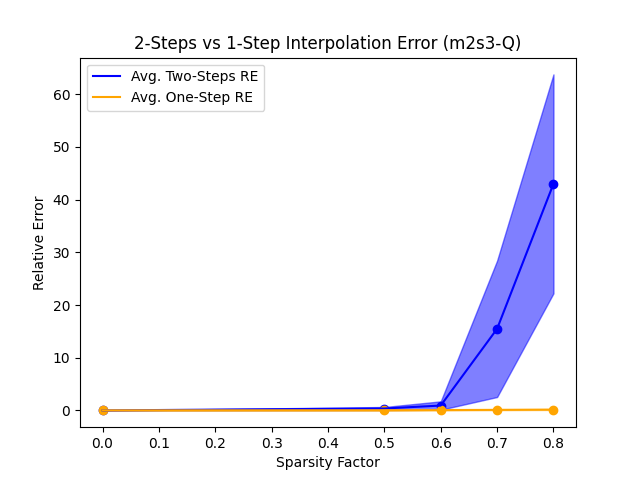}
        \caption*{$q$-error, Separable Polynomial}
    \end{subfigure}
\end{subfigure}

\caption{Interpolation relative errors for the Two-Mass-Three-Spring System.}
\label{fig:error_m2s3_all}
\end{figure}

\begin{figure}[htbp!]
\centering

\begin{subfigure}{\textwidth}
    \centering
    \begin{subfigure}{0.36\textwidth}
        \includegraphics[width=\linewidth]{ 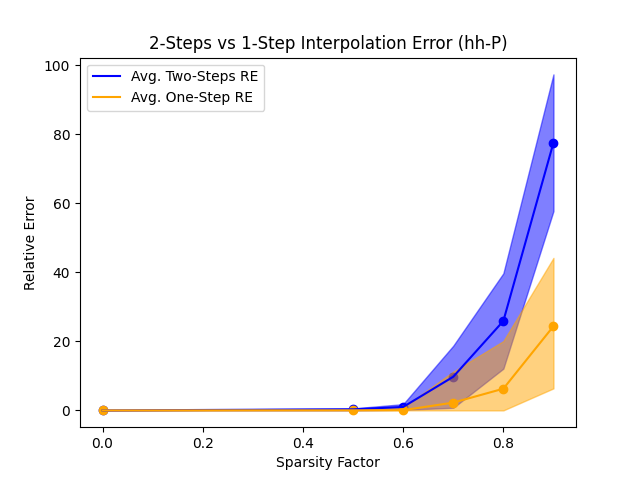}
        \caption*{$p$-error, Gaussian}
    \end{subfigure}
    \begin{subfigure}{0.36\textwidth}
        \includegraphics[width=\linewidth]{ 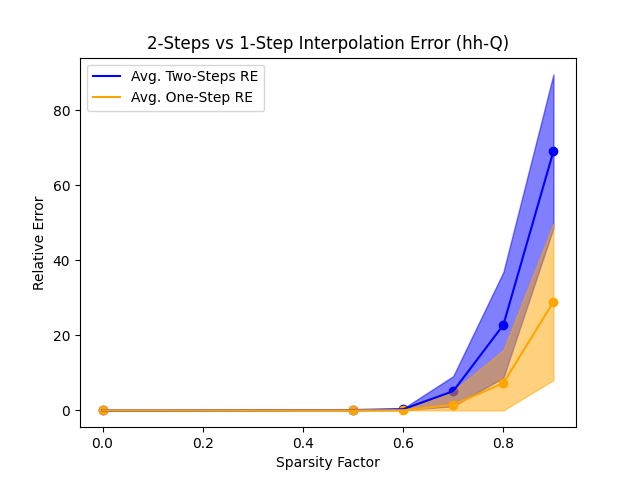}
        \caption*{$q$-error, Gaussian}
    \end{subfigure}
\end{subfigure}

\begin{subfigure}{\textwidth}
    \centering
    \begin{subfigure}{0.36\textwidth}
        \includegraphics[width=\linewidth]{ 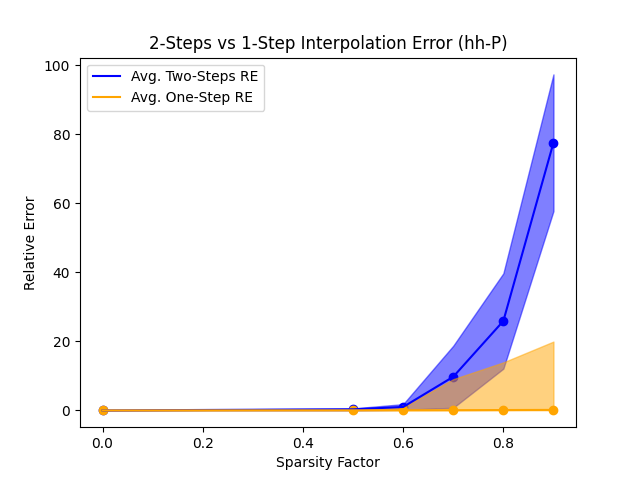}
        \caption*{$p$-error, Separable Polynomial}
    \end{subfigure}
    \begin{subfigure}{0.36\textwidth}
        \includegraphics[width=\linewidth]{ 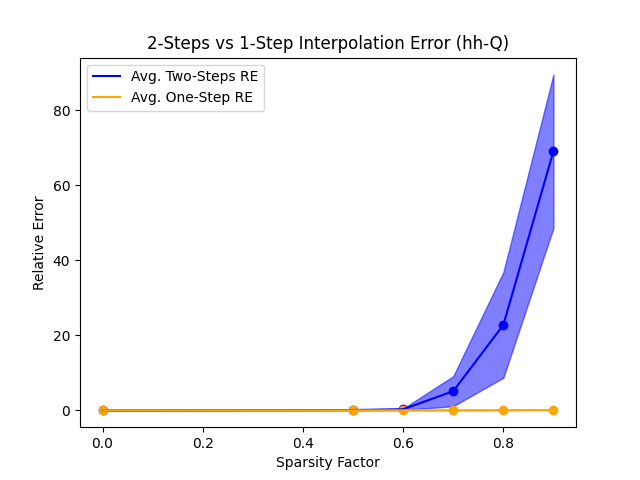}
        \caption*{$q$-error, Separable Polynomial}
    \end{subfigure}
\end{subfigure}

\caption{Interpolation relative errors for the Hénon-Heiles System.}
\label{fig:error_hh_all}
\end{figure}

\begin{figure}[htbp]
\centering

\begin{subfigure}{\textwidth}
    \centering
    \begin{subfigure}{0.36\textwidth}
        \includegraphics[width=\linewidth]{ 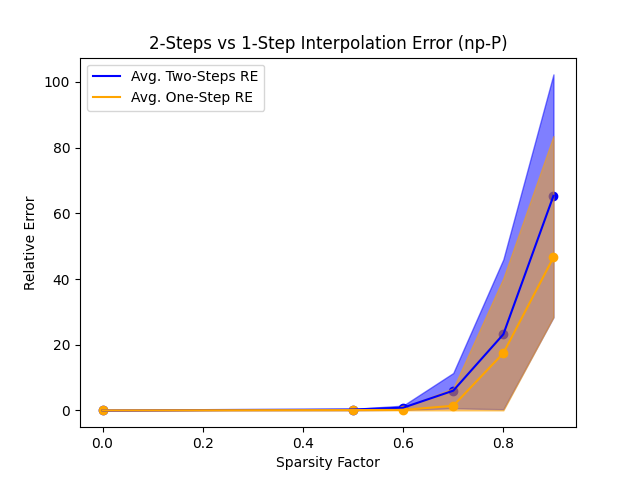}
        \caption*{$p$-error, Gaussian}
    \end{subfigure}
    \begin{subfigure}{0.36\textwidth}
        \includegraphics[width=\linewidth]{ 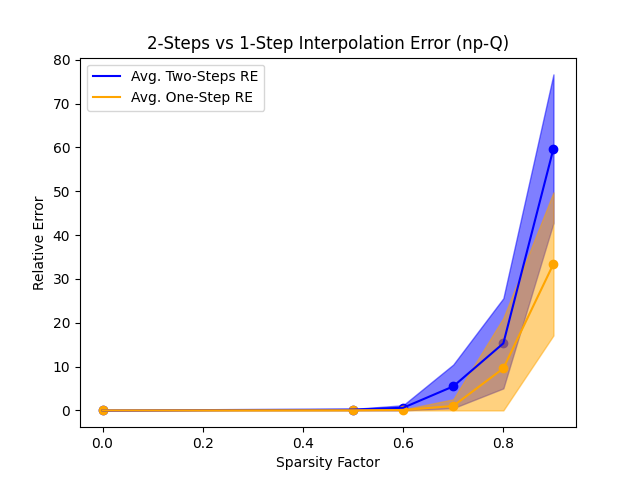}
        \caption*{$q$-error, Gaussian}
    \end{subfigure}
\end{subfigure}

\begin{subfigure}{\textwidth}
    \centering
    \begin{subfigure}{0.36\textwidth}
        \includegraphics[width=\linewidth]{ 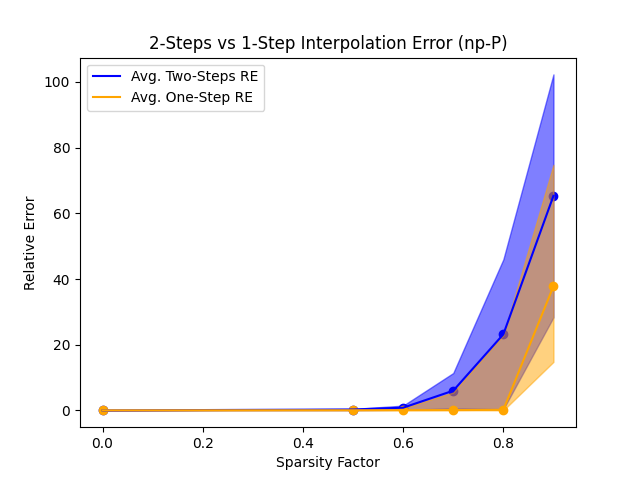}
        \caption*{$p$-error, Separable Polynomial}
    \end{subfigure}
    \begin{subfigure}{0.36\textwidth}
        \includegraphics[width=\linewidth]{ 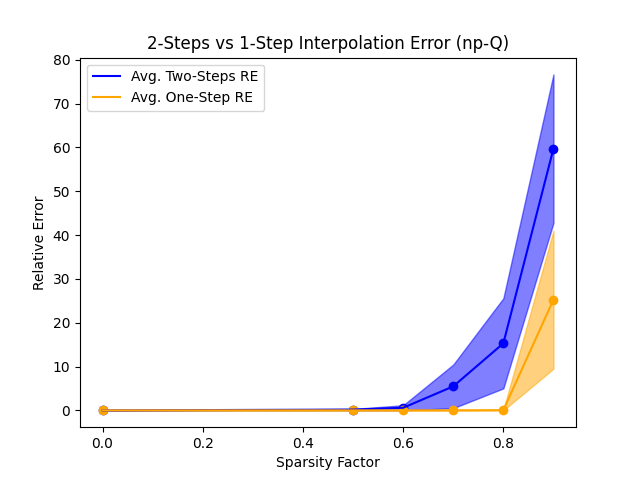}
        \caption*{$q$-error, Separable Polynomial}
    \end{subfigure}
\end{subfigure}

\caption{Interpolation relative errors for the Nonlinear Pendulum System.}
\label{fig:error_np_all}
\end{figure}

\end{document}